\documentclass{amsart}
\usepackage{stmaryrd}
\usepackage{amscd}
\usepackage{amsmath}
\usepackage{amsfonts}
\usepackage{mathrsfs}
\usepackage{amssymb}
\usepackage{hyperref}
\usepackage[all]{xy}

\def\bbe{\mathbf{e}}

\def\bbk{\mathbf{k}}

\def\bbr{\mathbf{r}}

\def\bbt{\mathbf{t}}

\def\bbv{\mathbf{v}}

\def\bbF{\mathbf{F}}

\def\bbM{\mathbf{M}}

\def\bbR{\mathbf{R}}
\def\bbS{\mathbf{S}}

\def\AAA{\mathbb{A}}

\def\NN{\mathbb{N}}

\def\QQ{\mathbb{Q}}
\def\RR{\mathbb{R}}

\def\ZZ{\mathbb{Z}}

\def\gothm{\mathfrak{m}}

\def\gotho{\mathfrak{o}}
\def\gothp{\mathfrak{p}}

\def\gothR{\mathfrak{R}}

\def\calD{\mathcal{D}}

\def\calF{\mathcal{F}}

\def\calI{\mathcal{I}}

\def\calO{\mathcal{O}}

\def\calR{\mathcal{R}}

\def\sfF{\mathsf{F}}
\def\rmH{\mathrm{H}}

\newcommand{\alg}{{\mathrm{alg}}}

\newcommand{\id}{{{\rm id}}}

\renewcommand{\log}{\mathrm{log}}

\newcommand{\Car}{\mathrm{Car}}
\newcommand{\ZCar}{\mathrm{ZCar}}
\newcommand{\length}{\mathrm{length}}
\newcommand{\Supp}{{\mathrm{Supp}}}
\newcommand{\End}{{\mathrm{End}}}
\newcommand{\Gal}{{\mathrm{Gal}}}  
 \newcommand{\rank}{{\mathrm{rank}}}

\newcommand{\Frac}{{\mathrm{Frac}}}

\newcommand{\Irr}{{\mathrm{Irr}}}

\newcommand{\Spec}{{\mathrm{Spec}}}  \newcommand{\Sym}{{\mathrm{Sym}}}

\newcommand{\bs}{\backslash}

\newcommand{\Ref}{\mathrm{Ref}}
\newcommand{\fil}{\mathrm{fil}}
\newcommand{\gr}{\mathrm{gr}}
\newcommand{\dR}{\mathrm{dR}}

\newcommand{\bdd}{\mathrm{bd}}
\newcommand{\Log}{\mathrm{Log}}
\newcommand{\GL}{\mathrm{GL}}
\newcommand{\Reg}{\mathrm{Reg}}

\swapnumbers 
\theoremstyle{plain}
\newtheorem{theorem}[subsubsection]{Theorem}
\newtheorem{lemma}[subsubsection]{Lemma}

\newtheorem{corollary}[subsubsection]{Corollary}
\newtheorem{prop}[subsubsection]{Proposition}

\theoremstyle{definition}
\newtheorem{definition}[subsubsection]{Definition}
\newtheorem{caution}[subsubsection]{Caution}
\newtheorem{construction}[subsubsection]{Construction}

\newtheorem{remark}[subsubsection]{Remark}

\newtheorem{notation}[subsubsection]{Notation}

\begin{document}

\title[Cleanliness and log-characteristic cycles]{Cleanliness and log-characteristic cycles for vector bundles with flat connections}

\author{Liang Xiao}

\begin{abstract}
Let $X$ be a proper smooth algebraic variety over a field $k$ of characteristic zero and let $D$ be a divisor with  simple normal crossings.  Let $M$ be a vector bundle over $X-D$ equipped with a flat connection with possible irregular singularities  along $D$.  We define a cleanliness condition which roughly says that the singularities of the connection are controlled by the singularities at the generic points of $D$.  When this condition is satisfied, we compute explicitly the associated log-characteristic cycle, and relate it to the so-called refined irregularities.  As a corollary of a log-variant of Kashiwara-Dubson formula, we obtain the Euler characteristic of the de Rham cohomology of the vector bundle, under a mild technical hypothesis on $M$.
\end{abstract}

\maketitle
\tableofcontents

\section{Introduction}

Let $X$ be a proper  smooth algebraic variety of dimension $n$ over an algebraically closed field $k$ of charactersitic zero, and let $D$ be a divisor with  simple normal crossings.  We put $U = X -D$.
Let $M$ be a vector bundle over $U$ with a flat connection having possibly irregular singularities along $D$.  We may define the \emph{Euler characteristic} of the \emph{de Rham cohomology} of $M$ to be 
\[
\chi_\dR(M) = \sum_{i=0}^{2n} (-1)^i
\dim_k \mathrm H^i(U, M \otimes \Omega_U^\bullet).
\]

When $X$ is a smooth projective curve of genus $g(X)$ and $D$ is a finite subset of closed points, Deligne and Gabber \cite[Theorem~2.9.9]{katz} proved a formula to compute the Euler characteristic:
\[
\chi_\dR(M)=\rank(M)\chi(U)-\sum_{x\in D}\Irr_x(M),
\]
where $\chi(U)=2-2g(X)-\#D$ is the usual Euler characteristic of $U$ and $\Irr_x(M)$ is the irregularity of $M$
at $x$ (see Subsection~\ref{SS:irregularities} for a definition).

This formula says that the Euler characteristic of $M$ is given by certain geometric information (the first term) corrected by appropriate ramification information (the second term).  It is a natural question to ask for higher dimensional analogues of this.

The first step was taken by Kato \cite{kato-d-mod}, who proved a higher dimensional analogue of the above formula for line bundles with flat connections, under some cleanliness condition.  For example, when $X$ is two-dimensional and the irregularity of $M$ along each irreducible component  $D_j$ of $D$ is $r_j$, then the Euler characteristic of $M$ is
\begin{equation}
\label{E:euler-dim2}
\chi_\dR(M) = \rank(M) \chi(U) - \sum_j r_j \chi(D_j^\circ) + \sum_{j,j'} r_{j}r_{j'} (D_j \cdot D_{j'})
\end{equation}
where $D_j^\circ = D_j \bs (\cup_{j' \neq j} D_{j'})$, $(D_j \cdot D_{j'})$ is the intersection number, and $\chi(\cdot)$ is the usual Euler characteristic.

\vspace{5pt}
The cleanliness condition Kato assumed in his formula is \emph{necessary}, and is given in a very explicit form.   But when working with vector bundle of higher rank, it is very subtle  to rigorously define the cleanliness condition.  We start by explaining how to interpret the meaning of the cleanliness condition.  
A caveat is that these viewpoints lead to \emph{inequivalent} definitions of cleanliness; we list them from the strongest to the weakest. (Proofs maybe found in Section~2: Theorem~\ref{T:kedlaya-criterion} for (1)$\Rightarrow$(2), Theorem~\ref{T:numerical=>clean} for (2)$\Rightarrow$(3), and Proposition~\ref{P:local-calculation} for (3)$\Rightarrow$(4).)

(1)  In the formal neighborhood of each closed point of $X$ and up to a tamely ramified extension, $M$ has a ``good decomposition" in the sense of \cite{kedlaya-sabbah1, kedlaya-sabbah2}; vaguely speaking, it can be written as a direct sum of some differential modules which are tensor products of regular differential modules with some ``simple and explicit" rank one (irregular) differential modules.  See Definition~\ref{D:good-decomposition} for the precise definition.

(2) In the formal neighborhood of each closed point of $X$, the function(s) interpolating the irregularities along the exceptional divisors of toroidal blowups are linear, if appropriately normalized.  This is a weak version of Kedlaya's ``numerical cleanliness" \cite[Theorem~4.4.2]{kedlaya-sabbah1}.

(3) The irregularity of the connection  at each closed point of $D$ is ``controlled" by the irregularity at the generic points of $D$.  In particular, there is no expected contribution to the Euler characteristic from codimension $\geq 2$ strata.  This is the cleanliness condition we will work with throughout this paper.  One expects that this definition (as oppose to the previous two stronger statements) generalizes to analogous positive characteristic situation.  (See \cite{abbes-saito} for the description at the generic points in the analogous $\ell$-adic setting.)

(4) In terms of the log-characteristic variety $\Car(M)$ of $M$, the cleanliness condition should imply that $\Car(M)$ consists of only zero sections of the log-cotangent bundle and some line bundles over $D$; in particular, the fiber of $\Car(M)$ over each closed point of $X$ is at most one-dimensional.

\vspace{5pt}
The aim of this paper is to generalize Kato's result to vector bundles of arbitrary rank, under a very mind hypothesis essentially saying that all $r_j$ above are positive.  (See Theorem~\ref{T:main-theorem} and Corollary~\ref{C:EP-formula} for the precise statement.)  Roughly speaking, in computing Euler characteristic, we may ``pretend" that the differential module is a direct sum of rank 1 modules with specified irregularity properties.  
(But note that the corresponding $r_j$'s in \eqref{E:euler-dim2} may be rational numbers as opposed to integers.  So, interestingly, it is a priori not clear why the total Euler characteristic is an integer by just looking at the formula.)

The basic strategy is to compute the log-characteristic cycle of $M$ explicitly and then try to obtain the Euler characteristic from a log-variant of Kashiwara-Dubson formula.  As pointed out in (4), the log-characteristic cycle will be the sum of the zero section of the log-cotangent bundle (with multiplicity $\rank(M)$) and some line bundles over $D$. (They are usually \emph{not} the conormal bundles of the divisors because of the log-structure. See the discussion later in (a).)  The multiplicities of these line bundles are determined by the irregularities, and the positions of the line bundles in the log-cotangent bundle is determined by so-called refined irregularities.  When computing the Euler characteristic, the positions of these line bundles no longer matter and hence, the refined irregularities do not appear in the expression of the Euler characteristic formula.

There are several essential difficulties we need to overcome.

(a) The theory of logarithmic $\calD$-modules is \emph{very different} from the classical theory of nonlogarithmic $\calD$-modules.  For one thing, if we view the vector bundle $M$ (with a flat connection) over $U$ as a quasi-coherent sheaf $j_*M$ over $X$ ($j: U\to X$ being the natural immersion), it is coherent as a $\calD_X$-module but  it may \emph{not} be $\calD_X^\log$-coherent.  In some sense, the trouble comes from the piece with regular singularities; see Caution~\ref{C:not-coherent}. Therefore, we cannot apply the classical definition of the characteristic cycle directly. For another, we do not have the log-holonomicity (since we do not even have the coherence at the first place) and Bernstein inequality (see Caution~\ref{C:no-Berntein-inequality}).
To get around the first issue, we are forced to brutally define the log-characteristic cycle to be the ``limit" over all finitely generated $\calD_X^\log$-submodules.  
It is then not clear apriori that such a definition would give anything reasonable to work with.
But we prove, using a trick of Bernstein, that holonomicity implies log-holonomicity (in the sense that the dimension of the log-characteristic cycle is less than or equal to $\dim X$; note that, without Bernstein inequality, we do not expect an equality of dimension here). 
What made the situation even worse is that, unfortunately, we cannot establish the log-variant of Kashiwara-Dubson formula using the new log-characteristic cycle for a rather technical reason.  However, when  $M$ is clean and all irregularities are positive, such a formula holds; this explains why we had the assumption earlier.

(b)  Kato \cite{kato-d-mod} was working with line bundles, where one can choose a generator of the line bundle {\it Zariski locally} and everything may be explicitly written down.  When we work with higher rank vector bundles in this paper, we already need the full power of the theory of differential modules, developed by Kedlaya and the author \cite{kedlaya-xiao, kedlaya-sabbah1, xiao-refined}, to be able to give a description of the differential module over the \emph{formal completion at each closed point}; thus we may not glue the description at various closed points.   To be able to patch up the result globally, we will start with a conjectural definition of the log-characteristic variety using some global data, and check that this definition agrees with the actual log-characteristic cycle at each closed point.

(c)  The computation of the log-characteristic cycle is much more delicate than Kato's original computation.  
In our case, we need to first compute the underlying variety of the (log-)characteristic cycle; for this, we pass to the formal neighborhood of a closed point and check that, over this point (not the formal neighborhood), the log-characteristic variety agrees with the conjectural one obtained from refined irregularities.  This proves the equality on the characteristic varieties.  To get the multiplicities, we pass to the generic points of $D$ and do a more careful study in this case.

\vspace{5pt}
The cleanliness condition is a very strong condition to impose on $X$.
Kedlaya \cite{kedlaya-sabbah1, kedlaya-sabbah2} and independently Mochizuki \cite{mochizuki} proved that there exists a proper birational morphism $f: X' \to X$ which is an isomorphism when restricted over $U$, such that $f^{-1}(D)$ is a divisor with simple normal crossings and that $f^*M$ has a good formal model on $X'$ (i.e., satisfying condition (1)).  But one does not have control over how ``ramify" $f$ can be over $D$.  For example, it is still open that this morphism $f$ may be taken functorially with respect to smooth morphisms.  From another point of view, this is a type of (weak) resolution of singularities for $M$, that is to find a ``good compactification" of $U$ (where $M$ lives on) so that all the ramification information of $M$ is ``exposed" in codimension 1 strata.

\vspace{5pt}
We  mention that there are similar stories over a field of characteristic $p>0$, considering lisse $\ell$-adic sheaves or overconvergent $F$-isocrystals, in particular in some works of Abbes and Saito \cite{abbes-saito, saito}.  We point out that our approach is fundamentally different from their approach in that 
\begin{itemize}
\item our definition of the log-characteristic cycle is intrinsic to the vector bundle, and we prove that, under the cleanliness condition, the intrinsically defined log-characteristic cycle agrees with the ramification information read off at the generic points of the divisor;
\item In contrast, \cite{abbes-saito, saito} reads off the ramification information at the generic points of the divisor and use it to \emph{define} the log-characteristic cycle.
\end{itemize}
It is also worth pointing out that the conjectural formula \cite[Conjecture~1.13]{abbes-saito} depends mostly on the multiplicities of each irreducible component of the log-characteristic cycle defined there, but \emph{not} on the shape of the log-characteristic cycle (which is defined by ramification information at generic points.)  Our approach then gives a new interpretation of their definition in terms of more traditional theory of characteristic cycles, in the analogous situation for vector bundles with connections.

\subsubsection*{Structure of the paper}
In Section 1, we study the theory of logarithmic $\mathcal{D}$-modules.  We give the definition of log-holonomicity for not necessarily finitely generated $\calD_X^\log$-modules and define the  log-characteristic cycles for vector bundles with flat connections.
In Section 2, we first review the theory of nonarchimedean differential modules and make some generalizations for our need.  Then we define various cleanliness conditions and discuss their relations. In Section 3, we state and prove the main theorem, which computes the log-characteristic cycles of vector bundles with flat connections under the cleanliness condition.

\subsubsection*{Acknowledgement}

I thank Kiran Kedlaya for helpful discussions and sharing ideas.
I am grateful for Kazuya Kato's groundbreaking paper \cite{kato-d-mod}, on which the calculation of this paper is based on.
I thank Matthew Morrow for suggesting the lecture notes on $\calD$-modules by Braverman and Chmutova, which inspired the proof of Theorem~\ref{T:log-holonomicity}.
Thanks also go to Adriano Marmora and J\'er\^ome Poineau for organizing the wonderful conference on \textit{Berkovich space and $p$-adic differential equations}; it stimulates the current project. 
I also thank the anonymous referee for careful reading of the manuscript and many helpful suggestions which largely improve the presentation.
I thank Daniel Caro, Dennis Gaitsgory, Bin Li, Chenyang Xu, Zhiwei Yun, and Weizhe Zheng for inspirations and interesting discussions. I thank Tsinghua Mathematics Science Center and Morningside Center of Mathematics for their hospitality when I visited, during which time I started to work on this paper.

\subsection*{Notation and convention}\

Throughout this paper, we use $k$ to denote a field of characteristic zero.

For a noetherian integral (formal) scheme $X$ over $k$, let $\calO_X$ denote the structure sheaf on $X$ and let $k(X)$ denote the field of rational functions on $X$.  
We use $|X|$ to denote the set of closed points.
If $X$ is affine, we also use $\calO_X$ to denote the ring of global sections of the structure sheaf;
for a closed point  $x\in |X|$, we use $\gothm_x$ to denote the maximal ideal of $\calO_X$ corresponding to $x$.
If $X$ is affine and $u \in \calO_X$, we use $V(u)$ to denote the closed (formal) subscheme associated to $\calO_X / (u)$.  

A \emph{smooth pair} $(X,D)$ consists of an irreducible smooth variety $X$ over $k$ and a divisor $D$ with  simple normal crossings. (Our convention of simple normal crossings require all irreducible components of $D$ to be smooth.)  A morphism $f: (X', D') \to (X, D)$ between two smooth pairs is a morphism $f: X' \to X$ of varieties such that $f(X'-D') \subseteq X-D$.  Given a smooth pair $(X, D)$, we equip it with the natural log-structure.  

Unless otherwise stated, all differentials and derivations are continuous, and are relative to $k$.  

We will frequently say vector bundles to mean locally free sheaves of finite rank.
For a locally free coherent sheaf $\calF$ over a scheme $X$, we let $\Sym^\bullet_{\calO_X} \calF^\vee$ denote the sheaf of symmetric algebra over $\calF^\vee$; the associated scheme is the physical vector bundle associated to $\calF$.
A connection $\nabla$ on a vector bundle $M$ over a smooth scheme $U$ is called \emph{integrable} if the composition of $\nabla$ with the induced morphism $\nabla^{(1)}: M \otimes \Omega^1_U \to M \otimes \Omega^2_U$ is zero.  Here we choose to use the notation ``integrable" over ``flat" because we want to avoid possible confusion with the algebro-geometric meaning of flatness.

In this paper, we only implicitly use nonlog-characteristic cycles/varieties in Subsection~\ref{S:log-holonomicity}.  Aside from this, we will exclusively discuss log-characteristic cycles/varieties.  We will try to emphasize this as often as possible.  But we sometimes give in for simpler notation, e.g. $\Car(M)$ and $\ZCar(M)$.

\section{Log-characteristic cycles}

\subsection{General framework of characteristic cycles}
\label{S:framework-CZar}
We discuss some slightly general framework of filtered rings and their characteristic cycles.  The results here are elementary and are probably in the literature (e.g. \cite{laumon}), but it might be hard to extract the exact statements we need.  For completeness, we reproduce them here for the convenience of readers.  (Results from this subsection apply equally well to the case when $k$ is of finite characteristic.)

\subsubsection{Filtered rings}
\label{SS:filtered-ring}

Let $(D, \fil_\bullet D)$ be an increasingly filtered possibly non-commutative $k$-algebra. We assume the following:

(i) $\fil_\alpha D = 0$ if $\alpha<0$, $D = \cup_{\alpha \geq0} \fil_\alpha D$;

(ii) $\gr_\bullet D$ is a \emph{commutative noetherian} $k$-algebra (which implies that $D$ itself is noetherian).

A homomorphism $f:(D', \fil_\bullet D') \to (D, \fil_\bullet D)$ between two such filtered $k$-algebras is called \emph{strict} if the filtration on $D'$ is exactly the filtration induced by $f$.  A homomorphism $f:(D', \fil_\bullet D') \to (D, \fil_\bullet D)$ induces a homomorphism $\gr_\bullet: \gr_\bullet D' \to \gr_\bullet D$.

We often write $D_0$ for $\fil_0D$; it is a commutative noetherian $k$-algebra.

A standard example to keep in mind is the ring of differential operators $\calD_X$ defined later.

\begin{definition}
\label{D:good-filtration}
For a $D$-module $M$, a filtration $\fil_\bullet M$ on $M$ is called 
 \emph{admissible} if

(i) $\fil_\alpha M =0$ when $\alpha \ll 0$ and $M = \cup_\alpha M_\alpha$,

(ii) each $M_\alpha$ is $D_0$-coherent, and

(iii) $\fil_\alpha D \cdot \fil_\beta M \subseteq \fil_{\alpha+\beta} M$ for each $\alpha, \beta$.

\noindent
We call it \emph{good} if it is admissible and it satisfies:

(iv) $\gr_\bullet M$ is a finitely generated $\gr_\bullet D$-module.

\end{definition}

\begin{definition}
Let $M$ be a finitely generated $D$-module.  It is well-known that $M$ has a good filtration $\fil_\bullet M$.  Define the \emph{characteristic variety} $\Car(M) = \Car_D(M)$ to be the support of $\gr_\bullet M$ as a $\gr_\bullet D$-module; it is a closed subscheme of $\Spec (\gr_\bullet D)$. We define the \emph{characteristic cycle} to be
\[
\ZCar(M)=\ZCar_D(M) = \sum \length(\gr_\bullet(M)_\eta) \overline{\{\eta\}},
\]
where the sum is taken over all generic points $\eta$ of irreducible components of $\Car(M)$.
These do not depend on the choice of good filtrations. (See for example, \cite[Lemma~D.3.1]{HTT})
\end{definition}

\begin{remark}
\label{R:log-char-for-sub}
If $M' \subseteq M$ is a sub-$D$-module, we have $\Car(M') \subseteq \Car(M)$.  Note that this does not imply that the generic points of $\Car(M')$ is a subset of those of $\Car(M)$.  If moreover $\Car(M') = \Car(M)$, the multiplicity of $\ZCar(M)$ at every generic point of $\Car(M)$ is greater than or equal to the corresponding multiplicity of $\ZCar(M')$.
\end{remark}

\begin{lemma}
\label{L:charcycles-base-change}
Let $f: (D, \fil_\bullet D) \to (D', \fil_\bullet D')$ be a strict morphism as in Definition~\ref{D:good-filtration} such that $D'_0$ is flat over $D_0$ and $\gr_\bullet f$ induces an isomorphism $\gr_\bullet D' \simeq D'_0 \otimes_{D_0} \gr_\bullet D$. We write 
\[
g: \Spec (\gr_\bullet D') =\Spec(D'_0) \times_{\Spec (D_0)}\Spec(\gr_\bullet D)   \to \Spec (\gr_\bullet D)
\]
for the natural projection.   For a finitely generated $D$-module $M$, we denote $M' = D' \otimes_D M$.  Then we have $\ZCar(M') = g^*( \ZCar(M))$.
\end{lemma}
\begin{proof}
By induction on $\alpha$ and using the flatness of $D'_0$ over $D_0$, the homomorphism $f$ induces natural isomorphisms $D'_0 \otimes_{D_0}\fil_\alpha D \stackrel \sim \to \fil_\alpha D'$ as \emph{left $D'_0$-modules} for each $\alpha \in \ZZ$. This implies that, we have $D'_0 \otimes_{D_0} D \simeq D'$ as left $D'_0$- and right $D$-modules.  As a consequence, we have a natural isomorphism $D'_0 \otimes_{D_0} M \stackrel \simeq \to M'$ as left $D'_0$-modules; we identify them.

Now, we choose a good filtration $\fil_\bullet M$ on $M$ with respect to $D$.  We define a filtration on $M'$ (under the aforementioned identification) by 
$
\fil_\alpha M' = D'_0 \otimes_{D_0} \fil_\alpha M
$ for all $\alpha \in \ZZ$.
This filtration obviously satisfies conditions (i)--(iii) of Definition~\ref{D:good-filtration}.
Moreover, we have $\gr_\bullet M' \simeq D'_0 \otimes_{D_0} \gr_\bullet M$ as a graded left $D'_0$-module.  Hence, to check the conditions (iv) of Definition~\ref{D:good-filtration}, it suffices to prove that the action of $\gr_\bullet D' \simeq D'_0 \otimes_{D_0} \gr_\bullet D$ on $\gr_\bullet M'$ is the one induced by the action of $\gr_\bullet D$ on $\gr_\bullet M$.  Indeed, this follows from the fact that for any $\alpha \in \ZZ$, any $x \in \fil_\alpha D$ and any $a \in D'_0$, we have $xa - ax \in \fil_{\alpha-1} D' \simeq D'_0 \otimes_{D_0} \fil_{\alpha-1} D$.  This shows that $\fil_\bullet M'$ is a good filtration.

The fact that $\gr_\bullet M' \simeq D'_0 \otimes _{D_0} \gr_\bullet M$ implies that $\Car(M') = g^{-1}(\Car(M))$.
Since $g$ is flat, it satisfies the Going-down property (\cite[Lemma~10.11]{eisenbud}).  So, if an irreducible component of $\Car(M)$ in $\Spec( \gr_\bullet D)$ intersects with the image of $g$, its generic point must be in the image too.  Hence, we have $\ZCar(M') = g^*(\ZCar(M))$.
\end{proof}

We record the following homological algebra result for future reference.

\begin{prop}
\label{P:bernstein-trick}
Let $(D, \fil_\bullet D)$ be as in Definition~\ref{D:good-filtration} and assume that $\gr_\bullet D$ is regular of pure dimension $n$ over $k$. Let $M$ be any finitely generated
left $D$-module. Then we have 
\[
\min\{j\,|\, \mathrm{Ext}_D^j(M,D) \neq 0\} = n - \dim(\Car(M)).
\]
\end{prop}
\begin{proof}
This is classical.  See for example \cite[Theorem V.2.2]{Alg-D}.
\end{proof}

\begin{remark}
\label{R:bernstein-trick}
A standard trick using the above proposition is due to Bernstein: one takes two different filtrations $\fil_\bullet$ and $\fil'_\bullet$ on $D$ such that $\gr_\bullet D$ and $\gr'_\bullet D$ are both regular of pure dimension $n$ over $k$.  For a finitely generated left $D$-module $M$, we may compute two characteristic cycles $\Car(M)$ and $\Car'(M)$ separately with respect to each of the two filtrations on $D$.  Then $\dim(\Car(M)) = \dim(\Car'(M))$ because they are both equal to a number which is independent of the filtrations.
\end{remark}

\subsection{$\calD$-modules and logarithmic variants}
In this subsection, we recall the definition of characteristic cycles for algebraic $\calD$-modules and their logarithmic variants.  We warn the readers that the theory for log-characteristic cycles is \emph{very different} from the classical theory of nonlog-characteristic cycles.  For one thing, the  logarithmic $\calD$-modules associated to a vector bundle with an integrable connection are \emph{not} finitely generated as logarithmic $\calD$-modules (see Caution~\ref{C:not-coherent}); for another, even without the first issue, the log-holonomicity is a quite delicate concept (see also Caution~\ref{C:no-Berntein-inequality}).

\subsubsection{Local setup}
\label{SS:local-setup}
We consider the following three local situations at the same time.  Let $m\leq n$ be two natural numbers.

(a) (Geometric) Let $X$ be a smooth affine variety with local parameters $x_1, \dots, x_n$, that is an \'etale morphism $p: X \to \AAA^n$, where $x_1, \dots, x_n$ are standard coordinates of $\AAA^n$.  Set $D = p^{-1}(V(x_1\cdots x_m))$.  Then $(X, D)$ is a smooth pair.

(b) (Formal) We take $X = \Spec (R_{n,0})$ with $R_{n,0} = k\llbracket x_1, \dots, x_n\rrbracket$, and set $D = V(x_1\cdots x_m)$.

(c) (CDVF)  We take $X = \Spec (k(D)\llbracket x_1 \rrbracket)$, where $k(D)$ is a finite extension of $k(x_2, \dots, x_n)$ or the fraction field of $k\llbracket x_2, \dots, x_n\rrbracket$.  We set $D = V(x_1)$ and $m=1$ in this case.

In either case, we put $U = X - D$ and use $j : U \hookrightarrow X$ to denote the natural morphism.  We have $\Omega_X^1 \simeq \bigoplus_{i=1}^n \calO_X dx_i$ (note that all differentials are assumed to be continuous); let $(\partial_i = \partial/\partial x_i)_{i=1, \dots, n}$ be the dual basis; it consists of mutually commutative derivations. The sheaf of logarithmic differential forms is
\[
\Omega_{X}^1(\log D)  = \bigoplus_{i=1}^m \calO_X \frac{dx_i}{x_i} \oplus \bigoplus_{i=m+1}^n \calO_X dx_i,
\]
as a subsheaf of $j_*\Omega^1_U$.
The dual basis of $\Omega_{X}^1(\log D)$ is given by $x_1\partial_1, \dots, x_m\partial_m, \partial_{m+1}, \dots, \partial_n$; they are also mutually commutative derivations.

By taking  the completion at a closed point, one can pass from (a) to (b) (with possibly a larger $k$); by taking the completion along the generic point of  $D_1$, one can pass from (a) or (b) to (c) (with possibly a larger $k$).

\begin{definition}
Keep the notation as above.
Let $\calD_X^{(\log)}$ denote the sheaf of rings of \emph{(logarithmic) differential operators}
on $X$ (over $k$); it is generated by $\calO_X$ and derivations 
$\partial_1, \dots,  \partial_n$ (resp. $x_1\partial_1, \dots, x_m\partial_m$, $\partial_{m+1}, \dots, \partial_n$).

We define the filtration $\calD_{X, \bullet}^{(\log)}$ on $\mathcal{D}_X^{(\log)}$
given by the order of differential operators, i.e. $\calD_{X, n}^{(\log)} = \{ D \in \calD_X^{(\log)} \,|\, D\textrm{ as a differential operator has order} \leq n\}$. In particular, $\calD_{X, n}^{(\log)} = 0$ if $n <0$.  The filtrations induce canonical isomorphisms
$$
\mathrm{gr}_\bullet(\mathcal{D}_X)\simeq \Sym^\bullet_{\calO_X}(\Omega_X^{1,\vee}), \quad
\mathrm{gr}_\bullet(\mathcal{D}_X^\log)\simeq \Sym^\bullet_{\calO_X}(\Omega^1_X(\log D)^\vee)
$$
We define the the \emph{logarithmic cotangent bundle} to be 
$
T^*\!X^{\log}= \Spec(\Sym^\bullet_{\calO_X}(\Omega_X^1(\log
D)^\vee)).
$

Now, we may apply the discussion of previous subsection to define, for a \emph{finitely generated} $\calD_X^{\log}$-module $M$, its \emph{log-characteristic cycle} $\Car(M)$.  (Here we omit the superscript $\log$ because we will exclusively study log-characteristic cycles in this paper.)
\end{definition}

\begin{caution}
\label{C:not-coherent}
Given a vector bundle $M$ over $U$ with an integrable connection, it is \emph{not} true in general that $j_*M$ is a finitely generated $\calD_X^\log$-module!  For example, if $M = \calO_U$ is the structure sheaf with trivial connection, $j_*\calO_U$ is \emph{not} coherent unless $U = X$.  Vaguely speaking, the nature of this trouble is caused by ``regular part" of $M$; whereas the ``irregular part" of $M$ is considered good.  To get around this trouble, we need to extend the definition of log-holonomicity to not necessarily finitely generated $\calD_X^\log$-modules.  (See Definition~\ref{D:log-char-cycle}.)

When $M$ is known to be regular along an irreducible component of $D$, one can avoid this non-finitely generated issue by taking the so-called Deligne-Malgrange extension.  However, in the situation of mixing regular and irregular, we do not know any sensible way of extending the vector bundle.  An even worse scenario is when $M$ is irregular along some irreducible component of $D$ generically, but when restricted to some particular curve (meeting this component transversally), $M$ becomes regular.
\end{caution}

\begin{caution}
\label{C:no-Berntein-inequality}
We also remark that the Bernstein inequality \emph{fails} for logarithmic $\calD$-modules.  For example, $X = \Spec (k[x])$, $D=  V(x)$, and $M = \frac 1{x} k[x] /k[x] \approx k$.  It is a $\calD_X^\log$-module (but \emph{not} a $\calD_X$-module) because $\calD_X^\log$ is generated by $x\partial_x$ (instead of $\partial_x$).  One computes easily that $\ZCar(X)$ is just the original point of $T^*X^\log$, which has dimension $0$.

Gaitsgory pointed out to me that the failure of Bernstein inequality is related to the fact that the Poisson structure on $T^*\!X^\log$ is degenerate over $D$.
Moreover, the degeneration of Poisson structure is also reflected in that irreducible components of the characteristic variety are not positioned to be conormal bundles.

Another minor point is that  taking log-characteristic cycles may not be additive for direct sums of ``log-holonomic" $\calD_X^\log$-modules because the lower dimensional pieces might be ``eaten up" by bigger dimensional ones; but the Euler characteristic is still additive in view of the analogous Kashiwara-Dubson formula (Theorem~\ref{T:log-Kashiwara-Dubson}).
\end{caution}

\begin{definition}
\label{D:log-holonomicity}
Assume that we are in the geometric situation \ref{SS:local-setup}(a).  Let $M$ be a (not necessarily finitely generated) $\calD_X^\log$-module.  We say that $M$ is \emph{log-holonomic} if for \emph{any} finitely generated $\calD_X^\log$-submodule $M_0 \subseteq M$, we have $\dim \Car(M_0) \leq n$.  Because of the lack of Bernstein inequality (Caution~\ref{C:no-Berntein-inequality}), this inequality may be strict.
\end{definition}

When $M$ comes from a vector bundle with an integrable connection, one can refine this definition and define log-characteristic cycles.

\begin{definition}
\label{D:log-char-cycle}
Keep the notation as in \ref{SS:local-setup}.  
Let $M$ be a vector bundle over $U$ with an integrable connection.  We choose a \emph{coherent} $\calO_X$-submodule $M_0$ of $j_*M$ such that $M_0|_U = M$.  Let $\widetilde M_0 = \calD_X^\log \cdot M_0$ denote the $\calD_X^\log$-submodule of $j_*M$ generated by $M_0$; it is automatically a coherent $\calD_X^\log$-module.  Define the \emph{log-characteristic cycle} of $j_*M$ to be $\ZCar(j_*M):=\ZCar(\widetilde M_0)$, the log-characteristic cycle of $\widetilde M_0$ as a $\calD_X^\log$-module.  This is independent of the choice of $M_0$ by  Lemma~\ref{L:log-char-independence} below.
\end{definition}

\begin{lemma}
\label{L:log-char-independence}
Let $M$ be a vector bundle over $U$ with an integrable connection.  As above, we choose coherent $\calO_X$-submodules $M_0$ and $M'_0$ of $j_*M$ such that $M_0|_U= M'_0|_U= M$, and we form $\widetilde M_0 = \calD_X^\log \cdot M_0$ and $\widetilde M'_0 = \calD_X^\log \cdot M'_0$.  Then, we have $\ZCar(\widetilde M_0)=\ZCar(\widetilde M'_0)$.
\end{lemma}
\begin{proof}
First, there exists $N \in \NN$ such that $(x_1 \cdots x_m)^N M_0 \subseteq M'_0 \subseteq (x_1 \cdots x_m)^{-N} M_0$.   By Remark~\ref{R:log-char-for-sub} (first matching the supports of two cycles and then checking the multiplicity at each generic point), it suffices to prove the lemma for the case $M'_0 = (x_1 \cdots x_m)^N M_0$ for any $N \in \NN$.  

Consider the map $\phi: M \to M$ given by $\phi(a) = (x_1 \cdots x_m)^N a$ for $a\in M$.  It induces an $\calO_X$-linear isomorphism between $M_0$ and $M'_0$.  We have $\phi(x_i\partial_i(a)) = x_i\partial_i(\phi(a)) - N\phi(a)$ for $i=1, \dots, m$ and $\partial_i(\phi(a)) = \phi(\partial_i(a))$ for $i =m+1, \dots, n$.

Note that, for $\alpha \in \ZZ_{\geq 0}$, $\calD_{X,\alpha}^\log$ is also generated over $\calO_X$ by polynomials in $x_1\partial_1 - N, \dots, x_m\partial_m - N, \partial_{m+1}, \dots, \partial_n$ of degree $\leq \alpha$.
This implies that $\phi(\calD_{X,\alpha}^\log \cdot M_0) = \calD_{X, \alpha}^\log \cdot M'_0$. In particular, $\phi(\widetilde M_0) = \widetilde M'_0$.
We take the good filtration $\fil_\bullet \widetilde M_0 = \calD_{X,\alpha}^\log \cdot M_0$ of $\widetilde M_0$ and then $\fil'_\bullet\widetilde M'_0 := \phi(\fil_\bullet\widetilde M_0) = \calD_{X, \alpha}^\log \cdot M'_0$ is also a good filtration for $\widetilde M'_0$.

Moreover, for any $\alpha \in \ZZ$, $i=1, \dots, m$, and any $a \in \fil_\alpha\widetilde M_0$, we have $\phi(x_i\partial_i(a)) - x_i\partial_i(\phi(a)) = -N\phi(a) \in \fil'_\alpha \widetilde M'_0$ which dies in $\gr'_{\alpha+1} \widetilde M'_0$.  Hence the isomorphism $\gr_\bullet \phi: \gr_\bullet \widetilde M_0 \to \gr_\bullet \widetilde M'_0$ is equivariant for the action of $x_i\partial_i$ (for $i =1, \dots, n$).  In other words, $\gr_\bullet \phi: \gr_\bullet \widetilde M_0 \to \gr_\bullet \widetilde M'_0$ is an isomorphism of $\gr_\bullet \calD_X^\log$-modules.  The statement of the lemma follows.
\end{proof}

\begin{remark}
It would be interesting to know if one can define logarithmic cycles for a more general class of holonomic $\calD_X$-modules.  (See also Theorem~\ref{T:log-holonomicity}.)

Also, it would be interesting to know if $\Car(j_*M)$ always has pure dimension $n$.  We will prove in Theorem~\ref{T:main-theorem} that this is the case if $M$ is clean in the sense of Definition~\ref{D:clean-global}.
\end{remark}

\begin{corollary}
\label{C:logcharcycles-base-change}
Assume that we are in one of the following situations:

(i) We are in the geometric local setup~\ref{SS:local-setup}(a).  Let $z$ be a closed point of $p^{-1}(\{0\})$.  Then we have a natural morphism $g: X' = \Spec \calO_{X, z}^\wedge \to X$; $g^*M$ may be viewed as a vector bundle over $U' = \Spec\big( \calO_{X, z}^\wedge[1/x_1 \cdots x_m]\big)$.  Write $j':  U'\to X'$ for the natural embedding.

(ii) We are in geometric or formal local setup \ref{SS:local-setup}(a)(b).  Let $\eta_1$ denote a generic point of $D_1:=V(x_1)$ and let $k(X)^{\wedge,\eta_1}$ denote the completion of $k(X)$ with respect to the valuation corresponding to $\eta_1$. We consider the natural morphism $g: X'=\Spec \calO_{X, \eta_1}^\wedge \to X$; $g^*M$ may be viewed as a vector bundle over $U' = \Spec (k(X)^{\wedge, \eta_1})$.  Write $j':  U'\to X'$ for the natural embedding.

(iii) We are in any local situation~\ref{SS:local-setup}.  Let $X''$ be \'etale over $X$ and let $X' = \Spec \big(\calO_{X''}[x_1^{1/h_1}, \dots, x_m^{1/h_m}]\big)$ for some positive integers $h_1, \dots, h_m$.  We have a natural morphism $g: X' \to X$ and $g^*M$ becomes a vector bundle over $U'= \Spec \big(\calO_{X'}[1/x_1 \cdots x_m]\big)$; write $j': U' \to X'$ for the natural embedding.

Then we have $\ZCar(j'_*g^*M) = \tilde g^*(\ZCar(j_*M))$, where $\tilde g: T^*\!X'^\log \to T^*\!X^\log$ is the natural morphism.
\end{corollary}
\begin{proof}
This follows from combining Lemma~\ref{L:charcycles-base-change} with Lemma~\ref{L:log-char-independence}.
\end{proof}

\subsubsection{Global situation}
\label{SS:global-situation}
Let $(X, D)$ be a smooth pair, i.e., $X$ is a smooth variety of dimension $n$ over $k$ and $D=\bigcup D_j$ is
a divisor with  simple normal crossings, where $D_j$'s are  irreducible
components of $D$.  Denote $U = X-D$.
Here \emph{simple normal crossings} means that, {\it each irreducible component of $D$ is smooth} and we can cover $X$ by open subvarieties $\{V_i\}$, each of which is as in the geometric local situation \ref{SS:local-setup}(a).  For each $i$, let $j_i: V_i \cap U \to V_i$ be the natural embedding.  Then the definition of logarithmic cotangent bundle on each $V_i$ glues and gives the \emph{logarithmic cotangent bundle} $T^*\!X^\log$.  Also, there is a quasi-coherent sheaf $\calD_X^\log$ of $k$-algebras whose restriction to each $V_i$ is $\calD_{V_i}^\log$.

\begin{definition}
Let $M$ be a vector bundle over $U$ with an integrable connection or a coherent $\calD_X^\log$-module.  We define the \emph{log-characteristic variety} (resp. \emph{log-characteristic cycle}) of $j_*M$ or $M$  to be the subvariety (resp. cycle) of $T^*\!X^\log$ whose restriction to $V_i$ is the log-characteristic variety (resp. log-characteristic cycle) of $j_*(M|_{V_i \cap U})$ or $M|_{V_i}$.  We denote them by $\Car(j_*M)$ (resp. $\ZCar(j_*M)$) or $\Car(M)$ (resp. $\ZCar(M)$).
\end{definition}

\subsection{Holonomicity v.s. log-holonomicity}
\label{S:log-holonomicity}

In this subsection, we study the relation between holonomicity and log-holonomicity.
The goal of this subsection is to prove the following.

\begin{theorem}
\label{T:log-holonomicity}
Assume that we are in the geometric local situation \ref{SS:local-setup}(a).
Let $M$ be a holonomic $\calD_X$-module, and hence also a (not necessarily finitely generated) $\calD_X^\log$-module.
Then as a $\calD_X^\log$-module, $M$ is log-holonomic in the sense of Definition~\ref{D:log-holonomicity}.
\end{theorem}
\begin{proof}
The proof uses a standard trick of Bernstein, which we found in the lecture notes of Braverman and Chmutova \cite{braverman}.

We first prove this theorem when $X = \AAA^n$ and $D = V(x_1\cdots x_m)$.  In this case, 
\[
\calD_X = k[x_1, \dots, x_n]\{ \partial_1, \dots, \partial_n\}
\textrm{ and }
\calD_X^\log = k[x_1, \dots, x_n]\{ x_1\partial_1, \dots, x_m\partial_m,\partial_{m+1}, \dots, \partial_n\},
\]
where the curly brackets mean that the corresponding $k$-algebras are not commutative and the generators satisfy natural relations.
We provide the two $k$-algebras with new filtrations: for $\alpha \in \ZZ_{\geq 0}$, $\fil'_\alpha \calD_X$ is the $k$-vector subspace of $\calD_X$ generated by $s_1\cdots s_\alpha$, where each $s_i \in \{x_1, \dots, x_n, \partial_1, \dots, \partial_n\}$; and $\fil'_\alpha \calD_X^\log$ is the $k$-vector subspace of $\calD_X^\log$ generated by $s_1\cdots s_\alpha$, where each $s_i $ belongs to $ \{x_1, \dots, x_n, x_1\partial_1, \dots, x_m\partial_m,\partial_{m+1}, \dots, \partial_n\}$.  In other words, we require each $x_i$ to have degree $1$ instead of $0$.  With respect to the new filtrations, we have
\[
\gr'_\bullet \calD_X = k[x_1, \dots, x_n, \xi_1, \dots, \xi_n] \textrm{ and } \gr'_\bullet \calD_X^\log = k[x_1, \dots, x_n, \xi^\log_1, \dots, \xi^\log_n],
\]
where $\xi_i$ is a proxy of $\partial_i$ for each $i$ and $\xi_i^\log$ is a proxy of $x_i\partial_i$ if $i \leq m$ and of $\partial_i$ if $i >m$.  In particular, they are all free commutative polynomial rings with $2n$ variables.

By Proposition~\ref{P:bernstein-trick} and Remark~\ref{R:bernstein-trick}, we know the holonomicity of $M$ with respect to the filtration  $\fil'_\bullet \calD_X$ and we need only to prove that, for any finitely generated $\calD_X^\log$-submodule $M_0\subseteq M$, we have $\dim (\Car'(M_0)) \leq n$ for the new filtration $\fil'_\bullet D_X^\log$.  

Now, applying (the argument of) the numerical Lemma~\ref{L:braverman} below to $\calD_X$ and $M$, we get a filtration $\fil'_\bullet M$ good for $\fil'_\bullet\calD_X$ such that
$\dim_k \fil'_\alpha M \leq c \alpha^n$ for all $\alpha \in \NN$ and for some fixed constant $c>0$.  Then for any finitely generated $\calD_X^\log$-submodule $M_0 \subseteq M$, we define a filtration by $\fil'_\alpha M_0 = \fil'_{2\alpha} M \cap M_0$ for any $\alpha \in \ZZ$; it is admissible (but almost never good).  However, we have $\dim_k \fil'_\alpha M_0 \leq \dim_k \fil'_{2\alpha} M\leq c\cdot 2^n \cdot \alpha^n$ for all $\alpha \in \NN$.  Apply the other direction of Lemma~\ref{L:braverman}, we deduce $\dim (\Car'(M_0)) \leq n$.

Now, we reduce the general case to the special case we studied above.  First, we recall that $X$ comes with an \'etale morphism $p: X \to \AAA^n$. It is well-known that $p_*M$ is still holonomic for $\calD_{\AAA^n}$ and hence the above argument implies that $p_* M$ is log-holonomic for $\calD_{\AAA^n}^\log$.  By Lemma~\ref{L:charcycles-base-change}(a), we know that $p^*p_*M$ is also log-holonomic for $\calD_{X}^\log$.  The natural homomorphism $M \to p^*p_*M$ is injective, yielding the log-holonomicity of $M$ itself.
\end{proof}

\begin{lemma}
\label{L:braverman}
Let $(D, \fil_\bullet D)$ be an abstract filtered $k$-algebra as in \ref{SS:filtered-ring} such that $\gr_\bullet D \simeq k[x_1, \dots, x_{2n}]$ is a free commutative polynomial algebra with $2n$ variables.  Let $M$ be a finitely generated $D$-module.  Then for any $r \in \ZZ_{\geq0}$, $\dim(\Car(M)) \leq r$ if and only if there exists an admissible \emph{(not necessarily good)} filtration $\fil_\bullet M$ on $M$ and a constant $c \in \RR_{>0}$ such that 
\[
\dim_k (\fil_\alpha M )\leq c \alpha^r, \textrm{ for all } \alpha \geq 1.
\]
\end{lemma}
\begin{proof}
We found this lemma in the lecture notes on algebraic $\calD$-modules by Braverman and Chmutova \cite[Corollary~2.10]{braverman}.  Since this is not a proper reference, we include the proof here.

We first assume that $\dim(\Car(M)) \leq r$.  We provide $M$ with a good filtration and the condition implies that the function $h(\alpha) = \dim_k (\fil_\alpha M)$ is the Hilbert polynomial for $\gr_\bullet M$ when $\alpha \gg 0$.  In particular, $h(\alpha) = c_1 \alpha^s + $ lower degree terms, where $s =\dim(\Car(M)) \leq r$.  This gives the estimate we want.

Conversely, if $\dim_k \fil_\alpha M \leq c\alpha^r$ for some admissible filtration $\fil_\bullet M$, we first ``refine" it into a good filtration.  Choose $\beta \in \ZZ$ such that $\fil_\beta M$ generates $M$ (as a left $D$-module) and define a (good) filtration on $M$ by $\fil'_\alpha M= \fil_\alpha M$ if $\alpha \leq \beta$ and $\fil'_\alpha M = \fil_{\alpha-\beta}D \cdot \fil_\beta M$ if $\alpha > \beta$.  Now, $\dim_k (\fil'_\alpha M) \leq \dim_k (\fil_\alpha M) \leq c \alpha^r$ for $\alpha \in \NN$.
By the Hilbert polynomial argument above, we have $\dim(\Car(M)) \leq r$.
\end{proof}

\begin{remark}
In the formal or CDVF local situation \ref{SS:local-setup}(b)(c), we do not know whether the analogous Theorem~\ref{T:log-holonomicity} still holds.
\end{remark}

\subsection{Logarithmic Kashiwara-Dubson formula}

The classical Kashiwara-Dubson formula expresses the Euler characteristic of the de Rham cohomology of a holonomic $\calD_X$-module in terms of the intersection number of its characteristic cycle with the zero section of the cotangent bundle.  However, the logarithmic variant of such formula is more delicate because we need to deal with $\calD_X^\log$-modules which are not finitely generated.  
We restrict ourselves to the case of $\calD_X^\log$-modules coming from a vector bundle with an integrable connection (see Theorem~\ref{T:log-Kashiwara-Dubson}).  The formula will be conditional on a finite generation hypothesis, which holds in our later application.

We first record the following formula which is valid for finitely generated $\calD_X^\log$-modules.

\begin{theorem}[Kashiwara-Dubson]\label{T:Kashiwara-Dubson}
Assume that $X$ is proper and $M$ is a \emph{coherent} log-holonomic $\calD_X^\log$-module. Then the Euler characteristic of the log-de Rham cohomology of $M$ is
\begin{equation}
\label{E:Kashiwara-Dubson}
\chi_\dR(M):=\sum_{i=1}^{2n} (-1)^i \dim \mathrm H^i \big(X, M \otimes \Omega_X^\bullet(\log D)\big) = (-1)^n\cdot\deg([X], \ZCar(M))_{T^*\!X^\log},
\end{equation}
where $[X]$ is the zero section and $(\cdot, \cdot)_{T^*\!X^\log}$ is the intersection in $T^*\!X^\log$.
\end{theorem}
\begin{proof}
This theorem is still in the classical realm; its proof may be found in many references (e.g, \cite{laumon}). 
\end{proof}

\begin{remark}
\label{R:naive approach not working}
Assume that we are in the global situation \ref{SS:global-situation}.
 We choose a coherent $\calO_X$-submodule $M_0$ of $j_*M$ such that $M_0|_U = M$.
By Lemma~\ref{L:log-char-independence}, we know that $\ZCar(j_*M)$ is well-defined and does not depend on the choice of $M_0$.  One may try to na\"ively take the direct limit of the above theorem over all such $M_0$ to compute the Euler characteristic of $j_*M$.  However, one has to verify that \emph{each} cohomological group stabilizes in the direct limit.
Unfortunately, we do not know how to verify this condition.  In our application later, we need to assume that all irregularities are positive so that $\widetilde M_0 = j_*M$, to avoid this technical difficulty.  It would be interesting to know if such a hypothesis may be removed.
\end{remark}

\begin{theorem}
\label{T:log-Kashiwara-Dubson}
Assume that we are in the global situation \ref{SS:global-situation} and assume that $X$ is proper.  Let $M$ be a vector bundle over $U$ with an integrable connection.  Let $M_0$ be a coherent $\calO_X$-submodule of $j_*M$ such that $M_0|_U =  M$.  For $r\in \NN$, we put $\widetilde M_{r} : = \calD_X^\log \cdot M_0(rD)$.

Suppose that the natural morphisms
\[
\rmH^\star\big(X, \widetilde M_{ r} \otimes_{\calO_X} \Omega^\bullet_X(\log D) \big) \rightarrow \rmH^\star\big(U, M \otimes_{\calO_U} \Omega_U^\bullet \big)
\]
are isomorphisms for $r$ sufficiently large.  Then
 the Euler characteristic of $M$ is
\[
\chi_\dR(M):=\sum_{i=1}^{2n} (-1)^i \dim \mathrm H^i \big(U, M \otimes \Omega_U^\bullet\big) = (-1)^n\cdot\deg([X], \ZCar(j_*M))_{T^*\!X^\log}
\]
\end{theorem}
\begin{proof}
The statement follows from classical Kashiwara-Dubson formula:
\begin{align*}
\chi_\dR(M)&= \sum_{i=1}^{2n} (-1)^i \dim \mathrm H^i \big(X, \widetilde M_r \otimes \Omega_X^\bullet(\log D) \big)\\
&\stackrel{\eqref{E:Kashiwara-Dubson}}
=(-1)^n\cdot\deg([X], \ZCar(\widetilde M_r))_{T^*\!X^\log}\\
& = (-1)^n\cdot\deg([X], \ZCar(j_*M))_{T^*\!X^\log}.
\end{align*}
Here the first equality uses our hypothesis.
\end{proof}

\section{Nonarchimedean differential modules }
In this section, we first recall the theory of nonarchimedean differential modules, and then we discuss various definitions of cleanliness condition.

\subsection{Differential modules over a field}
We first recall the definition of irregularities and refined irregularities.  Some of our setup is made specific to the case of residual characteristic zero.  For more details or a general treatment including positive residual characteristic case, one may consult \cite{kedlaya-xiao, xiao-refined}.

\begin{notation}
For a nonarchimedean field $(F, |\cdot|)$, we use $\gotho_F$ to denote the valuation ring and $\kappa_F$  the residue field.  For $s \in \RR_{>0}$, we put $\mathfrak m_F^{<s} = \{x \in F | |x|<s\}$, $\mathfrak m_F^{\leq s} = \{x \in F | |x|\leq s\}$, and $\kappa_F^{[s]} = \mathfrak m_F^{\leq s} / \mathfrak m_F^{<s}$; in particular, $\kappa_F^{[1]} = \kappa_F$.

In case when $F$ is discretely valued, we fix a uniformizer $\pi_F$.  We frequently write $\pi_F^b \kappa_{F^\alg}$ for $b \in \QQ$ to mean $\kappa_{F^\alg}^{[|\pi_F|^b]}$.  This should not cause any ambiguity.
\end{notation}

\begin{notation}
\label{N:bounded-disc}
For a complete discrete valuation field $(F, |\cdot|)$, we use $F\llbracket t\rrbracket_0$ to denote the ring of bounded functions on an open unit disc over $F$.  Put it another way, $F\llbracket t \rrbracket_0 = \gotho_F\llbracket t \rrbracket[\frac 1{\pi_F}]$.
\end{notation}

\subsubsection{Setup}
\label{SS:setup-field}
Let $(F, |\cdot|)$ be a complete nonarchimedean field with residual characteristic zero.  (We do not exclude the case when $F$ is trivially normed.)    Assume that $F$ admits $n$ commuting derivations $\partial_1,\dots,\partial_n$ of \emph{rational type}, i.e. there exist elements $x_1, \dots, x_n \in F$ (called \emph{rational parameters}) such that
\[
\partial_i(x_j) = \left\{ \begin{array}{ll}1 & \textrm{if }i =j,\\ 0 &\textrm{if } i \neq j,\end{array}\right. \quad \textrm{and the operator norm } \ |\partial_i|_F = |x_i|^{-1} \textrm{ for any } i.
\]

A \emph{$(\partial_1, \dots, \partial_n)$-differential module} (or simply \emph{differential module}) over $F$ is a finite dimensional $F$-vector space $V$ with commuting actions of $\partial_1,\dots,\partial_n$,
satisfying the Leibniz rule.

\begin{remark}
\label{R:enlarge-F}
We remark that the condition $\partial_1, \dots, \partial_n$ being of rational type with respect to $x_1, \dots, x_n$ is preserved if 
\begin{itemize}
\item[(i)] we replace $F$ by a finite extension \cite[Lemma~1.4.5]{kedlaya-xiao}, or 
\item[(ii)] we replace $F$ by the completion of $F(t)$ with respect to $\eta$-Gauss norm for some $\eta$ and declare $\partial_j(t) = 0$ for any $j$.
\end{itemize}
In particular, if we take $\eta = |x_j|$ in the second case, then $\partial_1, \dots, t\partial_j, \dots, \partial_n$ are rational type with respect to $x_1, \dots, x_j/t, \dots, x_n$.
\end{remark}

\subsubsection{$\partial$-radii}
We first assume that $n=1$ and write $\partial$ for $\partial_1$ and $x$ for $x_1$.

For a differential module $V$ over $F$, we define the \emph{$\partial$-radius} and \emph{intrinsic $\partial$-radius} of $V$ to be 
\[
R_\partial(V) = \big(\lim_{s\to \infty}|\partial^s|_{V}^{1/s}\big)^{-1}, \textrm{ and }IR_\partial(V) = |x|^{-1} \cdot R_\partial(V),
\]
where $|\partial^s|_V$ is the operator norm for a fixed $F$-norm $|\cdot|_V$ on $V$.  The definition of (intrinsic) $\partial$-radii does not depend on the choice of the norm $|\cdot|_V$.
We always have $IR_\partial(V) \leq 1$.

Let $V_1, \dots, V_r$ denote the Jordan-H\"older constituents of $V$, as $\partial$-differential modules over $F$.
We use $\calR_\partial(V)$ to denote the multiset consisting of, for every $i$, the number $R_\partial(V_i)$ with multiplicity $\dim V_i$.
We say that $V$ has \emph{pure (intrinsic) $\partial$-radius} if all $R_\partial(V_i)$'s are the same.  By \cite[Theorem~1.4.21]{kedlaya-xiao}, $V$ can be uniquely written as the direct sum of differential modules with distinct \emph{pure} $\partial$-radius.

\subsubsection{Partially intrinsic radii}
\label{SS:log-structure}
For general $n$, we will specify a \emph{log-structure}, that is a subset $\Log$ of $\{\partial_1,\dots, \partial_n\}$.  Without loss of generality, we assume that $\Log = \{\partial_1,\dots, \partial_m\}$ for a fixed nonnegative integer $m\leq n$.  We write $\Log^* = \{\frac{dx_1}{x_1}, \dots, \frac{dx_m}{x_m}, dx_{m+1}, \dots, dx_n\}$.

If $V$ is a differential module, we define the \emph{partially intrinsic radius} (or \emph{intrinsic radius} if $m=n$) to be
\[
IR^\sharp(V) = \max\{IR_{\partial_1}(V), \dots, IR_{\partial_m}(V), R_{\partial_{m+1}}(V), \dots, R_{\partial_n}(V)\};
\]
here we singled out derivatives $\partial_1, \dots, \partial_m$ in the log-structure $\Log$ to take their intrinsic radii instead of radii.

In general, by \cite[Theorem~1.5.6]{kedlaya-xiao}, $V$ may be (uniquely) written as the direct sum $V_1 \oplus \cdots \oplus V_r$ of differential modules, where each $V_i$ has \emph{pure} $\partial_j$-radius for all $j$.  We define the \emph{partially intrinsic subsidiary radii} (or \emph{intrinsic subsidiary radii} if $m=n$) to be the multiset $\calI\calR^\sharp(V)$ consisting of $IR^\sharp(V_i)$ with multiplicity $\dim V_i$ for $i = 1, \dots, r$.  Let $IR^\sharp(V) = IR^\sharp(V;1) \leq \cdots \leq IR^\sharp(V;\dim V)$ denote the elements of $\calI\calR^\sharp(V)$ in increasing order.
We say that $V$ has \emph{pure (partially) intrinsic radius} if $\calI\calR^\sharp(V)$ consists of $\dim V$ copies of $IR^\sharp(V)$.

\subsubsection{Irregularities}
\label{SS:irregularities}
Assume that $F$ is discretely valued, $x_1 = \pi_F$, and $\partial_1 \in \Log$.  Assume moreover that  $x_2, \dots, x_n \in \gotho_F^\times$.  We define the \emph{subsidiary irregularities} $\Irr(V; i) = \log_{|\pi_F|} IR^\sharp(V; i)$ and $\mathcal{I}rr(V) = \{\Irr(V)=\Irr(V;1), \dots, \Irr(V;\dim V)\}$; they are nonnegative rational numbers by \cite[Theorem~1.5.6]{kedlaya-xiao}.
The usual irregularities would be the sum of all subsidiary irregularities, but we will not use this concept in this paper.
  In case that all subsidiary irregularities are the same rational number, we say that $V$ has \emph{pure irregularity}.  If $\Irr(V) = 0$, we say that $V$ is \emph{regular}.  We remark that all definitions of irregularities do not depend on the log-structure as long as $\partial_1 \in \Log$.

\subsubsection{Refined radii and refined irregularities}
\label{SS:refined-irr}
Let $V$ be a differential module over $F$.  We fix a log-structure as in \ref{SS:log-structure}.  We further assume that $|x_{m+1}| \leq 1,\dots, |x_n|\leq 1$.
We now recall the definition of refined radii from \cite{xiao-refined}.  To avoid an unimportant technical difficulty, we will only give the definition in the case when $F$ is discretely valued; for the general case, we refer the readers to \cite[Definition~1.4.10]{xiao-refined} and Remark~\ref{R:refined general} later.

Assume that $F$ is discretely valued.
We first assume that $V$ has pure partially  intrinsic radius $IR^\sharp(V)$.  By \cite[Lemma~1.4.14]{xiao-refined}, there exists a norm $|\cdot|_V$ on $V$ such that 
\begin{itemize}
\item[(i)] it admits an orthogonal basis, and 
\item[(ii)] the operator norms $|x_j\partial_j|_V \leq IR^\sharp(V)^{-1}$ for $j = 1, \dots, m$ and $|\partial_j|_V \leq IR^\sharp(V)^{-1}$ for $j = m+1, \dots, n$;
\end{itemize}
we call such a norm \emph{good}.  (Note that this is weaker than the convention used in \cite[Definition~1.4.11]{xiao-refined} for multi-derivative case $n>1$.)

If the partially intrinsic radii $IR^\sharp(V) < \min\{1,|x_{m+1}|, \dots, |x_n|\}$, by possibly enlarging the valued group of $F$ as in \ref{SS:setup-field}, we may assume that $|\cdot|_V$ admits an orthonormal basis.
In this case, let $N_j$ denote the matrix of $x_j\partial_j$ if $j \leq m$ or of $\partial_j$ if $j>m$ acting on this chosen basis; they commute with each other.  The \emph{refined partially intrinsic radii} (or \emph{refined intrinsic radii} if $m=n$) is defined to be the multiset $
\mathcal{R}\mathrm{ef}^\sharp(V)$ consisting of 
\begin{equation}
\label{E:definition of refined radii}
\theta_1 \frac{dx_1}{x_1} + \cdots + \theta_m \frac{dx_m}{x_m} + \theta_{m+1} dx_{m+1} + \cdots+ \theta_ndx_n \in\bigoplus_{\omega \in \Log^*} \kappa_{F^\alg}^{[IR^\sharp(V)^{-1}]} \omega
\end{equation}
 for each common (genearalized) eigenvalues $(\theta_1, \dots, \theta_n)$ (with multiplicities) of $N_1, \dots, N_n$ modulo $\mathfrak m_{F^\alg}^{<IR^\sharp(V)^{-1}}$.

If $IR^\sharp(V) \geq \min\{1,|x_{m+1}|, \dots, |x_n|\}$, we conventionally write $\mathcal{R}\mathrm{ef}^\sharp(V) = \{0, \dots, 0\}$, a multiset consisting only $0$ of multiplicity $\dim V$.

For a general differential module $V$, applying the above construction to the Jordan-H\"older factors $V_1, \dots, V_r$ of $V$, we define $\mathcal{R}\mathrm{ef}^\sharp(V) = \cup_{i=1}^r \mathcal{R}\mathrm{ef}^\sharp(V_i)$.  We say $V$ has \emph{pure refined intrinsic radius} (or \emph{pure refined irregularity} if $F$ satisfies the conditions in \ref{SS:irregularities}) if $\mathcal{R}\mathrm{ef}^\sharp(V)$ consists of multiples of a same element. We order the elements in $\mathcal R \mathrm{ef}^\sharp(V)$ as $\Ref^\sharp(V;1), \dots, \Ref^\sharp(V;\dim V)$ so that
$
\Ref^\sharp(V; i) \in \bigoplus_{\omega \in \Log^*} \kappa_{F^\alg}^{[IR^\sharp(V; i)^{-1}]} \omega.
$
This choice of order may not be unique; we fix such a choice; however see Remarks~\ref{R:order-refined-irr} and \ref{R:no-global-defn-clean}.

When $F$ satisfies the conditions in \ref{SS:irregularities}, we also call $\calR\mathrm{ef}^\sharp(V)$ the \emph{refined irregularities} of $V$.  In particular, $\Ref^\sharp(V; i) \in \bigoplus_{\omega \in \Log^*} (\pi_F^{-\Irr(V; i)} \kappa_{F^\alg})\omega$.  Again, in this case, the definition of $\calR\mathrm{ef}^\sharp(V)$ does not depend on the choice of log-structure if we identify $\bigoplus_{\omega \in \Log^*} (\pi_F^{-\Irr(V; i)} \kappa_{F^\alg})\omega$ with $\bigoplus_{j=1}^n (\pi_F^{-\Irr(V; i)} \kappa_{F^\alg})\frac{dx_j}{x_j}$. Sometimes we omit the sharp from the notation for simplicity.

\begin{remark}
\label{R:refined general}
We now sketch the definition of refined partially intrinsic radii in the case when $F$ is not necessarily discretely valued.  For details, we refer to \cite[Definition~1.4.10]{xiao-refined} and the discussion preceding it.

The definition proceeds in several steps: we first define the refined radii for each one of the derivations; this is possible because we may always find a good norm good for one derivation (\cite[Lemma~1.3.9]{xiao-refined}).  Then we show that (\cite[Theorem~1.3.26]{xiao-refined}), up to making a finite extension of $F$, we may decompose $V$ into a direct sum of differential modules with pure refined radius for each $\partial_i$.  For each direct summand, we define the refined partially intrinsic radii by assembling the refined radii for each $\partial_i$ as in \eqref{E:definition of refined radii}.
\end{remark}

\begin{remark}
\label{R:refined-CDVF}
When $F$ satisfies the conditions in \ref{SS:irregularities}, we have an additional restriction on the refined irregularities.
For $i$ such that $\Irr(V;i)>0$, write $\Ref^\sharp(V;i) = x_1^{-\Irr(V;i)} \big(\theta_1 \frac{dx_1}{x_1} + \cdots + \theta_n \frac{dx_n}{x_n}\big)$ for $\theta_1, \dots, \theta_n \in \kappa_F^\alg$.  By \cite[Proposition~1.4.17]{xiao-refined}, we have $x_j\partial_j(x_1^{-\Irr(V;i)}\theta_1) = x_1 \partial_1(x_1^{-\Irr(V;i)}\theta_j)$ as an equality in $x_1^{-\Irr(V;i)} \kappa_F^\alg$ for any $j \neq 1$.  This implies that 
$
\theta_j = - x_j\partial_j(\theta_1)\big/\Irr(V;i) .
$
In other words, every $\theta_j$ for $j \neq 1$ is determined by $\theta_1$.  In particular, $\theta_1 \neq 0$.  (This fact is also hinted by \cite[Proposition~2.5.4]{kedlaya-sabbah1}.)
We also point out that similar phenomenon does not happen for the mixed characteristic analogue.
\end{remark}

\begin{prop}
\label{P:decomposition-field}
Let $V$ be a differential module over $F$.

(i) If $F$ is discretely valued, the sum $\Irr(V; 1) + \cdots + \Irr(V; \dim V)$ is a nonnegative integer.

(ii) We have a unique direct sum decomposition $V = \bigoplus_{r \in (0,1]}V_r$ of differential modules such that $V_r$ has pure partial intrinsic radius $IR^\sharp(V_r) = r$.

(iii) Assume that $F$ satisfies the conditions in \ref{SS:irregularities}.  If $F'$ is a \emph{finite} extension of $F$ of ramification degree $h$ such that $h\cdot\Irr(V;i) \in \ZZ$ and $\Ref(V;i) \in \bigoplus_{j=1}^n (\pi_{F'}^{-h\cdot \Irr(V; i)} \kappa_{F'}) \frac{dx_j}{x_j}$, we obtain a unique  direct sum decomposition $V \otimes F' = \bigoplus_{\vartheta} V_\vartheta$ of differential modules over $F'$, where the direct sum runs over all $\vartheta \in \bigoplus_{j=1}^n \pi_{F'}^b \kappa_{F'}\frac{dx_j}{x_j}$ for some $b \in \NN$, such that every $V_\vartheta$ has pure irregularity $b$ and pure refined irregularity $\vartheta$.

Moreover, if we group $\mathcal R\mathrm{ef}(V)$ into $G = \mathrm{Gal}(F' / F)$-orbits $\{G\vartheta\}$, the above decomposition descents to a unique direct sum  decomposition over $F$: $V = \bigoplus_{\{G\vartheta\}} V_{\{G\vartheta\}}$, where $V_{\{G\vartheta\}}$ has refined irregularities in $\{G\vartheta\}$ with same multiplicity on each element in $\{G\vartheta\}$.
\end{prop}
\begin{proof}
(i) and (ii) are well known; see for example \cite[Proposition~1.3.4]{kedlaya-xiao}.  (iii) is proved in \cite[Theorem~1.3.26]{xiao-refined}.
\end{proof}

\begin{corollary}
\label{C:strong-integrality}
Assume that $F$ satisfies the conditions in \ref{SS:irregularities}. 
Let $V$ be a differential module over $F$ with pure irregularity.  Assume that all refined irregularities of $V$ form several copies of a same $\mathrm{Gal}(F' / F)$-orbit for some finite Galois extension $F'$ of $F$. Let $r$ be the number of elements in the $\mathrm{Gal}(\kappa_{F'} / \kappa_F)$-orbit of an element in $\mathcal R \mathrm{ef}^\sharp(V)$.  
Then $\dim V \cdot \Irr(V) / r \in \ZZ$.
\end{corollary}
\begin{proof}
Let $F''$ be the maximal unramified extension of $F$ inside $F'$.  
Applying the decomposition of Proposition~\ref{P:decomposition-field}(iii) to $V'' = V \otimes_F F''$, as a differential module over $F''$, we see that $V''$ splits into a direct sum of $r$ differential modules, whose refined irregularities give the $\Gal(F'/F'')$-orbits of the refined irregularities of $V$.  Hence, each direct summand is of dimension $\dim V/r$ over $F''$.
Applying Proposition~\ref{P:decomposition-field}(i) to any such direct summand shows that $\Irr(V) \cdot \dim V/r \in \ZZ$.
\end{proof}

\begin{remark}
When the differential operators are not of rational type, all above definitions and results are still valid, if the radii is strictly bigger than the the inverse of operator norms.  See \cite[Remarks~1.3.29, 1.4.22]{xiao-refined}.
\end{remark}

We record a technical but useful lemma for future reference.

\begin{lemma}
\label{L:refined-integral}
Let $R$ be a unique factorization domain of characteristic zero and let $S = R((\pi_F))$.  We write $F$ for the completion of $\Frac(S)$ with respect to the $\pi_F$-adic valuation, and let $|\cdot|_F$ denote a norm on $F$ given by the $\pi_F$-adic valuation. Let $R^\alg$ denote the integral closure of $R$ in an algebraic closure of $\Frac(R)$.

Assume either $\partial$ is a nontrivial derivation on $R$, extended to $F$ naturally by setting $\partial(\pi_F) =0$, or $\partial = \partial / \partial \pi_F$.  In the former case, we assume that $\partial$ is of rational type with respect to some $u \in R$; in the latter case, we set $u = \pi_F$.  When talking about $\partial$-differential modules, we take $\Log = \emptyset$. 
Let $M$ be a $\partial$-differential module over $S$, that is a locally free module over $S$ of finite rank $d$, equipped with an action of $\partial$ subject to Leibniz rule.  Assume that $R_\partial(M\otimes F) = |\pi_F|^b< |u|$ and let $M_b$ be the unique differential submodule of $M \otimes F$ with pure $\partial$-radius $|\pi_F|^b$.  Then all refined partially intrinsic radii $\Ref^\sharp(M_b)$ 
actually lie in $\pi_F^{-b}R^\alg \subseteq \pi_F^{-b}\Frac(R)^\alg$, in other words, they are reductions of the zeros of some monic polynomial $X^d+a_1 X^{d-1} +\cdots +a_d $ with $a_i \in \pi_F^{\lceil ib\rceil} R$ for $i=1,\dots, d$.
\end{lemma}
\begin{proof}
Let $r$ be the multiplicity of $b$ in $\calR_\partial(V \otimes F)$, which is equal to $\dim_F M_b$.  To prove the lemma, we may adjoin $\pi_F^{1/N}$ to $S$ for some appropriate $N$ so that, by \cite[Theorem~1.4.21]{kedlaya-xiao}, we may assume that all radii in $\calR_\partial(M \otimes F)$ are integer powers of $|\pi_F|$.

Let $\bbv \in M \otimes \Frac(S)$ be a  cyclic vector  (see e.g. \cite[Theorem~5.4.2]{kedlaya-course}) so that $\bbv, \partial(\bbv), \dots, \partial^{d-1}(\bbv)$ form a basis of $M \otimes \Frac(S)$ over $\Frac(S)$; the action of $\partial$ is determined by $\big(\partial^d + a_1\pi_F^{-b}\partial^{d-1} + \dots + a_d\pi_F^{-bd} \big)\bbv = 0$ for some $a_1, \dots, a_d \in \Frac(S)$.  By \cite[Proposition~1.3.2]{kedlaya-xiao}, we have $a_i \in \gotho_F$ for $i=1, \dots, r-1$, $a_r \in \gotho_F^\times$, and $a_i \in \pi_F\gotho_F$ for $i = r+1, \dots, d$.  Moreover, by \cite[Corollary~1.3.13]{xiao-refined}, the reductions of the roots of $X^r + a_1 X^{r-1} + \cdots +a_r=0$ in $\kappa_F^\alg = \Frac(R)^\alg$ are exactly $\pi_F^b \Ref^\sharp(M_b)$.

Now, to prove the lemma, it suffices to show that the reduction $\bar a_i$ of each $a_i$ in $\kappa_F = \Frac(R)$ lies in $R$.
This is equivalent to proving that, for any irreducible element $\lambda$ of $R$ with $v_\lambda$ the corresponding valuation on $R$, $v_\lambda(\bar a_i) \geq 0$ for every $i=1, \dots, r$.
We fix such an irreducible $\lambda$.

Let $R_\lambda$ denote the valuation ring in the $v_\lambda$-adic completion of $\Frac(R)$; it may be also written as $\kappa_\lambda\llbracket \lambda \rrbracket$ with residue field $\kappa_\lambda$.  Since $\partial$ preserves $R$, it extends to a \emph{continuous} derivation on $R_\lambda$.  Set $S_\lambda = R_\lambda((\pi_F)) = \kappa_\lambda((\pi_F))\llbracket \lambda\rrbracket_0$ (see Notation~\ref{N:bounded-disc}) and  $M_\lambda = M \otimes_S S_\lambda$.  Let $\widetilde F$ denote the completion of $\Frac (S_\lambda)$ for the $\pi_F$-adic valuation; it contains $F$ as a subfield (with compatible norms).  Now, $\bbv$ is also a cyclic vector of $M_\lambda \otimes \widetilde F$ and we take a basis of $M_\lambda \otimes \widetilde F$ to be $\bbv, \pi_F^{-b}\partial(\bbv), \dots, \pi_F^{-b(d-1)} \partial^{d-1}(\bbv)$; it gives a norm on $M_\lambda \otimes \widetilde F$.  Let $A$ denote the matrix of $\partial$ acting on this basis.
By the lattice lemma \cite[Lemma~2.2.3]{kedlaya-xiao}, we can find a basis $m_1, \dots, m_d$ of $M_\lambda$ (over $S_\lambda$) defining the same norm restricted from $M_\lambda \otimes \widetilde F$.  

Now we let $B$ denote matrix of $\partial$ acting on this new basis and let $X^d + \tilde a_1\pi_F^{-b} X^{d-1} + \cdots + \tilde  a_d\pi_F^{-db}$ denote the characteristic polynomial of $B$.  If we use $N \in \GL_d(\gotho_{\widetilde F})$ to denote the transformation matrix between the two bases, we have $B = N^{-1}AN + N^{-1} \partial(N)$.  We know that $|N^{-1} \partial(N)| \leq |\pi_F|^{-1}$ if  $\partial = \partial / \partial \pi_F$ and $\leq 1$ otherwise.
By \cite[Theorem~4.2.2]{kedlaya-course}, for any $i=1, \dots, d$,  $|a_i\pi_F^{-ib} - \tilde a_i\pi_F^{-ib}| \leq|\pi_F|^{-i}$ if $\partial = \partial / \partial \pi_F$, and $\leq 1$ otherwise.   In particular, $a_i$ is congruent to $\tilde a_i$ modulo $\pi_F$ and hence $\bar a_i$ lies in $\kappa_\lambda\llbracket \lambda \rrbracket$; in other words, $v_\lambda(\bar a_i) \geq 0$.  This concludes the proof of the lemma.
\end{proof}


\begin{remark}
We also want to point out that this argument only applies to the submodule with the smallest $\partial$-radii.  
In general, one expects the ``product" of refined $\partial$-radii from the pieces with $l$ smallest $\partial$-radii for any $l$, to lie in $R^\alg$, if suitably normalized.
\end{remark}

\subsubsection{Refined irregularities over higher dimensional local fields}
Equip $\QQ^m$ with the lexicographic order: $\mathsf i = (i_1, \dots, i_m)< \mathsf j=  (j_1, \dots, j_m)$ if and only if 
\[
i_1 = j_1, \dots, i_{l-1} = j_{l-1}, \textrm{ and }
i_l < j_l  \textrm{ for some } l \in \{1, \dots, m\}.
\]

We will abuse the notation $\underline 0 = (0, \dots, 0)$ in various contexts, e.g. as elements in $\QQ^m$ or in $\bigoplus_{\omega \in \Log^*} x_1^b \kappa_\sfF^\alg \omega$.  This should not cause any confusion.

Let $\mathsf F = \mathbf k((x_m))\cdots((x_1))$ be the $m$-dimensional local field, where $\mathbf k$ is a trivially normed field.  The residue field $\kappa_\sfF$ of $\sfF$ is $\mathbf k((x_m))\cdots((x_2))$.
We then define a multi-indexed valuation $\mathsf v = (v_1, \dots, v_m): \sfF^\times \to \ZZ^m \subset \QQ^m$, where, for $\alpha \in \sfF^\times$, $v_1(\alpha)$ is the $x_1$-valuation of $\alpha$ and inductively, $v_i(\alpha)$ is the $x_i$-valuation of the reduction of $\alpha x_1^{-v_1(\alpha)} \cdots x_{i-1}^{-v_{i-1}(\alpha)}$ in $\mathbf k((x_m))\cdots ((x_i))$.   We denote $\mathsf O_\sfF =\{x \in \sfF| x=0 \textrm { or }\mathsf v(x) \geq \underline 0\}$ and $\mathsf M_\sfF =\{x \in F| x=0 \textrm{ or }\mathsf v(x) > \underline 0\}$.

We assume that $\mathbf k$  contains $k(x_{m+1}, \dots, x_n)$; and we assume that $\sfF$ admits continuous actions of differential operators $\partial_1 = \partial / \partial x_1, \dots, \partial_n = \partial / \partial x_n$.  When considering differential modules, the log-structure is given by $\Log =\{\partial_1, \dots, \partial_m\}$ and $\Log^* =\{\frac{dx_1}{x_1}, \dots, \frac{dx_m}{x_m}, dx_{m+1}, \dots, dx_n\}$.

For any $b \in \QQ$, the valuation $\mathsf v$ naturally gives rise to a valuation (still using the same notation) $\mathsf v: x_1^b\kappa_\sfF^\alg \bs \{ 0\}\to \QQ^m$; it induces a valuation $\mathsf v^\sharp: \big(\bigoplus_{\omega \in \Log^*}   x_1^b\kappa_\sfF^\alg \cdot \omega\big) \bs\{ 0\} \to \QQ^m$ given by
\[
\mathsf v^\sharp\Big(\theta_1 \frac{dx_1}{x_1}+ \cdots + \theta_m \frac{dx_n}{x_m}+ \theta_{m+1} dx_{m+1} + \cdots + \theta_n dx_n \Big):=\min \big\{ \mathsf v(\theta_1), \dots, \mathsf v(\theta_n) \big\}.
\]

If $V$ is a differential module over $\sfF$ of dimension $d$ with pure irregularity $\Irr(V)>0$, we define the \emph{multi-valuational irregularities} to be
\[
\mathsf{Irr}^\sharp(V) = \{-\mathsf v^\sharp(\Ref^\sharp(V; i))|i=1, \dots, \dim M\}.
\]
Let $\mathsf{Irr}^\sharp(V;1), \dots, \mathsf{Irr}^\sharp(V;\dim V)$ be elements of $\mathsf{Irr}^\sharp(V)$ in decreasing order.  In particular, the first argument of each $\mathsf{Irr}^\sharp(V; i)$ is simply $\Irr(V;i)$. \emph{Be aware that we may not have $-\mathsf v^\sharp(\Ref^\sharp(V;i)) = \mathsf{Irr}^\sharp(V;i)$ as there is no canonical order for $\Ref^\sharp(V;i)$ as pointed out in \ref{SS:refined-irr}.}
We also define the \emph{multi-valuational refined irregularities} to be
\[
\mathsf{Ref}^\sharp(V) = \big\{\big(-\mathsf v^\sharp(\Ref^\sharp(V; i)), \bar \vartheta_i\big)|\;i=1, \dots, \dim M\big\},
\]
where $\bar \vartheta_i$ is the reduction of $x_1^{v_1(\Ref^\sharp(V;i))}\cdots x_m^{v_m(\Ref^\sharp(V;i))} \Ref^\sharp(V;i)$ in $\bigoplus_{\omega \in \Log^*}\mathbf k^\alg \cdot \omega$.  We order the elements of $\mathsf{Ref}^\sharp(V)$ as $\mathsf{Ref}^\sharp(V;1), \dots, \mathsf{Ref}^\sharp(V; d)$, in decreasing order on the first argument.  (Again, note that a new order may be taken among all refined partially intrinsic radii.)

In general, $\mathsf{Irr}^\sharp(V)$ is the union of $\mathsf{Irr}^\sharp(V_i)$ for $V_i$ Jordan-H\"older factors of $V$. Here, those Jordan-H\"older factors $V_i$ with $\Irr(V_i)=0$ contribute $(0, \dots, 0)$ with multiplicity $\dim V_i$ to $\mathsf{Irr}^\sharp(V)$.

\begin{remark}
\label{R:Irr=0}
Careful readers may have noticed that, even when $\Irr(V)=0$, one can use Deligne-Malgrange lattice (see \cite[Section~2.4]{kedlaya-sabbah1}) to extract some information on the valuations $v_2, \dots, v_{m}$, and hence develop certain version of the multi-valuational irregularities this way.  However, we do not take this approach because (a) the information on the irregular pieces is what we need for our main theorem, and (b) this is a special phenomenon for residual characteristic zero case and we hope to keep parallel with the treatment in characteristic $p>0$ where Deligne-Malgrange lattices are not available.
\end{remark}

\subsection{Differential modules over $R_{n,m}$}

\subsubsection{Setup}
Let $k$ be a field of characteristic $0$. For $n\geq m\geq 0$, set
$$R_{n,m}:=k\llbracket x_1,\dots,x_n\rrbracket [x_1^{-1},\dots, x_m^{-1}].$$

For $\underline r=(r_1,\dots,r_n)\in[0,\infty)^n$, let $|\cdot|_{\underline r}$ denote the $(e^{-r_1},\dots,e^{-r_n})$-Gauss norm on $R_{n,m}$ and let $F_{\underline r}$ be the completion of $\Frac(R_{n,m})$ with respect to $|\cdot|_{\underline r}$.
Note that $F_{\underline r}$ is a complete nonarchimedean differential field of rational type with respect to $\partial_1=\partial /\partial x_1, \dots, \partial_n = \partial / \partial x_n$ with rational parameters $x_1,\dots,x_n$.

Let $S_{n,m}$ be the Fr\'echet completion of $R_{n,m}$ with respect to the norms $|\cdot|_{\underline r}$ for ${\underline r} \in (0, \infty)^n$.

Let $e_1, \dots, e_n$ be the standard base vectors  of $\mathbb R^n$.  For $j= 1, \dots, m$, we write $F_{(j)}$ and $|\cdot|_{(j)}$ for $F_{e_j}$ and $|\cdot|_{e_j}$, and write $\mathfrak o_{(j)}$ for $\gotho_{F_{(j)}}$.

\subsubsection{Differential modules}

Let $M$ be a \emph{differential module} over $R_{n,m}$, that is a locally free module $M$ over $R_{n,m}$ of finite rank $d$, with an integrable connection
$\nabla: M\to M\otimes \Omega^1_{R_{n,m}/k},$  i.e. with commuting actions of $\partial_1=\partial/\partial x_1,\dots,\partial_n=
\partial/\partial x_n $ (subject to the Leibniz rule).

For the first part of this subsection, we consider intrinsic radii with respect to the \emph{full log-structure} $\Log' = \{\partial_1, \dots, \partial_n\}$; in this case we omit the superscript $\sharp$ and write $IR(M \otimes F_{\underline r};i)$ for the subsidiary intrinsic radii.  We put $g_i(M,{\underline r})=-\log (IR(M \otimes F_{\underline r}; i))$ and  $G_i(M,{\underline r})=g_1(M,{\underline r})+\cdots+ g_i(M,{\underline r})$ for $i=1, \dots, d$.

\begin{remark}
For any $\lambda \in \RR_{>0}$, $|\cdot|_{\lambda {\underline r}} = |\cdot|_{\underline r}^\lambda$ and hence $F_{\lambda {\underline r}}$ is isomorphic to $F_{\underline r}$ only with a different norm.  It follows that $g_i(M, \lambda {\underline r}) = \lambda  g_i(M, {\underline r})$ for any $i=1, \dots, d$.  Also, under our convention, $F_{(0, \dots, 0)}$ is trivially normed and hence $g_i(M, (0, \dots, 0))  =0$ for all $i = 1, \dots, d$.
\end{remark}

\begin{notation}
For positive integers $\underline h = (h_1, \dots, h_m) \in \NN^m$ and a finite extension $k'$ of $k$, we put 
\[
R'_{n,m,\underline h} = k'\llbracket x_1^{1/h_1}, \dots, x_m^{1/h_m}, x_{m+1}, \dots, x_n\rrbracket [x_1^{-1}, \dots, x_m^{-1}].
\]  For $\underline r \in [0, \infty)^n$, we write $F'_{\underline r, \underline h} = k'F_{\underline r}[x_1^{1/h_1}, \dots, x_m^{1/h_m}]$ for the completion of $\Frac(R'_{n,m,\underline h})$ with respect to $|\cdot|_{\underline r}$ (extended to $F'_{\underline r, \underline h}$.
\end{notation}

\begin{theorem}\label{T:variation}
Let $M$ be a non-zero differential module of rank $d$ over $R_{n,m}$.
We have the following properties.

\emph{(i) (Variation)}
The functions $G_i(M,{\underline r})$ are continuous, convex, and piecewise linear for all $\underline r \in [0, \infty)^n$.  Moreover, if ${\underline r}, \underline r' \in (0, \infty)^n$ with $r_j= r'_j$ for $j = 1, \dots, m$ and $r_j \leq r'_j$ for $j = m+1, \dots, n$, then $G_i(M, {\underline r}) \leq G_i(M, \underline r')$.

\emph{(ii) (Decomposition)} Fix $l \in \{1, \dots, d-1\}$.  Suppose that the function $G_l(M,{\underline r})$ is linear, and $g_l(M,\underline  r) > g_{l+1}(M, {\underline r})$ for all ${\underline r} \in (0, \infty)^n$.
Then $M$ admits a unique direct sum decomposition $M_1\oplus M_2$ such that for each ${\underline r} \in (0, \infty)^n$,
$\calI\calR(M_1 \otimes F_{\underline r})$ consists of the smallest $l$ elements of $\calI\calR(M \otimes F_{\underline r})$.

\emph{(iii) (Refined intrinsic radii decomposition)}
Assume that $g_1(M, {\underline r}) = \cdots = g_d(M, \underline r) = b_1 r_1 + \cdots +b_mr_m$ are affine functions over $(0,\infty)^n$ (in particular, it is constant in the last $n-m$ coordinates).  Let $h_i$ denote the denominator of $b_i$ for all $i$.   Then there exists a finite extension $k'$ of $k$ and a multiset $\Ref'(M) \subset \oplus_{i=1}^n k' \frac {dx_i}{x_i}$ such that we have a unique  direct sum  decomposition of differential modules:
\[
M \otimes_{R_{n,m}} R'_{n,m,\underline h} = \bigoplus_{\vartheta \in \Ref'(M)} M_\vartheta,
\]
such that $M_\vartheta \otimes F'_{\underline r, \underline h}$ has pure refined intrinsic radius $x_1^{-b_1}\cdots x_m^{-b_m}\vartheta$ for all ${\underline r} \in (0, \infty)^n$.
\end{theorem}
\begin{proof}
(i) is proved in \cite[Theorem~3.3.9]{kedlaya-xiao} and (ii) is proved in \cite[Theorem~3.3.6]{kedlaya-sabbah1}.
We now prove (iii).  There is nothing to prove when $m=0$, so we assume hereafter that $m>0$.  Also, we may replace $x_j$ by $x_j^{1/h_j}$ and $k$ by $k'$ and assume that $h_j =1$ and $k=k'$ for all $j = 1, \dots, m$.  By \cite[Theorem~4.3.6]{xiao-refined}, we have a decomposition of $M$ over $S_{n,m}$ satisfying the required property; this corresponds to a projector $\bbe \in \End(M) \otimes_{R_{n,m}} S_{n,m}$.  By the following lemma, we deduce that this projector $\bbe$ in fact lives in $\End(M)$, yielding (iii).
\end{proof}

\begin{lemma}
\label{L:Kiran's decompletion}
Let $M$ be a differential module over $R_{n,m}$.  Suppose that $v$ is a vector of $M \otimes_{R_{n,m}} S_{n,m}$ such that $\partial_i(v) = 0$ for every $i$.  Then $v \in M$.
\end{lemma}

\begin{proof}
This is exactly what was proved in \cite[Theorem~3.3.6]{kedlaya-sabbah1}, although {\it loc. cit.} is stated in a different form with our $M$ being the $\End(M)$ therein.
\end{proof}

\begin{definition}
For a non-zero differential module $M$ of over $R_{n,m}$, we say that $M$ is \emph{numerically clean} if the functions $g_i(M, \underline r)$ are linear in $\underline r$ for all $i$.
\end{definition}

\begin{remark}\label{R:clean-nonstable}
The numerical condition is preserved under taking subobjects but is \emph{not} stable under taking direct sums, because the functions $g_i$ from different direct summands may not be well-ordered. However, one can develop an explicit recipe to make toroidal blow-ups at the intersection of irreducible components so that, pulling back to this blowup, $M$ becomes numerically clean. 
\end{remark}

In \cite{kedlaya-sabbah1}, Kedlaya introduced a condition which is slightly stronger than numerical cleanliness.  It has the advantage of being equivalent to $M$ having an explicit form (locally), and is stable under proper birational base change.

\begin{definition}
Let $M$ be a  differential module over $R_{n,m}$. We say $M$ is \emph{regular} if $M=0$ or $g_1(M,\underline r)\equiv 0$.  By Theorem~\ref{T:variation}(i), this is equivalent to $IR(M \otimes F_{(j)};1 ) = 1$ for all $j=1, \dots, m$.
\end{definition}

\begin{definition}
\label{D:good-decomposition}
For $\phi \in R_{n,m}$, we define a differential module $E(\phi)$ over $R_{n,m}$ of rank 1 with generator $\mathbf e$ by
\[
\partial_i \mathbf e = \partial_i(\phi) \mathbf{e}, \textrm{ for } i=1,\dots, n.
\]

Let $M$ be a  differential module over $R_{n,m}$. A \emph{good decomposition} is an isomorphism
$$M\simeq\bigoplus_{\alpha\in A} E(\phi_\alpha)\otimes_{R_{n,m}} \mathrm{Reg}_\alpha$$
for some $\phi_\alpha\in R_{n,m}$ and some regular differential modules $\mathrm{Reg}_\alpha$, satisfying the following two conditions:

(1) For $\alpha\in A$, if $\phi_\alpha \notin R_{n,0}$, then $\phi_\alpha=ux_1^{-i_1}\cdots x_m^{-i_m}$, for some unit $u\in R_{n,0}^\times$
and some nonnegative integers $i_1,\dots,i_m$.

(2) For $\alpha,\beta\in A$, if $\phi_\alpha-\phi_\beta \notin R_{n,0}$, then $\phi_\alpha-\phi_\beta=ux_1^{-i_1}\cdots x_m^{-i_m}$, for some unit $u\in R_{n,0}^\times$
and some nonnegative integers $i_1,\dots,i_m$.
\end{definition}

\begin{theorem}
\label{T:kedlaya-criterion}
Let $M$ be a non-zero differential modules over $R_{n,m}$ of rank $d$. The following conditions are equivalent:

(1) There exist a finite extension $k'$ and a positive integer $h$ such that $M\otimes_{R_{n,m}}R_{n,m}[x_1^{1/h}, \dots, x_m^{1/h}]$ admits a good decomposition.

(2) The functions $G_1(M,r),\cdots, G_d(M,r)$ and $G_{d^2}(M\otimes M^\vee, r)$ are linear in $r$.

(3) The functions $G_d(M,r)$ and $G_{d^2}(M\otimes M^\vee, r)$ are linear in $r$.
\end{theorem}
\begin{proof}
See \cite[Theorem~4.4.2]{kedlaya-sabbah1}.
\end{proof}

\begin{definition}
We say that $M$ admits a \emph{good formal structure} at $x$ if $M \otimes R_{n,m}$ satisfies the equivalent conditions in Theorem~\ref{T:kedlaya-criterion}.
\end{definition}

\begin{remark}
We have implications
(i) $M$ and $\End(M) = M \otimes M^\vee$ being numerically clean $\Rightarrow$
(ii) $M$ admitting good formal structure $\Rightarrow$ (iii) $M$ being numerically clean.  The first implication is not an equivalence; see \cite[Example~4.4.5]{kedlaya-sabbah1} for a counterexample.  This failure is very similar to the instability of cleanliness under taking direct sums, as explained in Remark~\ref{R:clean-nonstable}.
\end{remark}

\subsubsection{A different log-structure}
For the rest of this subsection, we discuss the situation with a different choice of log-structure: $\Log = \{\partial_1, \dots, \partial_m\}$.  We write $\Log^* =\{\frac{dx_1}{x_1}, \dots, \frac{dx_m}{x_m}, dx_{m+1}, \dots, dx_n\}$.
For a differential module $M$ over $R_{n,m}$, we set $g_i^\sharp(M,\underline r)=-\log( IR^\sharp(M \otimes F_{\underline r}; i))$ and  $G_i^\sharp(M,{\underline r})=g_1^\sharp(M,{\underline r})+\cdots+ g_i^\sharp(M,{\underline r})$ for $\underline r\in [0,\infty)^n$ and $i=1, \dots, d$.  

Similarly, we have $g_i^\sharp(M, \lambda {\underline r}) = \lambda  g_i^\sharp(M, {\underline r})$ and $ g_i^\sharp(M, (0, \dots, 0)) =0$ for any $\lambda \in \RR_{>0}$ and all $i = 1, \dots, d$.

Let $\bbk = \Frac(k\llbracket x_{m+1}, \dots, x_n\rrbracket)$.  Some of our decomposition theorems will only work over $\bbR_{n,m} = \bbk\llbracket x_1, \dots, x_m\rrbracket [x_1^{-1}, \dots, x_m^{-1}]$.  For $\underline r \in [0, \infty)^m \times \{0\}^{n-m}$, the norm $|\cdot|_{\underline r}$ also extends to $\Frac(\bbR_{n,m})$; let $\bbF_{\underline r}$ denote the completion.  Each $\bbF_{\underline r}$ contains $F_{\underline r}$ as a subfield and it actually equals to $F_{\underline r}$ if $\underline r \in (0, \infty)^m \times \{0\}^{n-m}$.  

For a differential module $\bbM$  over $\bbR_{n,m}$, we similarly define $g_i^\sharp(\bbM,\underline r)=-\log( IR^\sharp(\bbM \otimes \bbF_{\underline r}; i))$ and   $G_i^\sharp(\bbM,{\underline r})=g_1^\sharp(\bbM,{\underline r})+\cdots+ g_i^\sharp(\bbM,{\underline r})$ for any $\underline r \in [0, \infty)^m \times \{0\}^{n-m}$ and $i=1, \dots, d$.  If $\bbM = M \otimes_{R_{n,m}}\bbR_{n,m}$ is the base change of a differential module $M$ over $R_{n,m}$, we have $g_i^\sharp(\bbM,\underline r) = g_i^\sharp(M,\underline r)$ for any $\underline r \in [0, \infty)^m \times \{0\}^{n-m}$.

\begin{prop}
\label{P:weak-decomposition}
Let $\bbM$ be a nonzero differential module of rank $d$ over $\bbR_{n,m}$.  We have the following properties.

\emph{(i${}^\sharp$) (Variation)} The functions $G^\sharp_i(\bbM,{\underline r})$ are continuous, convex, and piecewise linear for all $\underline r \in [0, \infty)^m \times \{0\}^{n-m}$.

\emph{(ii${}^\sharp$) (Weak decomposition)}    Fix $l \in \{1, \dots, d-1\}$.  Suppose that the function $G_l^\sharp(\bbM, \underline r)$ is linear, and $g_l^\sharp(\bbM, \underline r)> g^\sharp_{l+1}(\bbM, \underline r)$ for all $\underline r \in (0, \infty)^m \times \{0\}^{n-m}$.  Then $\bbM$ admits a unique direct sum decomposition $\bbM_1\oplus \bbM_2$ such that, for each $\underline r \in (0, \infty)^m \times \{0\}^{n-m}$,
$\calI\calR^\sharp(\bbM_1 \otimes \bbF_{\underline r})$ consists of the smallest $l$ elements of $\calI\calR^\sharp(\bbM \otimes \bbF_{\underline r})$.

\emph{(iii${}^\sharp$) (Weak refined intrinsic radii decomposition)}
Assume that $g_1^\sharp(\bbM, {\underline r}) = \cdots = g_d^\sharp(\bbM, \underline r) = b_1 r_1 + \cdots +b_mr_m$ are affine functions over $[0,\infty)^m \times \{0\}^{n-m}$.  Let $h_i$ denote the denominator of $b_i$ for all $i$.   Then we have a unique  direct sum decomposition of differential modules
\[
\bbM \otimes_{\bbR_{n,m}} \bbR_{n,m}[x_1^{1/h_1}, \dots, x_m^{1/h_m}] = \bigoplus_{\{G_\bbk\vartheta\}} \bbM_{\{G_\bbk\vartheta\}},
\]
where the direct sum is taken over all $G_\bbk = \Gal(\bbk^\alg/\bbk)$-orbits of elements in $\bigoplus_{\omega \in \Log^*} \bbk^\alg \cdot \omega$, such that, for all ${\underline r} \in (0, \infty)^m \times \{0\}^{n-m}$, the refined intrinsic radii of $\bbM_{\{G_\bbk\vartheta\}} \otimes \bbF_{\underline r}[x_1^{1/h_1}, \dots, x_m^{1/h_m}]$ consists of $x_1^{-b_1}\cdots x_m^{-b_m}g(\vartheta)$ for $g \in G_\bbk$ with the same multiplicity on each element.
\end{prop}
\begin{proof}
(i${}^\sharp$) follows from \cite[Theorem~3.3.9]{kedlaya-xiao} applied to $\bbM$.
For (ii${}^\sharp$) and (iii${}^\sharp$), let $\bbS_{n,m}$ be the Fr\'echet completion of $\bbR_{n,m}$ with respect to the norms $|\cdot|_{\underline r}$ for ${\underline r} \in (0, \infty)^m \times \{0\}^{n-m}$.  (In case (iii${}^\sharp$), we may first replace $x_j$ by $x_j^{1/h_j}$ and $k$ by $k'$ and assume that $h_j =1$ and $k=k'$.)
In these two cases, we invoke \cite[Theorem~3.4.2]{kedlaya-xiao} and \cite[Theorem~4.3.6]{xiao-refined}, respectively, to obtain the desired decomposition over $\bbS_{n,m}$.  Each direct summand corresponds to a projector $\bbe \in \End(\bbM) \otimes_{\bbR_{n,m}} \bbS_{n,m}$.  Then, we
apply Lemma~\ref{L:Kiran's decompletion} (forgetting the actions of $\partial_{m+1}, \dots, \partial_n$) to deduce that this projector $\bbe$ in fact lives in $\End(\bbM)$, yielding (ii${}^\sharp$) and (iii${}^\sharp$).
\end{proof}

Under a stronger hypothesis, we can extend the decomposition to differential modules over $R_{m,n}$.
\begin{theorem}
\label{T:strong-decomposition}
Let $M$ be a nonzero differential module of rank $d$ over $R_{n,m}$.
Fix $l \in \{1, \dots, d-1\}$.  Assume that $g_1^\sharp(M, {\underline r}) = \cdots = g_l^\sharp(M, r) = b_1 r_1 + \cdots +b_mr_m$ are affine functions over $(0,\infty)^m \times \{0\}^{n-m}$ and assume that $g_l^\sharp(M, \underline r) > g_{l+1}^\sharp(M, \underline r)$ for $\underline r  \in (0,\infty)^m \times \{0\}^{n-m}$.  Let $h_i$ denote the denominator of $b_i$ for all $i$.  We have the following.

(i) There exists a (complete) local ring $\mathfrak R$ finite over $k\llbracket x_{m+1}, \dots, x_n\rrbracket$ such that for all ${\underline r} \in (0, \infty)^m \times \{0\}^{n-m}$,  
\[
\Ref^\sharp(M \otimes F_{\underline r}; 1), \dots, \Ref^\sharp(M \otimes F_{\underline r}; l) \in \bigoplus_{\omega\in \Log^*}x_1^{-b_1}\cdots x_m^{-b_m} \gothR\cdot \omega
\]

(ii) Let $\gothm_{\gothR}$ denote the maximal ideal of $\gothR$ and write $k' = \gothR / \gothm_\gothR$ which is a finite extension of $k$. 
We have a unique direct sum decomposition of differential modules:
\[
M \otimes_{R_{n,m}} R'_{n,m,\underline h} = M_0 \oplus \bigoplus_{\underline \lambda} M_{\underline \lambda},
\]
where the direct sum runs through all elements $\underline \lambda =(\lambda_\omega)_{\omega \in \Log^*}\in k'^n\bs \{\underline 0\}$, such that
\begin{itemize}
\item for any $\underline r \in (0, \infty)^m \times \{0\}^{n-m}$, all elements in $x_1^{b_1} \cdots x_m^{b_m} \Ref^\sharp(M_{\underline \lambda} \otimes F'_{\underline r, \underline h})$ is congruent to $\sum_{\omega \in \Log^*} \lambda_\omega \omega$ modulo $\gothm_\gothR$, and
\item for any $\underline r \in (0, \infty)^m \times \{0\}^{n-m}$, any Jordan-H\"older factor of $M_0 \otimes F'_{\underline r, \underline h}$ either has partially intrinsic radii $>b_1r_1+ \cdots + b_mr_m$, or has the refined partially intrinsic radii lie in $\bigoplus_{\omega \in \Log^*} x_1^{-b_1} \cdots x_m^{-b_m}\gothm_\gothR \cdot \omega$.
\end{itemize}
\end{theorem}
\begin{proof}
(i)  By Proposition~\ref{P:weak-decomposition}(ii${}^\sharp$), we can first separate a differential submodule $\bbM_1$ of $M \otimes \bbR_{n,m}$ that accounts for the $l$ smallest partially intrinsic radii of $M \otimes F_{\underline r}$ with $\underline r \in [0, \infty)^m \times \{0\}^{n-m}$.  Then we apply Proposition~\ref{P:weak-decomposition}(iii${}^\sharp$) to $\bbM_1$ and conclude that, for any $\underline r \in (0, \infty)^m \times \{0\}^{n-m}$,
\begin{equation}
\label{E:weak-refined}
\Ref^\sharp(M \otimes F_{\underline r}; 1), \dots, \Ref^\sharp(M \otimes F_{\underline r}; l) \in \bigoplus_{\omega\in \Log^*}x_1^{-b_1}\cdots x_m^{-b_m} \Frac(k\llbracket x_{m+1}, \dots, x_n\rrbracket)^\alg\cdot \omega
\end{equation}
We now look at the point $\underline \bbr = (\bbr_1, \dots, \bbr_n)$ with $\bbr_1 = \cdots = \bbr_m = 1$ and $\bbr_{m+1} = \cdots  = \bbr_n=0$.
We apply Lemma~\ref{L:refined-integral} to the case when the ring $R$ therein is equal to
\[
k\big \llbracket \frac{x_2}{x_1}, \dots, \frac{x_m}{x_1}, x_{m+1}, \dots, x_n\big \rrbracket \big[\big(\frac{x_2}{x_1}\big) ^{-1}, \dots, \big(\frac{x_m}{x_1}\big)^{-1}]
\]
and $\pi_F = x_1$; this implies that, 
the coefficient (on every $\omega$) of each $x_1^{b_1} \cdots x_m^{b_m}\Ref^\sharp(M \otimes F_{\underline r}; i)$ is a zero of an irreducible polynomial
of the form $X^{d'} + a_1 X^{d'-1} + \cdots + a_{d'}$, where $a_j \in R$.
But the coefficient on $\omega$ of $x_1^{b_1} \cdots x_m^{b_m}\Ref^\sharp(M \otimes F_{\underline r}; i)$ belongs to $\Frac(k\llbracket x_{m+1}, \dots, x_n\rrbracket)^\alg$.
This means that $a_j \in k\llbracket x_{m+1}, \dots, x_n\rrbracket$, and hence the coefficient belongs to a complete local ring $\gothR$ finite over $k\llbracket x_{m+1}, \dots, x_n\rrbracket$. 

(ii)  We may replace $x_j$ by $x_j^{1/h_j}$ and $k$ by $k'$ and hence assume that $h_j=1$ and $k=k'$.  By Proposition~\ref{P:weak-decomposition}(ii${}^\sharp$) and (iii${}^\sharp$), we have the desired decomposition over $\bbR_{n,m}$, that is $M \otimes \bbR_{n,m} =  \bbM_0\oplus \bigoplus_{\underline \lambda} \bbM_{\underline \lambda} ,
$
where the direct sum runs through all $\underline \lambda =(\lambda_\omega)_{\omega \in \Log^*}\in k^n\bs \{\underline 0\}$, such that
\begin{itemize}
\item for any $\underline r \in (0, \infty)^m \times \{0\}^{n-m}$, all elements in $x_1^{b_1} \cdots x_m^{b_m} \Ref^\sharp(\bbM_{\underline \lambda} \otimes \bbF_{\underline r})$ is congruent to $\sum_{\omega \in \Log^*} \lambda_\omega \omega$ modulo $\gothm_\gothR$, and
\item for any $\underline r \in (0, \infty)^m \times \{0\}^{n-m}$, any Jordan-H\"older factor of $\bbM_0 \otimes \bbF_{\underline r}$ either has partially intrinsic radii $>b_1r_1+ \cdots + b_mr_m$, or has the refined partially intrinsic radii lie in $\bigoplus_{\omega \in \Log^*} x_1^{-b_1} \cdots x_m^{-b_m}\gothm_\gothR \cdot \omega$.
\end{itemize}
(In fact, Proposition~\ref{P:weak-decomposition} provides us with a much finer decomposition; we regroup the summand into $\bbM_0$ and $\bbM_{\underline \lambda}$'s accordingly.)

We need to make this decomposition descend to $R_{n,m}$ by ``gluing" with some other decompositions using Lemma~\ref{L:proj-intersect} and Remark~\ref{R:proj-intersect} below.

We look at the point $\underline \bbr = (\bbr_1, \dots, \bbr_n)$ with $\bbr_1 = \cdots = \bbr_m = 1$ and $\bbr_{m+1} = \cdots  = \bbr_n=0$. For each $j =m+1, \dots, n$, we use $F_{\underline \bbr, (j)}$ to denote the completion of $\Frac(k\llbracket x_1, \dots, \hat x_j, \dots, x_n\rrbracket)$ with respect to the Gauss norm $|\cdot|_{\underline \bbr}$ (restricted to this subfield with no variable $x_j$).  Let $F'_{\underline \bbr, (j)}$ denote the completion of $F_{\underline \bbr, (j)}(t)$ with respect to the $\mathrm e$-Gauss norm (on $t$); we set $\partial_{j'}(t) = 0$ for $j' = 1, \dots, n$.  We consider a new set of differential operators $\partial'_1 = t^{-1}\partial_1, \dots, \partial'_m = t^{-1}\partial_m, \partial'_{m+1}=\partial_m, \dots, \partial'_n = \partial_n$; they are of rational type with rational parameters $tx_1, \dots, tx_m, x_{m+1}, \dots, x_n$, when viewed as differential operators on the completion of $\Frac(F'_{\underline \bbr, (j)}\llbracket x_j\rrbracket_0)$ for any Gauss norm (on $x_j$). (For $\llbracket \cdot \rrbracket_0$, see Notation~\ref{N:bounded-disc}.)

We consider the differential module $M \otimes F'_{\underline \bbr, (j)}\llbracket x_j \rrbracket_0$ (with respect to $\partial'_1, \dots, \partial'_n$) and take the \emph{trivial} log-structure $\Log'=\emptyset$.  Let $E_j$ denote the completion of $\Frac(F'_{\underline \bbr, (j)} \llbracket x_j \rrbracket_0)$ with respect to the $1$-Gauss norm on $x_j$.  For a differential module $N$ over $F'_{\underline \bbr, (j)}\llbracket x_j\rrbracket_0$, we write $IR'(N \otimes E_j; i)$ for the subsidiary partially intrinsic radii and $\Ref'(N \otimes E_j;i)$ for the refined partially intrinsic radii, with respect to this new choice of differential operators and log-structure. We have $IR^\sharp(N \otimes E_j; i) = IR'(N \otimes E_j; i)$ and $\Ref^\sharp(N \otimes E_j;i) = \Ref'(N \otimes E_j;i)$ for any $i$, if we identify $\bigoplus_{\omega \in \Log^*} \kappa_{E_j^\alg}^{[s]} \cdot \omega$ with $\bigoplus_{i=1}^m \kappa_{E_j^\alg}^{[s]} \cdot tdx_i \oplus \bigoplus_{i=m+1}^n \kappa_{E_j^\alg}^{[s]} \cdot dx_i$ for any $s$.

We apply \cite[Corollary~4.2.9]{xiao-refined} (the $K$ therein is our $F'_{\underline \bbr, (j)}$ and the $E$ therein is the completion of $\Frac(F'_{\underline \bbr, (j)}\{\{\alpha/x_j, x_j \rrbracket_0)$, which is the same as our $E_j$ because $\Frac(F'_{\underline \bbr, (j)}\llbracket x_j \rrbracket_0)$ is dense in the former fraction field) to obtain a unique direct sum decomposition
\[
M \otimes F'_{\underline \bbr, (j)}\llbracket x_j \rrbracket_0 =
\bigoplus M'_{\underline \lambda, (j)}
\oplus M'_{0,(j)} 
\]
satisfying analogous conditions as in the statement of the theorem, namely, 
\begin{itemize}
\item the direct sum runs through all $\underline \lambda =(\lambda_\omega)_{\omega \in \Log^*}\in k^n\bs \{\underline 0\}$,

\item  all elements in $x_1^{b_1} \cdots x_m^{b_m} \Ref^\sharp(M'_{\underline \lambda,(j)} \otimes E_j)$ is congruent to $\sum_{\omega \in \Log^*} \lambda_\omega \omega$ modulo $\gothm_\gothR$, and

\item  any Jordan-H\"older factor of $M'_{0,(j)} \otimes E_j$ has partially intrinsic radii $>b_1+ \cdots + b_m$ or the refined partially intrinsic radii lie in $\bigoplus_{\omega \in \Log^*} x_1^{-b_1} \cdots x_m^{-b_m}\gothm_\gothR \cdot \omega$.
\end{itemize}
Strictly speaking, \cite[Corollary~4.2.9]{xiao-refined} gives a much finer decomposition and we regrouped them together (mostly into $M'_{0,(j)}$).

For a fixed $j$, the decomposition over $F'_{\underline \bbr, (j)}\llbracket x_j \rrbracket_0$ agrees with the decomposition over $\bbR_{n,m}$ in the sense that they both induce the same decomposition of $M$ over $E_j$ (given by Proposition~\ref{P:decomposition-field}(ii)(iii)).  More precisely, we have
\[
M'_{\underline \lambda, (j)} \otimes E_j\cong \bbM_{\underline \lambda} \otimes E_j \textrm { and } M'_{0,(j)}\otimes E_j \cong \bbM_0 \otimes E_j
\]
inside $M \otimes E_j$.

By Lemma~\ref{L:proj-intersect} and Remark~\ref{R:proj-intersect} below, we conclude that the decomposition of $M$ descends to 
\begin{equation}
\label{E:intersection-rings}
F'_{\underline \bbr, (j)}\llbracket x_j \rrbracket_0 \cap \bbR_{n,m} = \Frac(k\llbracket x_{m+1}, \dots, \hat x_j, \dots, x_n\rrbracket) \llbracket x_1, \dots, x_m, x_j\rrbracket[x_1^{-1}, \dots, x_m^{-1}]
\end{equation}
Applying
Lemma~\ref{L:proj-intersect} and Remark~\ref{R:proj-intersect} again to glue the decompositions over \eqref{E:intersection-rings} for each $j$ (along the decomposition over $\bbR_{n,m}$), we obtain the desired decomposition over $R_{n,m}$.

We remark that it is attempting to try to glue the decompositions of $M \otimes F'_{\underline \bbr, (j_1)}\llbracket x_{j_1}\rrbracket_0$ and $M \otimes F'_{\underline \bbr, (j_2)}\llbracket x_{j_2}\rrbracket_0$ ($j_1 \neq j_2$) directly.  But there is no ring containing both base rings $F'_{\underline \bbr, (j_1)}\llbracket x_{j_1}\rrbracket_0$ and $F'_{\underline \bbr, (j_2)}\llbracket x_{j_2}\rrbracket_0$.  We do need $\bbR_{n,m}$ to hold these $F'_{\underline \bbr, (j)}$'s together.
\end{proof}

\begin{lemma}
\label{L:proj-intersect}
Let 
\[
\xymatrix{
R \ar[r] \ar[d] &  S \ar[d] \\
T \ar[r] & U
}
\]
be a commuting diagram of \emph{inclusions} of integral domains,
such that the intersection
$S \cap T$ within $U$ is equal to $R$. Let $M$ be a finite
locally free $R$-module. Then the intersection of $M \otimes_R S$
and $M \otimes_R T$ within $M \otimes_R U$ is equal to $M$.
\end{lemma}
\begin{proof}
This is \cite[Lemma~2.3.1]{kedlaya-xiao}.
\end{proof}

\begin{remark}
\label{R:proj-intersect}
We remark on how this lemma is used.  We often apply this lemma to the module $\mathrm{End}(M)$ over $R$ for a differential module $M$.  When we have a ``desired" direct sum decomposition of $M \otimes_R S$ and $M \otimes_R T$ which coincide on 
$M \otimes_R U$, we view the projectors of the decomposition as elements in $\mathrm{End}(M) \otimes_R S$, $\mathrm{End}(M) \otimes_R T$, and $\mathrm{End}(M) \otimes_R U$, respectively.  By Lemma~\ref{L:proj-intersect}, we see that the projectors giving this direct sum decomposition actually comes from $\End(M)$.  Hence we may ``glue" the direct sum decomposition of $M \otimes_R S$ and $M \otimes_R T$ to get a direct sum decomposition of $M$ (over $R$).
\end{remark}

\subsection{Cleanliness condition}
We give the definition of cleanliness using refined irregularities developed in \cite{xiao-refined}.  We prove that the numerical cleanliness implies this cleanliness.  We keep the notation as in the previous subsection.

\begin{definition}
Let $\mathfrak S_m$ denote the permutation group of $\{1, \dots, m\}$.  Given  $\sigma \in \mathfrak S_m$, we have a natural embedding $\iota_\sigma: R_{n,m} \to \sfF_\sigma = \mathbf k((x_{\sigma(m)}))\cdots ((x_{\sigma(1)}))$, where $\sfF_\sigma$ is a $m$-dimensional local field with its norm given by the $x_{\sigma(1)}$-valuation such that $|x_{\sigma(1)}| = \mathrm e^{-1}$.  So $\sfF_\sigma$ contains $F_{(\sigma(1))}$ as a subfield.
\end{definition}

\begin{prop}
\label{P:variation-vs-refined}
Let $M$ be a differential module over $R_{n,m}$ of rank $d$.  Then there exists $\epsilon_1, \dots, \epsilon_m >0$ such that, if we let $C$ denote the interior of the convex hull of the set of points
\[
\big\{
(1,0, \dots, 0),
(1+\epsilon_1, 0,\dots, 0),(1+\epsilon_1, \epsilon_2, \dots, 0, 0), \dots, 
(1+\epsilon_1, \epsilon_2, \dots, \epsilon_{m}, 0, \dots, 0)
\big\},
\]
then we have a unique  direct sum  decomposition of differential modules $M \otimes \bbR_C^\bdd = \bbM_0 \oplus \bigoplus_{\underline b \in \QQ^n, b_1>0} \bbM_{\underline b}$ such that,
\begin{align}
\label{E:variation-vs-refined}
\textrm{for }\underline b \in \QQ^n \textrm{ with } b_1>0,& \quad
g_1^\sharp(\bbM_{\underline b}, \underline r) = \cdots= g_{\dim \bbM_{\underline b}}^\sharp (\bbM_{\underline b}, \underline r) = b_1r_1+ \cdots +b_mr_m, \textrm{ for all }\underline r \in C;\\
\label{E:variation-vs-refined1}
\textrm{for }\bbM_0, & \quad g_1^\sharp(\bbM_0, (1, 0, \dots, 0)) = \cdots= g_{\dim \bbM_0}^\sharp (\bbM_0, (1, 0, \dots, 0)) = 0.
\end{align}
where $\bbR_C$ is the Fr\'echet completion of $\bbR_{n,m}$ with respect to the norms $|\cdot|_{\underline r}$ for $\underline r \in C$ and $\bbR_C^\bdd \subset \bbR_C$ is the subring consisting of elements whose norms are bounded for $|\cdot|_{\underline r}$ for all ${\underline r} \in C$.

Moreover, the multiset $\mathsf{Ref}^\sharp(M \otimes \sfF_\id)$ is the same as
\[
\bigcup_{\underline b, b_1>0}\{(b_1, \dots, b_m, \vartheta_{\underline b, i})| i = 1, \dots, \rank (\bbM_{\underline b}) \} \cup\{(0, \dots, 0)\ \rank (\bbM_0) \textrm{ times}\},
\]
 where $\vartheta_{\underline b, i}$ is the reduction of $x_1^{b_1} \cdots x_m^{b_m} \Ref(\bbM_{\underline b} \otimes \bbF_{\underline r}; i)$ in $\bbk^\alg$ for (any) $\underline r \in C$.

\end{prop}
\begin{proof}
To get the decomposition, one first uses Proposition~\ref{P:weak-decomposition}(i${}^\sharp$) to show that the functions $g_i^\sharp(M,\underline r)$ are linear when restricted to some small enough $C$ and then invokes \cite[Theorem~3.4.4 and Remark~3.4.7]{kedlaya-xiao} to obtain the decomposition.  The equality of two multisets is proved in \cite[Theorem~4.3.8]{xiao-refined}.  Note that when $b_1=0$, $\mathsf{Ref} ^\sharp(M \otimes \sfF_\id)$ will always give $(0,\dots, 0)$. (See also Remark~\ref{R:Irr=0}.)
\end{proof}

\begin{theorem}
\label{T:clean-equiv-local}
Let $M$ be a differential module over $R_{n,m}$ of rank $d$. 
Put $b_{ij} = -\log( IR(M \otimes F_{(j)}; i))$.
The following conditions are equivalent:

(1) The functions $g_1^\sharp(M,\underline r),\dots, g_d^\sharp(M,\underline r)$ are linear in $\underline r \in [0, \infty)^m \times \{0\}^{n-m}$. If we write $g_i^\sharp(M, \underline r) = b_{i1}r_1+ \cdots + b_{im}r_m$ for any $i$ and any $\underline r \in [0, \infty)^m \times \{0\}^{n-m}$, then, for any $\underline r \in (0, \infty)^m \times \{0\}^{n-m}$, we have $x_1^{b_{i1}}\cdots x_m^{b_{im}}\Ref^\sharp(M \otimes F_{\underline r}; i) \in \bigoplus_{\omega \in \Log^*} \gothR\cdot \omega$ for some local ring $(\gothR, \gothm_\gothR)$ finite over $k\llbracket x_{m+1}, \dots, x_n\rrbracket$, and its reduction modulo $\gothm_\gothR$ is nontrivial. 

(2) For any $j \in \{1, \dots, m\}$ satisfying
\begin{equation}
\label{E:j-ramifies}
\min\{i|b_{ij}=0\} \geq \min\{i|b_{ij'}=0\} \textrm{ for any }j \neq j',
\end{equation}
we have, for some $\sigma \in \mathfrak S_m$  with $\sigma(1)=j$, $\mathsf{Ref}^\sharp(M \otimes_{R_{n,m}} \sfF_\sigma; i)$ for each $i$ is of the form
\[
(b_{i\sigma(1)}, \dots, b_{i\sigma(m)} , \vartheta_i),
\]
where $\vartheta_i$ is some element in $\bigoplus_{\omega \in \Log^*} \gothR\cdot \omega$ for some local ring $\gothR$ finite over $k\llbracket x_{m+1}, \dots, x_n\rrbracket$, whose reduction modulo $\gothm_\gothR$ is nontrivial.

(3) For any $j \in\{1, \dots, m\}$, $ i \in \{1, \dots, d\}$ and any $\sigma \in \mathfrak S_m$ such that $b_{ij}>0$ and $\sigma(1) =j$, $\mathsf{Ref}^\sharp(M \otimes_{R_{n,m}} \sfF_\sigma; i)$ is  of the form
$
(b_{i\sigma(1)}, \dots, b_{i\sigma(m)} , \vartheta_i)
$
, where $\vartheta_i$ is some element in $\bigoplus_{\omega \in \Log^*} \gothR\cdot \omega$ for some local ring $(\gothR, \gothm_\gothR)$ finite over $k\llbracket x_{m+1}, \dots, x_n\rrbracket$, whose reduction modulo $\gothm_\gothR$ is nontrivial.
\end{theorem}
\begin{proof}
$(1) \Rightarrow (3)$ follows immediately from Proposition~\ref{P:variation-vs-refined} above because $C$ therein  is a subset of $(0, \infty)^m \times \{0\}^{n-m}$.  $(3) \Rightarrow (2)$ is tautology.  So it suffices to  prove $(2) \Rightarrow (1)$. 
Let $j$ be the index in $(2)$. By Proposition~\ref{P:variation-vs-refined} again, there exist $\epsilon_1, \dots, \epsilon_m >0$ and $C$ therein, such that \eqref{E:variation-vs-refined} and \eqref{E:variation-vs-refined1} hold.
Note that condition \eqref{E:j-ramifies} implies that if $b_{ij} = 0$ then $b_{ij'} =0$ for all $j'$; thus we can upgrade condition 
\eqref{E:variation-vs-refined1}
to $g_1^\sharp(\bbM_0, \underline r) = \cdots  = g_{\dim \bbM_0}(\bbM_0, \underline r) = 0$ for all $\underline r \in C$.

  We now prove that the functions $g_i^\sharp(M, \underline r)$ are linear in $\underline r$ for $\underline r \in [0, \infty)^m \times \{0\}^{n-m}$ by induction on $i$.
For $i=0$, this is void.  Assume that $G_{i-1}^\sharp(M,\underline r)$ is linear in $\underline r$; we show that $g_i^\sharp(M,\underline r)$ is linear in $\underline r$. By the convexity in Proposition~\ref{P:weak-decomposition}(i${}^\sharp$), each $G_i^\sharp(M,\underline r)$ and hence $g_i^\sharp(M,\underline r) = G_i^\sharp(M,\underline r) - G_{i-1}^\sharp (M,\underline r)$ is convex over $[0, \infty)^m \times \{0\}^{n-m}$.  Moreover, we know that condition (2) implies that $g_i^\sharp(M,\underline r) = b_{i1} r_1 + \cdots + b_{im} r_m$ for all $\underline r \in C$ and the axes $\RR e_1 \cup \cdots \cup \RR e_m$.  This forces the equality  $g_i^\sharp(M,\underline r) = b_{i1} r_1 + \cdots + b_{im} r_m$ for all $\underline r \in [0, \infty)^m \times \{0\}^{n-m}$,
because a convex function agrees with a linear function on a convex polygon if they agree at all vertices (lying on $\RR e_1 \cup \cdots \cup \RR e_m$) and a point in the interior (some point in $C$).
This concludes the induction process.  The statement on refined partially intrinsic radii for $\underline r\in C$  follows from Proposition~\ref{P:variation-vs-refined} and for general $\underline r \in (0, \infty)^m \times \{0\}^{n-m}$ follows from further applying the variation property in Proposition~\ref{P:weak-decomposition}(iii${}^\sharp$).
\end{proof}

\subsubsection{Pointed local setup}
\label{SS:pointed-local}
From now on, we assume that we are in one of the following pointed local situations (centered at $z$).

(a) (Geometric) Assume that we are in the geometric local situation \ref{SS:local-setup}(a), that is $X$ being a smooth affine variety with an \'etale morphism $p: X \to \AAA^n$ such that $D = p^{-1}(V(x_1\cdots x_m))$, where the affine space has standard coordinates $x_1, \dots, x_n$.  We assume moreover that $p^{-1}(0) = \{z\}$ is a single closed point, and $D_1 = \pi^{-1}(V(x_1)), \dots, D_m= \pi^{-1}(V(x_m))$ are all irreducible.  We have $D = \cup_{j=1}^m D_j$.

(b) (Formal) We take $X = \Spec\, k\llbracket x_1, \dots, x_n\rrbracket$, $D_1 = V(x_1), \dots, D_m = V(x_m)$, $D=\cup_{j=1}^m D_j$, and $z$ be the origin.

We set $U = X-D$ and let $j: U \hookrightarrow X$ denote the natural embedding.
We take the log-structure to be $\Log = \{\partial_1, \dots, \partial_m\}$.
For each $j$, let $F_{(j)}$ denote the completion of $ k(X)$ with respect to the valuation corresponding to $D_j$. 
 In the formal case, $\calO_U \simeq R_{n,m} = k\llbracket x_1, \dots, x_n\rrbracket [ x_1^{-1}, \dots, x_m^{-1}]$.

Let $M$ be a vector bundle over $U$ with an integrable connection.

\subsubsection{Irregularity $\QQ$-divisors}
\label{SS:irregularity Q-divisor}
We define the \emph{irregularity $\QQ$-divisors} to be
\[
R_i = \sum_{j=1}^r b_{ij} \cdot D_j = \sum_{j=1}^r \Irr(M \otimes F_{(j)}; i) \cdot D_j,
\]
for $i =1, \dots, d$.  They are divisors of $X$ with rational coefficients.

\begin{definition}
\label{D:clean-global}
Let $N$ be a positive integer such that $Nb_{ij}$ are all integers.
Under the two setups of \ref{SS:pointed-local}, we consider the following:

(a) In the geometric case, consider the morphism $g: \AAA^n \to \AAA^n$ given by $x_i \mapsto x_i^N$ if $i=1, \dots, m$ and $x_i \mapsto x_i$ if $i=m+1, \dots, n$.  Denote $X' = X \times_{\AAA^n, g} \AAA^n$.  Let $\pi:  X' \to X$ denote the natural morphism; then there exists a \emph{unique} closed point $z'$ lying above $z$.

(b)  Consider $\pi: X' = \Spec\big( R_{n,0}[x_1^{1/N}, \dots, x_m^{1/N}] \big) \to X$.
Let $z'$ denote the origin point of $X'$.

Let $D'_j$ be the reduced subscheme of $\pi^{-1}(D_j)$ for $j =1, \dots, m$.  Write $D' = \cup_{j=1}^m D'_j$ and set $R'_i = \pi^*(R_i)$ as divisors; note that $R'_i$'s now are genuine divisors (with integer coefficients).

We say that $M$ is \emph{clean} at $z$ if, for some $j \in \{1, \dots, m\}$ satisfying \eqref{E:j-ramifies}, there exists an (integral) scheme $\tilde D_j$ \emph{finite} over a neighborhood of $z$ in $D'_j$ such that, after \emph{reordering} the refined irregularities $\Ref^\sharp(M\otimes F_{(j)}; 1), \dots,\Ref^\sharp(M\otimes F_{(j)}; d) $, we have, for every point $\tilde z$ lying above $z$,
\begin{equation}
\label{E:clean-condition}
\Ref^\sharp(M\otimes F_{(j)}; i) \in \Omega^1_{X'}(\log D')(R'_i) \otimes \calO_{\tilde D_j, \tilde z},\textrm{ and it generates a direct summand of the latter}.
\end{equation}
This does not depend on the choice of $N$ because the morphism $\pi$ is log-\'etale; neither does it depend on the choice of $\tilde D_j$ (up to further shrinking the neighborhood of $z$) by Lemma~\ref{L:clean-independence} below.

By abuse of language, we write $\calO_{\tilde D_j}(R_i)$ for $\calO_{X'}(R'_i) \otimes \calO_{\tilde D_j}$; this does not depend on the choice of $X'$, nor on the local parameter system $p: X \to \AAA^n$ in the geometric local setup.
\end{definition}

\begin{remark}
\label{R:order-refined-irr}
Recall that when defining the refined irregularities in \ref{SS:refined-irr}, there is no canonical choice of the order of refined irregularities.  However, in Definition~\ref{D:clean-global}, the choice of the order does matter; it is related to the ordering given by some valuation on higher dimensional local fields.  See Theorem~\ref{T:clean-independence-of-j} below.
\end{remark}

\begin{remark}
\label{R:clean-suffice-formal}
In the geometric case, $M$ is clean at $z$ if and only if $M \otimes \calO_{X,z}^\wedge$ is clean at $z$ in the formal case by faithful flatness of $\calO_{D_j,z}^\wedge$ over $\calO_{D_j,z}$.  (Note that the cleanliness condition is essentially about refined irregularities being integral; so we do not need approximation to prove the sufficiency implication.)
\end{remark}

\begin{lemma}
\label{L:clean-independence}
Let $R \to R'$ be an integral extension of rings and let $M$ be a finite free $R$-module.  Then, an element $x \in M$ generates a direct summand of $M$ as $R$-modules if and only if $x$ generates a direct summand of $M\otimes_R R'$ as $R'$-modules.
\end{lemma}
\begin{proof}

If $R \cdot x$ is a direct summand of $M$, then $R'\cdot x$ is a direct summand of $M \otimes_R R'$ by tensoring with $R'$.  Conversely, we assume that $R' \cdot x$ is a direct summand of $M \otimes _R R'$.  Pick a basis $\bbe_1, \dots, \bbe_l$ of $M$ and write $x= x_1 \bbe_1 + \cdots + x_l \bbe_l$ for $x_1, \dots, x_l \in R$.  Then we know that $x_1, \dots, x_l$ generate the unit ideal of $R'$ and hence they also generate the unit ideal of $R$ by going-up theorem \cite[Proposition~4.15]{eisenbud}.
\end{proof}

\begin{theorem}
\label{T:clean-independence-of-j}
The differential module $M$ is clean at $z$ if and only if the equivalence conditions in Theorem~\ref{T:clean-equiv-local} hold.  In this case, the condition \eqref{E:clean-condition} also holds for any $j \in \{1,\dots, m\}$ and any $i \in \{1, \dots, d\}$ such that $b_{ij}>0$.
\end{theorem}
\begin{proof}
We may replace $M$ by $M \otimes \calO_{X,z}^\wedge$ and assume that $M $ is a differential module over $R_{n,m}$.  We observe that the condition \eqref{E:clean-condition} is preserved when replacing $X$ by $X'$, $D$ by $D'$, and $k$ by $k^\alg$ as in Definition~\ref{D:clean-global}.  Hence we may assume that $b_{ji}$ are all integers and $X' = X$, $D' =D$, and $k = k^\alg$ in Definition~\ref{D:clean-global}.

We first prove that the cleanliness condition at $z$ implies the condition (2) of Theorem~\ref{T:clean-equiv-local}.
We note that for any $\sigma \in \mathfrak S_m$, the valuation on $ k\llbracket x_1, \dots, x_n\rrbracket$ induced by the natural embedding to $\sfF_\sigma$ extends (not uniquely) to an embedding of $\calO_{\tilde D_j, \tilde z}$ to $\sfF_\sigma^\alg$.   The cleanliness of $M$ at $z$ implies that, for some $j \in \{1, \dots, m\}$ satisfying \eqref{E:j-ramifies} and each point $\tilde z$ above $z$, we have $
\Ref^\sharp(M \otimes F_{(j)}; i) \in \Omega^1_X(\log D)(R_i) \otimes_{\calO_X}\calO_{\tilde D_j, \tilde z}$ for each $i$
and it generates a direct summand of the latter.  In explicit terms, this means that if we write out
\[
\Ref^\sharp(M \otimes F_{(j)}; i) = x_1^{-b_{i1}}\cdots x_m^{-b_{im}} \big(\theta_{i1} \frac{dx_1}{x_1} + \cdots +\theta_{im} \frac{dx_m}{x_m} + \theta_{i,m+1} dx_{m+1} + \cdots +\theta_{in} dx_n \big),
\]
we have $\theta_{ij} \in \calO_{\tilde D_j, \tilde z}$ and $\theta_{i1}, \dots, \theta_{in}$ together generate the unit ideal.   This implies that, for (any) $\sigma \in \mathfrak S_m$ with $\sigma(1) =j$, $\mathsf{Ref}^\sharp(M \otimes \sfF_\sigma)$ consists of \[
(b_{i\sigma(1)}, \dots, b_{i\sigma(m)},\bar \theta_{i1} \frac{dx_1}{x_1} + \cdots +\bar \theta_{im} \frac{dx_m}{x_m} + \bar \theta_{i,m+1} dx_{m+1} + \cdots +\bar \theta_{in} dx_n ),
\]
where $\bar \theta_{ij}$ is the reduction of $\theta_{ij}$ in $\gothR=\calO_{\tilde D_j}\big / \sqrt{(x_1, \dots, \hat x_j, \dots, x_m)}$.
Hence, $M$ satisfies the condition (2) in Theorem~\ref{T:clean-equiv-local}, and hence all other equivalent conditions.

Conversely, we assume the equivalent conditions in Theorem~\ref{T:clean-equiv-local}.  By Theorem~\ref{T:strong-decomposition}, $M$ is a direct sum of differential modules over $R_{n,m}$ of pure partially intrinsic radius (when tensored with $F_{\underline r}$ for any $\underline r \in (0, \infty)^m \times \{0\}^{n-m}$).  So, we replace $M$ by each direct summand and assume that $M$ has this property itself.  In particular, we assume $b_{1j}=\cdots =b_{dj}$ for any $j$.  Now, we prove that the condition (3) of Theorem~\ref{T:clean-equiv-local} implies condition~\eqref{E:clean-condition} for any $j \in \{1, \dots, m\}$ with $b_{1j}>0$.  We first invoke Lemma~\ref{L:refined-integral} to show that there exists an integral scheme $\tilde D_j$ finite over $D_j$ such that if we write
\[
\Ref^\sharp(M \otimes F_{(j)};i) = x_1^{-b_1} \cdots x_m^{-b_m} \big( \theta_{i1} \frac{dx_1}{x_1} + \cdots + \theta_{im} \frac{dx_m}{x_m} + \theta_{i,m+1} dx_{m+1} +\cdots +\theta_{in} dx_n \big),
\]
then we have $\theta_{ij'} \in \calO_{\tilde D_j, \tilde z}[1/x_1\cdots \hat x_j \cdots x_m]$ for any $j' \in \{1,\dots, n\}$. We fix $\sigma \in \mathfrak S_m$ such that $\sigma(1) = j$.  By the condition (3) of Theorem~\ref{T:clean-equiv-local}, we know that for \emph{any} embedding $\iota: \calO_{\tilde D_j}\llbracket x_j\rrbracket \hookrightarrow \sfF_\sigma^\alg$ extending the natural embedding $\calO_{D_j}\llbracket x_j\rrbracket \hookrightarrow \sfF_\sigma$, we have  $\min\{\mathsf v(\iota(\theta_{i1})), \dots, \mathsf v(\iota(\theta_{in}) )\} = \underline 0$ for any fixed $i$.  
This in particular implies that $\theta_{ij'} \in \calO_{\tilde D_j}$ for all $i$ and $j'$.
Moreover, for any point $\tilde z$ of $\tilde D_j$ lying above $z$, we can find an embedding $\iota$ as above such that the maximal ideal $\gothm_{\tilde z}$ is given by $\{x \in \calO_{\tilde D_j, \tilde z} | \mathsf v(\iota(x)) > \underline 0 \}$.  Hence, for any fixed $i$, there exists some $\theta_{ij'} \notin \gothm_{\tilde z}$.  This implies that, for any fixed $i$ and any point $\tilde z$ above $z$, $\Ref^\sharp(M \otimes F_{(j)}; i) \in \Omega^1_{X'}(\log D)(R_i) \otimes \calO_{\tilde D_j, \tilde z}$
and it generates a direct summand.
\end{proof}

\begin{remark}
\label{R:clean=>numerical}
When $z$ is the intersection of exactly $n$ divisors $D_1, \dots, D_n$, the proof of the theorem implies that the cleanliness at $z$ is equivalent to numerical cleanliness at $z$.  This is however false for other points on $D$.  See Remark~\ref{R:clean<numerical}.
\end{remark}

\begin{theorem}
\label{T:numerical=>clean}
If $M\otimes \calO_{X, z}^\wedge$ is numerically clean at $z$ then $M $ is clean at $z$.
\end{theorem}
\begin{proof}
The proof is very similar to the proof above, but we have to be very careful at a few places, which hints why the converse of the theorem does not hold in general.
By Remark~\ref{R:clean-suffice-formal}, it suffices to assume that we are in the formal situation, that is $M$ is a finite and flat differential module over $R_{n,m}$.
Since the numerical cleanliness condition is preserved when replacing $X$ by $X'$ as in Definition~\ref{D:clean-global} and replacing $k$ by its algebraic closure, we may assume that $r_{ji}$ are all integers and $X' = X$ and $D' =D$ in Definition~\ref{D:clean-global}, and that $k$ is algebraically closed.

By Theorem~\ref{T:variation}(ii), we may assume that $M \otimes F_{\underline r}$ has pure intrinsic radius $b_1r_1 + \cdots + b_nr_n$ for all $\underline r \in [0, \infty)^n$ (with respect to the full log-structure).  When $b_1=\cdots =b_n =0$, $M$ is regular and it is obviously clean at $z$.  From now on, we assume that this is not the case; without loss of generality, we assume that $b_1>0$.  
As in proof of Theorem~\ref{T:clean-independence-of-j}, we first invoke Lemma~\ref{L:refined-integral} to show that there exists an integral scheme $\tilde D_1$ finite over $D_1$ such that if we write
\[
\Ref^\sharp(M \otimes F_{(1)}; i) = x_1^{-b_1} \cdots x_m^{-b_m} \big( \theta_{i1} \frac{dx_1}{x_1} + \cdots + \theta_{im} \frac{dx_m}{x_m} + \theta_{i,m+1} dx_{m+1} +\cdots +\theta_{in} dx_n \big),
\]
then we have $\theta_{ij} \in \calO_{\tilde D_1, \tilde z}[1/x_2\cdots x_m]$ for all $j \in \{1,\dots, n\}$.  Note that this is written in a form adapted to the log-structure $\Log=\{\partial_1, \dots, \partial_m\}$ but not the full log-structure.  Applying 
Theorem~\ref{T:clean-equiv-local} to $M \otimes R_{n,n}$, we know that for \emph{any} embedding $\iota: \calO_{\tilde D_1}\llbracket x_1\rrbracket \hookrightarrow \sfF_\mathrm{id}^\alg$ extending the natural embedding $\calO_{D_1}\llbracket x_1\rrbracket \hookrightarrow \sfF_\mathrm{id}$, we have 
\begin{equation}
\label{E:minimals}
\min\big\{\mathsf v(\iota(\theta_{i1})), \dots, \mathsf v(\iota(\theta_{im})), \mathsf v(\iota(x_{m+1}\theta_{i,m+1})), \dots, \mathsf v(\iota(x_n\theta_{in}) )\big\} = \underline 0
\end{equation}
for any fixed $i$.
Note that, for $j = m+1, \dots, n$, the facts that $\theta_{ij} \in \calO_{\tilde D_1, \tilde z}[1/x_2\cdots x_m]$ and that $\mathsf v(\iota(x_j\theta_{ij})) \geq \underline 0$ force $\mathsf v(\iota(x_j\theta_{ij})) > \underline 0$.  The minimum in \eqref{E:minimals} is taken from the first $m$ terms, i.e. $\min\{\mathsf v(\iota(\theta_{i1})), \dots, \mathsf v(\iota(\theta_{im}) \} = \underline 0$.  Now, we proceed exactly as in the proof of Theorem~\ref{T:clean-independence-of-j} to conclude.
\end{proof}

\begin{remark}
\label{R:clean<numerical}
We remark that cleanliness does not imply numerical cleanliness in general.  We construct a counterexample as follows.  Let $X = \AAA^2$ with coordinates $x$ and $y$, and let $D$ be the $x$-axis.  Consider the differential module $M = k[x,y][y^{-1}]\cdot \bbe$ given by $\partial_x \bbe = \frac 1{y^2}\bbe$ and $\partial_y  \bbe = -2\frac x{y^3}\bbe$; in other words, $\bbe$ is a proxy of $e^{x/y^2}$.  The refined partially intrinsic radii of $M$ along $D$ is $\frac 1{y^2}dx -\frac{2x}{y^2}\frac{dy}y$, which is clean everywhere on $D$ by definition.  However, at the origin, the corresponding function $g_1(M, F_{r_1, r_2}) = \max\{2r_2-r_1, 0\}$ for $r_1, r_2 \in [0,\infty)^2$ is not linear.
\end{remark}

Now, we switch back to the global situation.

\begin{definition}
\label{D:global clean}
Let $X$ be a smooth variety of dimension $n$ over $k$ and let $D= \cup D_j$ be a divisor with  simple normal crossings, where $D_j$ are irreducible components of $D$.
Let $M$ be a vector bundle over $U$ with an integrable connection.
We say that $M$ is \emph{(numerically) clean}, if for all closed point $z \in X$, $M \otimes \calO_{X,z}^\wedge$ is (numerically) clean.  Theorem~\ref{T:numerical=>clean} implies that numerically cleanliness $\Rightarrow$ cleanliness.
\end{definition}

\begin{remark}
\label{R:no-global-defn-clean}
Even if $M$ is clean over $X$ and $b_{ij} \in \ZZ_{>0}$ for all $i, j$ (which implies \eqref{E:j-ramifies}), it does \emph{not} mean that Definition~\ref{D:clean-global} holds globally, i.e., in the global situation above, we may not be able to find an integral scheme $\tilde D_j$ finite over $D_j$ such that, after reordering the refined irregularities, we have that $\Ref^\sharp(M \otimes F_{(j)}; i) \in \Omega_X^1(\log D)(R_i) \otimes \calO_{\tilde D_j}$ and it locally generates a direct summand.  The existence of $\tilde D_j$ is not the problem, but when $D_i \cap D_j$ is not connected, the reordering of refined irregularities in Definition~\ref{D:clean-global} might be different from points to points; this presents a difficulty in reasonably define the $\QQ$-divisor $R_i$.  In general, we do not expect any a priori reason for a uniform choice of $R_i$.

On the other hands, if $M$ is clean over $X$ and we have strict inequalities $b_{1j}> \cdots > b_{dj}>0$ for all $j$ (which automatically implies $b_{ij} \in \NN$ and \eqref{E:j-ramifies}), then there is a unique ordering of these refined irregularities that makes Definition~\ref{D:clean-global} work globally; in this case, there exists an integral scheme $\tilde D_j$ finite over $D_j$ such that $\Ref^\sharp(M \otimes F_{(j)}; i) \in \Omega_X^1(\log D)(R_i) \otimes \calO_{\tilde D_j}$ and it locally generates a direct summand.
\end{remark}

\begin{prop}
\label{P:generic-clean}
Keep the notation as in Definition~\ref{D:global clean}.  The set of closed points $V$ of $X$ at which $M$ is clean, is the set of closed points of an open subvariety of $X$.  We call this open subvariety the \emph{clean locus} of $M$.  Moreover, its complement has codimension $\geq 2$.
\end{prop}
\begin{proof}
Only in this proof, all varieties are viewed as a set of its closed points with Zariski topology.
To prove the proposition, we may assume that we are in the geometric local situation \ref{SS:local-setup}.  In this case, $D$ is the union of irreducible divisors $D_1, \dots, D_m$ of $X$.  It suffices to prove that the intersection of the clean locus $V$ with each $D_j$ is open and dense in $D_j$.
By reordering $D_j$, we may assume
\[
\textrm{if }j<j',\quad \min\{i|b_{ij}=0\} \leq \min\{i|b_{ij'}=0\}.
\]

We now prove that $V \cap D_j$ is open and dense in $D_j$ by induction on $j$.
The statement is void if $j=0$.  Assume that the statement is proved for all $j < j_0$ and we prove it for $j =j_0$.
By Theorem~\ref{T:clean-independence-of-j}, we know that the set $V \cap D_{j_0}$ is contained in the set $W$ where the condition \eqref{E:clean-condition} is fulfilled for each $i$ for which $b_{ij_0}>0$; it follows from the definition that $W$ is an open and dense subset of $D_j$.
Unfortunately, $V \cap D_{j_0}$ may still be different from $W \cap D_{j_0}$ because condition \eqref{E:clean-condition} can be used to test cleanliness only if the index $j_0$ satisfies condition \eqref{E:j-ramifies}.  (In other words, When $b_{ij_0} =0$, we need to know whether other $b_{ij}$ might still be positive in which case, the cleanliness will be determined by condition \eqref{E:clean-condition} on that $j$.)

So we need to prove that $W \bs V$ is closed in $W$. Since condition \eqref{E:j-ramifies} automatically holds for closed points $z \in D_{j_0}^\circ := D_{j_0} \bs \cup_{j<j_0}D_j$ by our ordering of the divisors, we have $V \cap D_{j_0}^\circ = W \cap D_{j_0}^\circ$. It then suffices to show that $(W \cap (D_{j_0} \bs D_{j_0}^\circ) )\bs V$ is closed in $W$. 
In fact, we will show that $(D_{j_0} \bs D_{j_0}^\circ )\bs V$ is closed in $D_{j_0} \bs D^\circ_{j_0}$ which obviously implies the previous sentence.
By induction hypothesis, for any $j < j_0$, $D_j \bs V$ is closed in $D_j$ and hence $(D_j \cap D_{j_0}) \bs V$ is closed in $D_j \cap D_{j_0}$; this implies that $(D_{j_0} \bs D_{j_0}^\circ) \bs V$ is closed in $(D_{j_0} \bs D_{j_0}^\circ)$, finishing the inductive proof. 
\end{proof}

\begin{remark}
The cleanliness condition is a very restrictive condition.  However, Kedlaya \cite{kedlaya-sabbah1, kedlaya-sabbah2} proved that, after certain blowups, one can achieve this condition.  The precise statement is the following.
\end{remark}

\begin{theorem}
Let $X$ be a smooth variety of dimension $n$ over $k$ and let $D$ be a divisor with  simple normal crossings.  Let $M$ be a differential module over $X-D$.  Then there exists a proper birational morphism of smooth pairs $f: (X', D') \to (X, D)$ such that $f|_{X'-D'}: X'-D' \to X-D$ is an isomorphism and $f^*M$ admits a good formal structure at each closed point of $X'$.  In particular, $f^*M$ is clean on $X'$.
\end{theorem}

\begin{remark}
One might question the need of introducing the (weaker version of) cleanliness since we can achieve good formal structure under proper birational pullback.  One reason is that the current version of cleanliness is closely tied to the conjectural log-characteristic cycles.  Another reason is that, in the analogous positive characteristic situation, one do not have a notion of ``good formal structure".  In fact, we do not expect to achieve ``numerically cleanliness" under birational proper pullback.
\end{remark}

\section{Main theorem}

\subsection{Statement of the main theorem}

\begin{definition}
\label{D:ZCar'}
We now define the conjectural log-characteristic cycles.  We first assume that we are in the local setup \ref{SS:local-setup}.  We do not assume that $M$ is clean to begin with, and hence the conjectural log-characteristic cycle may not be equal to the actual log-characteristic cycle.  (See Proposition~\ref{P:ZCar'<ZCar} though.)

We fix $j\in \{1, \dots, m\}$.  Let $F_{(j)}$ denote the completion of $k(X)$ with respect to the valuation corresponding to the divisor $D_j$; let $\mathfrak o_{(j)}$ and $\kappa_{(j)}$ be the corresponding valuation ring and residue field, respectively.  We now pass to the CDVF situation \ref{SS:local-setup}(c).  We have defined the refined irregularities of $M_{(j)}=M \otimes F_{(j)}$ in \ref{SS:refined-irr}:
\begin{equation}
\label{E:refined-irregularities}
\Ref(M_{(j)};i) \in \bigoplus_{l=1}^n x_j^{-\Irr(M_{(j)})} \kappa_{(j)}^\alg \frac{dx_l}{x_l} \textrm{ for }i = 1, \dots, d.
\end{equation}

We first assume that all refined irregularities of $M_{(j)}$ come from the same $\Gal(F'_{(j)}/ F_{(j)})$-orbit for some finite Galois extension $F'_{(j)}$ of $F_{(j)}$ containing $x_j^{-\Irr(M_{(j)})}$; in particular, $M_{(j)}$ has pure irregularity $\Irr(M_{(j)})$.  For each $i$, we view $\Ref(M_{(j)};i)$ as a homomorphism
\[
\Ref(M_{(j)};i): x_j^{\Irr(M_{(j)})}\kappa_{F'_{(j)}} \to \Omega^1_X(\log D) \otimes_{\calO_X} \kappa_{F'_{(j)}}.
\]
This defines a line $L_{ij}$ in the vector space $T^*\!X^\log \times_X \Spec\, \kappa_{F'_{(j)}}$.  Consider the pushforward morphism $\pi: T^*\!X^\log \times_X \Spec\, \kappa_{F'_{(j)}} \to T^*\!X^\log \times_X \Spec\, \kappa_{F_{(j)}}$.  Let $\overline L_{ij}$ denote the closure of $\pi_*(L_{ij})$ in $T^*\!X^\log \times_X D_j$.

We define the \emph{conjectural log-characterisitic cycle} over $D_j$ to be
\begin{equation}
\label{E:ZCar'j(M) pure break}
\ZCar'_j(M) = \frac{\rank M \cdot \Irr(M_{(j)})}{[F'_{(j)}:F_{(j)}]} \overline L_{ij}.
\end{equation}
By Corollary~\ref{C:strong-integrality}, the coefficient of the cycle $\overline L_{ij}$ is an integer; moreover, the definition of $\ZCar'_j(M)$ does not depend on the choice of $F'_{(j)}$ and $i$.
 
For general $M$, We write $M_{(j)}$ as a direct sum of $M_{(j), \{G\vartheta\}}$ by Proposition~\ref{P:decomposition-field}(iii), where $M_{(j), \{G\vartheta\}}$ satisfies the assumption above.  We define
the \emph{conjectural log-characterisitic cycle} over $D_j$ to be $\ZCar'_j(M) = \sum_{\{G\vartheta\}} \ZCar'_j(M_{(j), \{G\vartheta\}})$.

Finally, we define the \emph{conjectural log-characterisitic cycle} of $M$ to be
\[
\ZCar'(M) = \rank(M) \cdot [X] + \sum_{j=1}^m\ZCar'_j(M),
\]
where $[X]$ is the zero section of $T^*\!X^\log$.

We use $\Car'(M)$ to denote the support of $\ZCar'(M)$, called the \emph{conjectural log-characteristic variety} of $M$ (although it is often not irreducible as a scheme).

Now, we assume that we are in the global situation \ref{SS:global-situation}; the smooth pair $(X, D)$ is covered by open subvarieties $(V_i, V_i \cap D)$, each of which satisfies the local situation~\ref{SS:local-setup}(a). We define the \emph{conjectural log-characterisitic cycle} of $M$ to be the cycle $\ZCar'(M)$ of $T^*\!X^\log$ whose restriction to each $V_i$ is the conjectural log-characteristic cycle $\ZCar'(M|_{V_i})$ defined above.
\end{definition}

We point out the following immediate property of $\ZCar'(M)$.

\begin{lemma}
\label{L:ZCar'-stable}
Assume that we are in one of the following situations:

(i) We are in the geometric local setup~\ref{SS:local-setup}(a).  Let $z$ be a closed point of $p^{-1}(\{0\})$.  Then we consider the natural morphism $g: X' = \Spec \calO_{X, z}^\wedge \to X$ and view $g^*M$  as a vector bundle over $U' = \Spec\big( \calO_{X, z}^\wedge[1/x_1 \cdots x_m]\big)$.  

(ii) We are in geometric or formal local setup \ref{SS:local-setup}(a)(b).  Let $\eta_1$ denote the generic point of $D_1$.  We consider the natural morphism $g: X'=\Spec \calO_{X, \eta_1}^\wedge \to X$; $g^*M$ may be viewed as a vector bundle over $U' = \Spec (k(X)^{\wedge, \eta_1})$.  

(iii) We are in either case of the local setup~\ref{SS:local-setup}.  Let $X''$ be \'etale over $X$ and let $X' = \Spec \big(\calO_{X''}[x_1^{1/h_1}, \dots, x_m^{1/h_m}]\big)$ for some positive integers $h_1, \dots, h_m$.  We have a natural morphism $g: X' \to X$ and $g^*M$ becomes a vector bundle over $U' = \Spec \big(\calO_{X'}[1/x_1 \cdots 1/x_m] \big)$.

 Then we have $\ZCar'(g^*M) = \tilde g^*(\ZCar'(M))$, where $\tilde g: T^*\!X'^\log \to T^*\!X^\log$ is the natural morphism.
\end{lemma}
\begin{proof}
Since all the morphisms $g$ involved are (formally) log-\'etale, it is straightforward to check the equalities of the cycles.
\end{proof}

\begin{construction}
\label{C:global-description}
Assume that $M$ is clean over $X$.  Then we may give a global construction of the conjectural log-characteristic cycle as follows.

Write $D= \cup_{j=1}^r D_j$ as the union of irreducible components.
Put $D_{jj'} = D_j \cap D_{j'}$ if $j \neq j'$.
We write $D_{jj'} = \coprod_{\alpha = 1}^{\lambda_{jj'}}D_{jj', (\alpha)}$ as the union of irreducible components.
Fix $j \in \{1, \dots, r\}$, and let $F_{(j)}$ denote the completion of $ k(X)$ with respect to the valuation corresponding to the divisor $D_j$.  The refined irregularities of $M_{(j)} = M \otimes F_{(j)}$ are
\[
\Ref(M_{(j)}; i) \in \Omega_X^1(\log D) \otimes \pi_{F_{(j)}}^{-\Irr(M_{(j)}; i)}\kappa_{F^\alg_{(j)}}, \quad i=1, \dots, d.
\]
There exists a finite extension $F'_{(j)}$ of $F_{(j)}$ such that every $\Ref(M_{(j)}; i) \in \Omega_X^1(\log D) \otimes \pi_{F_{(j)}}^{-\Irr(M_{(j)}; i)}\kappa_{F'_{(j)}}$.  Let $D'_j$ denote the integral closure of $D_j$ in $\kappa_{F'_{(j)}}$.  Since we have assumed that $M$ is clean over $X$, there exists $\QQ$-divisor $R_i^{(j)} = \sum_{j'} \sum_{\alpha} b_{ij', (\alpha)}^{(j)}D_{jj',(\alpha)}$ of $D_j$ such that 
\begin{equation}
\label{E:where Ref belongs}
\Ref(M_{(j)}; i) \in \Omega^1_X(\log D) \otimes_{\calO_X} \calO_{D'_j}(R_i^{(j)})
\end{equation}
and it generates a direct summand of the latter.
(Note that $R_i^{(j)}$ is a rational divisor, we should understand $\Omega^1_X(\log D) \otimes_{\calO_X} \calO_{D'_j}(R_i^{(j)})$ locally as $\Omega^1_X(\log D) \otimes_{\calO_X} \calO_{D'_j}(R_i^{(j)})$, where the latter is as introduced at the end of Definition~\ref{D:clean-global}.)
In this case, we view $\Ref(M_{(j)}; i)$ as a morphism from $\calO_{D'_j}(-R_i^{(j)})$ to $\Omega^1_X(\log D) \otimes _{\calO_X}\calO_{D'_j}$ and let $L'_{ij}$ denote its image, viewed as a line subbundle of the base change of the cotangent bundle $T^*\!X^\log \times_X D'_j$.  We define
\begin{equation}
\label{E:ZCar'j(M) global}
\ZCar'_j(M) = \sum_{i=1}^d \frac{\Irr(M_{(j)}; i)}{[F'_{(j)}: F_{(j)}]}\pi_{j*}(L'_{ij}),
\end{equation}
where $\pi_{j*}$ is the natural morphism $T^*\!X^\log \times_X D'_j \to T^*\!X^\log \times_X D_j$.
This definition agrees with Definition~\ref{D:ZCar'} in the sense that $\overline L_{ij}$ in \eqref{E:ZCar'j(M) pure break} is a proper multiple of $\pi_{j*}L'_{ij}$ in \eqref{E:ZCar'j(M) global} accounting for the difference between the field extension and multiplicity of the refined conductor.

We remind the reader again that $b_{ij',(\alpha)}^{(j)}$ may not be the same as $b_{ij'}$ as it depends on $j'$ and on $\alpha$.  (See Remark~\ref{R:no-global-defn-clean}.)  We only know that, for fixed $j$ and $\alpha$, the multiset of numbers $\{b_{ij',(\alpha)}^{(j)}|i=1, \dots, d\}$ is the same as $\{ b_{ij'}|i=1, \dots, d\} = \calI rr(M \otimes F_{(j')})$ (but possibly in different order). More generally, whenever $j, j_1, \dots, j_t \in \{1, \dots, r\}$ such that $D_j \cap D_{j_1} \cap \cdots \cap D_{j_t} \neq \emptyset$ (and hence connected by our assumption), the cleanliness condition at any point of the intersection implies the equality of multisets of $t$-tuples
\begin{equation}
\label{E:intersection-indices}
\big\{(b_{ij_1, (\alpha_1)}^{(j)}, \dots, b_{ij_t, (\alpha_t)}^{(j)}) \;|\; i = 1, \dots, d\big\} = \big\{(b_{ij_1}, \dots, b_{ij_t}) \;|\; i = 1, \dots, d\big\},
\end{equation}
where $\alpha_s$ is label determined by $D_j \cap D_{j_1} \cap \cdots \cap D_{j_t} \subseteq D_{jj_s, (\alpha_s)}$.

\end{construction}

The following is the main theorem of this paper; its proof will occupy the rest of the section.

\begin{theorem}
\label{T:main-theorem}
Let $X$ be a smooth variety over $k$ and let $D$ be a divisor with  simple normal crossings.  
Let $(M, \nabla)$ be a vector bundle over $U = X-D$ with an integrable connection.  Let $j:U \hookrightarrow X$ denote the natural inclusion. Assume that $M$ is \emph{clean} on $X$.  Then $\ZCar'(M) = \ZCar(j_*M)$.
\end{theorem}

\begin{corollary}
\label{C:EP-formula}
Keep the notation as in Theorem~\ref{T:main-theorem}. Assume that $(M, \nabla)$ is clean on $X$ and $X$ is proper.  
Assume moreover that all $b_{ij}$'s from \ref{SS:irregularity Q-divisor} are \emph{positive}.
Let $R_i$ denote the irregularity $\QQ$-divisor as in \ref{SS:irregularity Q-divisor}.  Then we have
\begin{equation}
\label{E:EP-formula}
\chi_\dR(M) =(-1)^n\sum_{i=1}^d \deg \big(c(\Omega^1_X(\log D)) \cap (1-R_i)^{-1}\big),
\end{equation}
where $c(\cdot)$ denote the total Chern class. 
\end{corollary}
\begin{proof}
We will verify the technical condition of Theorem~\ref{T:log-Kashiwara-Dubson} in Lemma~\ref{L:tilde M0 is everything}.  Thus by Theorem~\ref{T:log-Kashiwara-Dubson} and Theorem~\ref{T:main-theorem}, we have
\[
\chi_\dR(M) = (-1)^n \cdot \deg\big([X], \ZCar(j_*M)\big)_{T^*\!X^\log} =(-1)^n \cdot \deg\big([X], \ZCar'(j_*M)\big)_{T^*\!X^\log}.
\]
It suffices to compute the latter intersection number.  For this, we use the description of $\ZCar(j_*M)$ in Construction~\ref{C:global-description}.
\begin{align*}
\big([X], \ZCar'(j_*M)\big)_{T^*\!X^\log}
&= d\big([X],[X]\big)_{T^*\!X^\log}
+ \sum_{j=1}^m \big([X], \ZCar'_j(j_*M)\big)_{T^*\!X^\log} \\
&=d\big([X],[X]\big)_{T^*\!X^\log}
+ \sum_{j=1}^m\sum_{i=1}^d \frac{\Irr(M_{(j)}; i)}{[F'_{(j)}: F_{(j)}]}\big([X],\pi_{j*}(L_{ij}) \big)_{T^*\!X^\log} 
\end{align*}
By \cite{fulton}, $\big([X],[X]\big)_{T^*\!X^\log} = \deg (c_n(\Omega_X^1(\log D)))$ and the last intersection is given by the intersection of the total Chern class of $\Omega^1_X(\log D)$ with the Segre class of $D'_j$ in $\overline L_{ij}$.  Hence, 
\begin{align*}
\chi_{\dR}(M) &= (-1)^n \cdot\deg \Big(d\big([X],[X]\big)_{T^*\!X^\log}
+ \sum_{j=1}^m\sum_{i=1}^d \frac{\Irr(M_{(j)}; i)}{[F'_{(j)}: F_{(j)}]}\big(c(\Omega^1_X(\log D)) \cdot \pi_{j*}c(\calO_{D'_j}(-R_i^{(j)}))^{-1} \big)_{T^*\!X^\log} \Big) \\
&= (-1)^n \cdot \deg\Big(d\cdot c_n(\Omega_X^1(\log D)) 
+ \sum_{j=1}^m\sum_{i=1}^d \Irr(M_{(j)}; i)c(\Omega^1_X(\log D))\cdot D_j \cdot (1-R_i^{(j)})^{-1} \Big) \\
&=  (-1)^n \sum_{i=1}^d \deg\Big(c_n(\Omega_X^1(\log D) )
+ \sum_{j=1}^m  \Irr(M_{(j)}; i) c(\Omega^1_X(\log D))\cdot D_j \cdot (1-R_i)^{-1} \Big) \\
& =(-1)^n \sum_{i=1}^d \deg\Big(c_n(\Omega_X^1(\log D) )
+ c(\Omega^1_X(\log D))\cdot R_i \cdot (1-R_i)^{-1} \Big)\\
&= (-1)^n \sum_{i=1}^d \deg\big(
c(\Omega_X^1(\log D) ) \cdot (1-R_i)^{-1} \big).
\end{align*}
Here the third equality follows from \eqref{E:intersection-indices}.
\end{proof}

\begin{remark}
It is not clear from the formula why the intersection number on the right hand side of \eqref{E:EP-formula} should a priori give an integer.  One may view this as certain global version of Hasse-Arf Theorem.
\end{remark}

\begin{lemma}
\label{L:tilde M0 is everything}
Keep the notation as in Corollary~\ref{C:EP-formula}.  Suppose that every $b_{ij}$ as in \ref{SS:irregularity Q-divisor} are positive, then  $j_*M$ is generated as a $\calD_X^\log$-module by \emph{any} coherent $\calO_X$-submodule $M_0$ of $j_*M$ for which $M_0|_U = M$.
\end{lemma}
\begin{proof}
To prove the lemma, it suffices to prove it over the completion at a closed point $x \in X$.  Hence we may reduce to the formal local setup \ref{SS:local-setup} immediately.

Now we may assume that $M$ is a differential module over  $X = \Spec R_{n,m}$, and the ring of differential operator is $\calD_X^\log = R_{n,m}\{x_1\partial_1, \dots, x_m\partial_m, \partial_{m+1} , \dots, \partial_n\}$.  We first show that it suffices to prove $\calD_X^\log \cdot M_0 = M$ for \emph{some} submodule $M_0$ of $M$ for which $M = M_0 \otimes_{R_{n,0}}R_{n,m}$.
Indeed, any other $R_{n,0}$-lattice $M'_0$ of $M$ will contain $(x_1\cdots x_m)^N M_0$ for some $N \in \NN$, and by the proof of Lemma~\ref{L:log-char-independence}, we have 
\[
\calD_X^\log\cdot (x_1\cdots x_m)^N M_0 = (x_1\cdots x_m)^N \calD_X^\log\cdot M_0
=(x_1\cdots x_m)^N M=M.
\]
We deduce that $\calD_X^\log\cdot M'_0 = M$.

Proving the existence of the $M_0$ as above is the technical part, which will follow from Proposition~\ref{L:xi_1-generates} (for the case when $l=m$ because all $b_{ij}$'s are positive).
\end{proof}

\subsection{Overall of the proof}
In this subsection, we reduce the proof of Theorem~\ref{T:main-theorem} to the calculation on $R_{n,m}$.

First of all, Theorem~\ref{T:main-theorem} is local on $X$, and we may assume that we are in the geometric local situation \ref{SS:local-setup}(a).

\subsubsection{Outline of the the proof}
The crucial step is to prove that the set of closed points on the log-characteristic variety is contained in the set of closed points on the conjectural log-characteristic variety, i.e. $|\Car(j_*M)| \subseteq |\Car'(M)|  $.
For this, we may assume that $k$ is algebraically closed.  We need only to show that for each closed point $z \in X$, we have 
\[
|\Car(j_*M)| \cap \big(T^*\!X^\log  \times_X \{z\}\big) \subseteq|\Car'(M)| \cap \big(T^*\!X^\log \times_X \{z\}\big) .
\]
For this, we may base change to $\calO_{X,z}^\wedge \simeq R_{n,0}=k\llbracket x_1, \dots, x_n\rrbracket$ and reduce to the pointed geometric local situation \ref{SS:pointed-local}(a) (centered at $z$). Now, $j_*M \otimes \calO_{X,z}^\wedge$ becomes a differential module over $R_{n,m} = k\llbracket x_1, \dots, x_n\rrbracket[x_1^{-1}, \dots, x_m^{-1}]$.  By Corollary~\ref{C:logcharcycles-base-change}(ii) and Lemma~\ref{L:ZCar'-stable}(iii), we need to show that
\begin{equation}
\label{E:reduction}
|\Car(j_*M \otimes R_{n,0})| \cap \big(T^*\!\Spec(R_{n,0})^\log  \times \{z\}\big) \subseteq |\Car'(M \otimes R_{n,0})| \cap \big(T^*\!\Spec(R_{n,0})^\log  \times \{z\}\big).
\end{equation}
We defer the discussion of its proof to \ref{SS:local-calculation} below.

We now assume $|\Car(j_*M)| \subseteq |\Car'(M)|  $.
It follows immediately that $\Car(j_*M)$ is restricted within the union of some finite set of ($n$-dimensional) varieties, namely, the zero section $[X]$ of $T^*\!X^\log$ and some line bundles over the irreducible components of $D$ (because $\ZCar'(M)$ is so).  We need only to prove that the multiplicity at each generic point of these varieties agrees.  In fact, to prove this, we do not even need to assume that $M$ is clean on $X$, i.e., we will prove the following proposition, whose proof will be carried out in \ref{SS:proof-ZCar'<ZCar}.

\begin{prop}
\label{P:ZCar'<ZCar}
Keep the notation as in Theorem~\ref{T:main-theorem} except that we do \emph{not} assume that $M$ is clean on $X$.  Then $\ZCar(j_*M) - \ZCar'(M)$ is a non-negative linear combination of 
cycles in $T^*\!X^\log$ supported on $T^*\!X^\log \times_X W$ for some closed subvariety $W \subseteq D$ of codimension $\geq 1$.
\end{prop}

\begin{remark}
(Under the cleanliness assumption,) one may hope to prove $\ZCar'(M \otimes R_{n,0}) = \ZCar(j_*M \otimes R_{n,0})$ directly from the local calculation.  However, we do not know how to prove this directly, unless $m=1$ or when $M$ has a good formal structure.  This is why the proof has to proceed in two steps: checking supports and then matching multiplicities.
\end{remark}

\subsubsection{Local calculation}
\label{SS:local-calculation}
Now, we are back to the proof of \eqref{E:reduction}.  Set $X= \Spec R_{n,0}$ and $D = V(x_1\cdots x_m)$; let $z$ be the origin.  Let $M$ be a differential module over $U = \Spec R_{n,m}$ clean at $z$.
We keep in mind that Corollary~\ref{C:logcharcycles-base-change}(iii) and Lemma~\ref{L:ZCar'-stable}(iii) always allow us to replace $R_{n,0}$ by $R_{n,0}[x_1^{1/h_1}, \dots, x_m^{1/h_m}]$ for positive integers $h_1, \dots, h_m$.
By the direct sum decomposition given by combining Theorem~\ref{T:clean-independence-of-j} with Theorem~\ref{T:strong-decomposition}(ii), we may as well assume that $M \otimes F_{\underline r}$ has pure partially intrinsic radius $\mathrm e^{-b_1r_1-\cdots -b_mr_m}$ for all $\underline r \in [0, \infty)^m \times \{0\}^{n-m}$, where $b_1, \dots, b_m$ are nonnegative integers, and there exist $\underline \theta =(\theta_1, \dots, \theta_n) \in k^n \bs \{ \underline 0\}$ and a local ring $\gothR$ finite over $k\llbracket x_{m+1}, \dots, x_n\rrbracket$ such that
\[
x_1^{b_1} \cdots x_m^{b_m}\Ref^\sharp(M \otimes F_{\underline r}; i) \equiv \theta_1 \frac{dx_1}{x_1} + \cdots +\theta_m \frac{dx_m}{x_m} + \theta_{m+1} dx_{m+1} + \cdots +\theta_n dx_n \mod \gothm_\gothR
\]
for any $i$ and any $\underline r \in (0, \infty)^m \times \{0\}^{n-m}$.  

If we let $\xi_j$ denote the image of $x_j\partial_j$ if $j\leq m$ and of $\partial_j$ if $j >m$, in $\gr_1 \calD_X^\log$, then we have  $\gr_\bullet \calD_X^\log \simeq R_{n,0}[\xi_1, \dots, \xi_n]$.  The claim \eqref{E:reduction} follows from the explicit and separate calculation in the following two propositions for the conjectural log-characteristic cycles and the (genuine) log-characteristic cycles.
 
\begin{prop}
Keep the notation as above.  If $b_1= \cdots= b_m = 0$, the conjectural log-characteristic cycle is the zero section $X$ of the log-cotangent bundle with multiplicity $d$.

If $(b_1, \dots, b_m) \neq (0, \dots, 0)$, then $|\Car'(M)| \cap \big(T^*\!X^\log \times_X \{z\}\big)$ is the closed subset $Z_\vartheta$ defined by $x_1=\cdots = x_n= 0$ and $\theta_j\xi_i = \theta_i\xi_j$ for all $i\neq j$; in particular, this is a line in $T^*\!X^\log \times_X \{z\}$.
\end{prop}
\begin{proof}
When $b_1= \cdots= b_m = 0$, $M$ is regular along each of $D_j$ and hence $\ZCar'(M) = d \cdot [X]$ by definition.

Now, we assume that $(b_1, \dots, b_m) \neq (0, \dots, 0)$.  It suffices to prove that if $b_j \neq 0$, then $|\ZCar'_j(M)| \cap \big(T^*\!X^\log \times_X \{z\}\big)$ is exactly $Z_\vartheta$.  Without loss of generality, we assume $j=1$.  Recall that  $F_{(1)}$ is the completion of $\Frac ( R_{n,0})$ with respect to the $x_1$-valuation.  Since $M$ is clean at $z$,  there exists an integral scheme $\tilde D_1$ finite over $D_1$ such that
\begin{align*}
\Ref^\sharp(M \otimes F_{(j)};i) &= x_1^{-b_1} \cdots x_m^{-b_m} \big(\theta'_{i1}\frac{dx_1}{x_1} + \cdots + \theta'_{im}\frac{dx_m}{x_m}+ \theta'_{i,m+1}dx_{m+1} + \cdots + \theta'_{in}dx_n \big)
\\ 
&\in x_1^{-b_1} \cdots x_m^{-b_m} \big(\calO_{\tilde D_1} \frac{dx_1}{x_1} \oplus \cdots\oplus \calO_{\tilde D_1} \frac{dx_m}{x_m} \oplus \calO_{\tilde D_1} dx_{m+1} \oplus \cdots \oplus \calO_{\tilde D_1} dx_n \big)
\end{align*}
for any $i$.  Applying  Proposition~\ref{P:variation-vs-refined} to $\sfF_\id$ implies that $\theta'_{ij} \equiv \theta_j$ modulo $\sqrt{(x_2, \dots, x_n)\calO_{\tilde D_1}}$ for any $i$ and $j$.  By the definition of $\ZCar'_1(M)$, we see that $|\Car'_1(M)| \cap \big(T^*\!X^\log \times_X \{z\}\big)$ is precisely given by $Z_\vartheta$, finishing the proof.
\end{proof}

\begin{prop}
\label{P:local-calculation}
Keep the notation as above. If $b_1= \cdots= b_m = 0$, $\ZCar(M) = d \cdot [X]$.
If $(b_1, \dots, b_m)  \neq (0, \dots, 0)$, then $|\ZCar(M)|\cap \big(T^*\!X^\log \times_X \{z\}\big)$ is contained in the closed subset defined by $\theta_j\xi_i = \theta_i\xi_j$ for all $i \neq j$, and $x_1=\cdots = x_n$.
\end{prop}
\begin{proof}
This is the crux of the proof of the main theorem.  We will prove it in Subsection~\ref{S:local-calculation}.
\end{proof}

\subsubsection{Proof of Proposition~\ref{P:ZCar'<ZCar}}
\label{SS:proof-ZCar'<ZCar}
We remind the reader that we do not assume any cleanliness on $M$ for this proof.
First of all, since $M$ is coherent over $U$, the log-characteristic cycle of $M$ over $U$ is the same as the characteristic cycle over $U$, which is simply $d$ copies of the zero section of $T^*\!U$.  Hence $\ZCar(M) - d \cdot [X]$ is a non-negative combination of cycles of $T^*\!X^\log$ supported on $T^*\!X^\log \times_X D$.

Now, fix $D_j$ an irreducible component of $D$.  We need only to show that $\ZCar'(M) - \ZCar(j_*M)$ has no support above the generic point $\eta_j$ of $D_j$.  By Corollary~\ref{C:logcharcycles-base-change}, we may assume that we are in the CDVF local setup \ref{SS:local-setup}(c), in other words, we are in the setup of Definition \ref{D:ZCar'}.  Proposition~\ref{P:ZCar'<ZCar} then follows from Proposition~\ref{P:ZCar=ZCar'-CDVF} below.

\begin{prop}
\label{P:ZCar=ZCar'-CDVF}
Assume that we are in the local CDVF situation~\ref{SS:local-setup}(c).  We take $F = k(X)$, $\calO_X = \gotho_F$, and $\pi_F = x_1$.  Let $M$ be a $(\partial_1, \dots, \partial_n)$-differential module of rank $d$ over $F$.  Then $\ZCar(M)$ is equal to $\ZCar'(M)$ as cycles in $T^*\!X^\log = \Spec( \gotho_F[\xi_1, \dots, \xi_n])$, where $\xi_1$ denote the image of $x_1\partial_1$ and $\xi_j$ denote the image of $\partial_j$ for $j =2, \dots, n$.
\end{prop}
\begin{proof}
By Corollary~\ref{C:logcharcycles-base-change}(iii) and Lemma~\ref{L:ZCar'-stable}(iii), we can always replace $F$ by $F'(x_1^{1/h})$ for a positive integer $h$ and a finite extension $F'$ of $F$.  By Hukuhara-Levelt-Turrittin decomposition (see for example \cite[Theorem~2.3.3]{kedlaya-sabbah1}), we may assume that $M = E(\phi) \otimes \Reg$, where 
\begin{itemize}
\item $E(\phi)$ is the differential module of rank 1, generated by $\bbe$ such that $\partial_j(\bbe) = \partial_j(\phi) \bbe$ for some $\phi \in F$, and
\item $\Reg$ is a regular differential module over $F$.
\end{itemize}
Let $b =- v_F(\phi)$ and, let $\theta_1$ denote the reduction of $x_1^{b+1}\partial_1(\phi)$ in $\kappa_F$ and let $\theta_j$ denote the reduction of $x_1^b \partial_j(\phi)$ in $\kappa_F$ for $j =2, \dots, n$.  When $b>0$, $V$ has  pure irregularity $b$ and pure refined irregularity 
\[
\Ref(M) =d\phi = x_1^{-b}\theta_1 \frac{dx_1}{x_1} + x_1^{-b}\theta_2dx_2 + \cdots+ x_1^{-b}\theta_ndx_n  \in (x_1^{- b} \kappa_{F})\frac{dx_1}{x_1} \oplus \bigoplus_{j=2}^n (x_1^{- b} \kappa_{F}) dx_j.
\]  
According to Definition~\ref{D:ZCar'}, $\ZCar'(M) = d\cdot [X] + d\cdot \Irr(V) \cdot Z_{\vartheta}$, where $[X]$ is the zero section $\xi_1 = \cdots =\xi_n = 0$ and $Z_\vartheta$ is the cycle defined by $\theta_j \xi_{j'} = \xi_{j}\theta_{j'}$ for all $j,j'$.

We pick an  $\gotho_F$-lattice $\Reg_0$ of $\Reg$ that is stable under $x_1\partial_1, \partial_2, \dots, \partial_n$.  (The existence of such lattice is well-known, see \cite[Proposition~2.2.15]{kedlaya-sabbah1} for example.)

We use $M_0 = \bbe \otimes \Reg_0$ to define the log-characteristic cycle as in Definition~\ref{D:log-char-cycle}.  There are two cases we need to treat separately.

(i) If $b = 0$, we then have $\calD_X^\log \cdot M_0 = M_0$ and it is $\calO_X$-coherent.  In particular, we can provide it with the trivial filtration and hence $\ZCar(M) = d \cdot [X]$ in $T^*\!X^\log$.  This agrees with the definition of $\ZCar'(M)$.

(ii) If $b>0$, Remark~\ref{R:refined-CDVF} shows that $\theta_1 \in \kappa_F^\times$.   This implies that $\gotho_F \cdot x_1\partial_1(M_0) = x_1^{-b}M_0$ and $\partial_j(M_0) \subseteq x_1^{-b}M_0$ for any $j=2, \dots, n$.  Hence, $\calD_X^\log \cdot M_0 = M$.  We give $M$ a filtration by $\fil_\alpha M = 0$ if $\alpha <0$ and $ x_1^{-\alpha b} M_0$ if $\alpha \geq 0$.  We pick an $\gotho_F$-basis $\bbe_1, \dots, \bbe_d$ of $\Reg_0$.  Then the action of $\xi_j$ on the graded module $\gr_\bullet M$ is given by
\[
\xi_j(x_1^{-\alpha b}\bbe \otimes \bbe_i) =
\left\{
\begin{array}{ll}
 x_1^{b+1}\partial_1(\phi) x_1^{-(\alpha+1)b} \bbe \otimes \bbe_i &\textrm{if }j=1\\
x_1^b  \partial_j(\phi) x_1^{-(\alpha+1)b} \bbe \otimes \bbe_i &\textrm{if }j \in \{2, \dots, n\}
 \end{array}\right.
\]
for any $i \in \{1, \dots, d\}$.  (Note that the action from $\Reg_0$-dies when considering $\gr_\bullet M$.)  This implies that as an $\gotho_F[\xi_1, \dots, \xi_n]$-module, $\gr_\bullet M$ is isomorphic to
\begin{align*}
&\big(\gotho_F[\xi_1, \dots, \xi_n, t] / (x_1^b t, \xi_1 - x_1^{b+1}\partial_1(\phi) t, \xi_j - x_j^b \partial_j(\phi)t; j =2, \dots, n)\big)^{\oplus d}
\\
\cong & \big(\gotho_F[\xi_1, \dots, \xi_n] / (x_1^b\xi_1, x_1^{b+1}\partial_1(\phi)\xi_j - x_1^b\partial_j(\phi)\xi_1; j =2, \dots, n)\big)^{\oplus d},
\end{align*}
here the isomorphism follows from $x_1^{b+1}\partial_1(\phi) \equiv \theta_1 \neq 0$ modulo $x_1 \gotho_F$.
Since  $x_1^{b+1}\partial_1(\phi) \equiv \theta_1$ and $x_1^{b}\partial_j(\phi) \equiv \theta_j$ modulo $x_1 \gotho_F$, $\ZCar(M)$ is exactly the same as $\ZCar'(M)$.
\end{proof}

\subsection{Local calculation using good formal structures}
\label{S:calculation-good-model}

In this subsection, we prove Proposition~\ref{P:local-calculation} in the case when we have a good formal structure at $z$.  This calculation is basically due to Kato \cite[\S1]{kato-d-mod}.  We include it here because it is a toy version of the calculation in the next subsection.  In fact, we will prove the following stronger result.

\begin{prop}
\label{P:calculation-good-model}
Put $X= \Spec R_{n,0}$, $D = V(x_1\cdots x_m)$, and let $z$ be the origin.  Set $\phi = \alpha x_1^{-b_1} \cdots x_m^{-b_m}$ with $\alpha \in R_{n,0}^\times$ and $b_1, \dots, b_m \in \NN$.
Let $M = E(\phi) \otimes \mathrm{Reg}$, where $E(\phi)$ is the differential module defined in Definition~\ref{D:good-decomposition} and $\mathrm{Reg}$ is a \emph{regular} differential module of rank $d$ over $R_{n,m}$. Then we have an equality of cycles
\[
\ZCar'(M)= \ZCar(M).
\]
\end{prop}
\begin{proof}
As usual, we use $\xi_j$ denote the image of $x_j\partial_j$ if $j \leq m$ and of $\partial_j$ otherwise, in $\gr_\bullet \calD_X^\log$; then $\gr_\bullet \calD_X^\log = R_{n,0}[\xi_1, \dots, \xi_n]$.  Write
\[
d\phi = x_1^{-b_1} \cdots x_m^{-b_m} \big(\theta_1 \frac{dx_1}{x_1} + \cdots + \theta_m \frac{dx_m}{x_m} + \theta_{m+1} dx_{m+1} + \cdots + \theta_n dx_n \big),
\]
where $\theta_j = - b_j\alpha + x_j \partial_j(\alpha)$ if $j \leq m$ and $\theta_j = \partial_j(\alpha)$ otherwise.

We first compute $\ZCar'(M)$.  Since Reg is regular, for each $b_j >0$, $M \otimes F_{(j)}$ has pure refined irregularity
$d\phi$, viewed as an  element in $x_1^{-b_1} \cdots x_m^{-b_m}\Omega^1_X(\log D) \otimes \calO_{D_j}$.
Hence $\ZCar'_j(M)$ is the cycle defined by $x_j = 0$ and $\bar \theta_l^{(j)}\xi_i =\bar \theta_i^{(j)}\xi_l$ for all $i, l$, with multiplicity $b_j$, where $\bar \cdot^{(j)}$ means the reduction from $\calO_X$ to $\calO_{D_j}$.
Then $\ZCar'(M)$ is the union of the zero section with multiplicity $d$ and all $\ZCar'_j(M)$.

Then we compute the log-characteristic cycle $\ZCar(M)$.  Let $\mathrm{Reg}_0$ be a \emph{regulating lattice} \cite[Theorem~4.1.4 and Definition 4.1.8]{kedlaya-sabbah1} of Reg, i.e., $\mathrm{Reg}_0$ is a \emph{free} differential module over $R_{n,0}$, equipped with derivations $x_1\partial_1, \dots, x_m\partial_m, \partial_{m+1}, \dots, \partial_n$, and an isomorphism $\mathrm{Reg} \simeq \mathrm{Reg}_0 \otimes_{R_{n,0}}R_{n,m}$.  We choose a basis $\bbe_1, \dots, \bbe_d$ for $\Reg_0$ over $R_{n,0}$.  Let $\mathbf e$ denote the standard generator of $E(\phi)$ as in Definition~\ref{D:good-decomposition}. As in Definition~\ref{D:log-char-cycle}, we take $M_0 = \bbe \otimes \mathrm{Reg}_0$.  

When $b_1 = \cdots= b_m = 0$, $M_0$ is stable under the action of $\calD_X^\log$ and it is coherent as an $R_{n,0}$-module.  In this case, we can provide $M_0$ with the trivial filtration and $\ZCar(M)$ is simply the zero section of $T^*\!X^\log$ with multiplicity $d$; this agrees with $\ZCar'(M)$.

From now on, we assume that $b_j$ are not all zero.  Without loss of generality, we assume that $b_1 , \dots, b_l>0$ and $b_{l+1} = \cdots =b_m = 0$.  This implies that $\theta_1 = -b_1 \alpha + x_1 \partial_1(\alpha) \in R_{n,0}^\times$.
Since $ x_1\partial_1(\bbe) = x_1^{-b_1} \cdots x_l^{-b_l} \theta_1\bbe$, we conclude that $R_{n,0}\cdot x_1\partial_1( M_0) = x_1^{-b_1} \cdots x_l^{-b_l}  M_0$.  
(Note that contribution from the action of $x_1\partial_1$ on $\Reg_0$ is negligible compare to the contribution from the action of $x_1\partial_1$ on $\bbe$.)
Moreover, $x_j \partial_j(M_0) \subseteq x_1^{-b_1} \cdots x_l^{-b_l}  M_0$ for all $j$.
This implies that $\widetilde M_0= \calD_X^\log \cdot M_0 =R_{n,l} \otimes_{R_{n,0}} M_0$.  
We provide it with the following filtration: $\fil_\alpha \widetilde M_0 = 0$ if $\alpha<0$ and $\fil_\alpha \widetilde M_0= (x_1^{-b_1} \cdots x_l^{-b_l})^\alpha  M_0$ if $\alpha \geq 0$.  Then we have
\[
\gr_\alpha\widetilde M_0 = \left\{
\begin{array}{ll}
0 & \textrm{if }\alpha<0,\\
\bigoplus_{i=1}^dR_{n,0} \cdot \bbe \otimes \bbe_i & \textrm{if }\alpha =0,\\
\bigoplus_{i=1}^d R_{n,0}/(x_1^{b_1}\cdots x_l^{b_l}) \cdot x_1^{-\alpha b_1} \cdots x_l^{-\alpha b_l} \bbe \otimes \bbe_i & \textrm{if }\alpha>0.
\end{array}
\right.
\]
The action of $\xi_j$ on this graded module is given by
\[
\xi_j\big( x_1^{-\alpha b_1} \cdots x_l^{-\alpha b_l} \bbe \otimes \bbe_i\big) = \theta_j \cdot x_1^{-(\alpha+1)b_1} \cdots x_l^{-(\alpha+1 )b_l} \bbe \otimes \bbe_i
\]
for any $j \in \{1, \dots, n\}$ and any $i \in \{1, \dots, d\}$.  This immediately implies that, as an $R_{n,0}[\xi_1, \dots, \xi_n]$-module, $\gr_\bullet \widetilde M_0$ is isomorphic to
\[
\big(R_{n,0}[\xi_1, \dots, \xi_n] / (x_1^{b_1}\cdots x_l^{b_l}\xi_1, \theta_1\xi_j - \theta_j\xi_1; j =2, \dots, n)\big)^{\oplus d}.
\]
Hence $\ZCar(M)$ is exactly the same as $\ZCar'(M)$.
\end{proof}

\begin{remark}
If we start with $M$ having good formal structure at each closed point of $X$, we may simplify the proof of Theorem~\ref{T:main-theorem} by skipping the argument at the generic points of $D$ (Proposition~\ref{P:ZCar=ZCar'-CDVF}), because Proposition~\ref{P:calculation-good-model} have already matched the multiplicity at the generic points.
\end{remark}

\subsection{Local calculation in the clean case}
\label{S:local-calculation}

Now, we prove Proposition~\ref{P:local-calculation} under the cleanliness condition.  We assume that $k$ is algebraically closed in this subsection.  We start by recalling our setup.

\subsubsection{Setup}
We put $\bbt = x_1^{b_1} \cdots x_m^{b_m}$ to simplify notation.
Let $M$ be a finite differential module over $R_{n,m}$ of rank $d$.  Assume that $M \otimes F_{\underline r}$ has pure partially intrinsic radius $\mathrm e^{-b_1r_1-\cdots -b_mr_m}$ for all $\underline r \in [0, \infty)^m \times \{0\}^{n-m}$, and there exist $\underline \theta =(\theta_1, \dots, \theta_n) \in k^n \bs \{\underline 0\}$ and a local ring $\gothR$ finite over $k\llbracket x_{m+1}, \dots, x_n\rrbracket$ such that
\[
\bbt \cdot \Ref^\sharp(M \otimes F_{\underline r}; i) \equiv \theta_1 \frac{dx_1}{x_1} + \cdots +\theta_m \frac{dx_m}{x_m} + \theta_{m+1} dx_{m+1} + \cdots +\theta_n dx_n \mod \gothm_\gothR
\]
for any $i$ and any $\underline r \in (0, \infty)^m \times \{0\}^{n-m}$.

Put $\Delta_j = x_j\partial_j$ if $j\leq m$ and $\Delta_j = \partial_j$ if $j>m$.  Set $\omega_j = \frac{dx_j}{x_j}$ if $j \leq m$ and $\omega_j = dx_j$ if $j>m$.

\subsubsection{Lattice over $R_{n,0}$}
Denote $M_{(i)} = M \otimes F_{(i)}$ for $i=1, \dots, m$. By \cite[Lemma~1.4.14]{xiao-refined}, we may find a norm $|\cdot|_{M_{(i)}}$ on $M_{(i)}$ such that $|\Delta_j|_{M_{(i)}} \leq |x_i|_{F_{(i)}}^{-b_i}$ for $j =1, \dots, n$.  Define
\[
M_0 = \big\{x \in M\;\big|\; |x|_{M_{(i)}} \leq 1 \textrm{ for }i=1, \dots, m\big\}.
\]

\begin{lemma}
The $R_{n,0}$-module $M_0$ is finite over $R_{n,0}$, generically of rank $d$;  it generates $M$ over $R_{n,m}$.
\end{lemma}
\begin{proof}
We first prove finite generation over $R_{n,0}$.  Note that $M$ is projective over $R_{n,m}$.  There exists a finite (projective) module $M'$ over $R_{n,m}$ such that $\tilde M = M \oplus M'$ is finite and free over $R_{n,m}$; let $\bbe_1, \dots, \bbe_s$ be a basis.  Assign $M' \otimes F_{(i)}$ any $F_{(i)}$-norm $|\cdot|_{M'_{(i)}}$ which induces an $F_{(i)}$-norm $|\cdot|_{\tilde M_{(i)}}$ on $\tilde M_{(i)} = \tilde M \otimes F_{(i)}$.  It suffices to show that $\tilde M_0 := \big\{x \in \tilde M\;\big|\; |x|_{\tilde M_{(i)}} \leq 1 \textrm{ for }i=1, \dots, m\big\}$ is finite over $R_{n,0}$ because $M_0$ is a submodule of $\tilde M_0$.  Write $\tilde M'_0 = \bigoplus_{l=1}^s R_{n,0}\bbe_l$ by choosing some basis.  For any $i$, consider a different norm $|\cdot|'_{\tilde M_{(i)}}$ by taking $\bbe_1, \dots, \bbe_s$ to be an orthonormal basis; this is (topologically) equivalent to $|\cdot|_{\tilde M_{(i)}}$.  In particular, there exists $N_i \in \ZZ$ such that $|x|'_{\tilde M_{(i)}} \leq |x_i|^{-N_i}_{F_{(i)}}\cdot |x|_{\tilde M_{(i)}}$ for any $x \in \tilde M_{(i)}$.  This implies that $\tilde M_0 \subseteq x_1^{-N_1}\cdots x_m^{-N_m}\tilde M'_0$.  Hence $\tilde M_0$ is finite, so is $M_0$.

For the second half of the lemma, pick any finitely generated $R_{n,0}$-submodule $M'_0$ of $M$ such that it generates $M$.  Then $|M'_0|_{(i)}$ is bounded above for any $i$.  In particular, this implies that $x_1^{a_1} \cdots x_m^{a_m}M'_0 \subseteq M_0$ for some $a_1, \dots, a_m \in \NN$ and hence $M_0$ generates $M$ over $R_{n,m}$ and has generic rank $d$.
\end{proof}

\begin{remark}
We do not know how to prove that $M_0$ is a free $R_{n,0}$ and it is even clear to us whether we should expect this to be true.
This is exactly the problem we need to work around.
\end{remark}

\begin{lemma}
\label{L:action-on-M}
For any $\alpha \in \ZZ$ and any $j$, we have $\Delta_j(\bbt^{-\alpha}M_0) \subseteq \bbt^{-\alpha-1}M_0$.
\end{lemma}
\begin{proof}
This follows immediately from the fact that $|\Delta_j|_{M_{(i)}} \leq |x_i|_{F_{(i)}}^{-b_i}$ for any $\Delta_j$.
\end{proof}

\subsubsection{Filtration on $\widetilde M_0$}
Without loss of generality, we assume that $b_1, \dots, b_l>0$ and $b_{l+1} = \cdots =b_m = 0$.  By Proposition~\ref{L:xi_1-generates} below, we have $\widetilde M_0 = \calD_X^\log \cdot M_0 = R_{n, l} \otimes_{R_{n,0}} M_0$. We provide $\widetilde M_0$  with the following filtration: $\fil_\alpha \widetilde M_0 = 0$ if $\alpha<0$ and $\fil_\alpha \widetilde M_0= \bbt^{-\alpha}  M_0$ if $\alpha \geq 0$.  Then we have
\[
\gr_\alpha\widetilde M_0 = \left\{
\begin{array}{ll}
0 & \textrm{if }\alpha<0,\\
M_0 & \textrm{if }\alpha =0,\\
\bigoplus_{i=1}^d R_{n,0}/(\bbt) \otimes_{R_{n,0}} \bbt^{-\alpha} M_0 & \textrm{if }\alpha>0.
\end{array}
\right.
\]

\begin{prop}
\label{L:xi_1-generates}
If $\theta_j \neq 0$ for some $j$, then we have $R_{n,0}\cdot \Delta_j\big(\bbt^{-\alpha}M_0 \big) = \bbt^{-\alpha-1}M_0$, for any $\alpha \in \ZZ$.
As a consequence, $\fil_\bullet \widetilde M_0$ is a good filtration.  In particular, $\widetilde M_0= M_0 \otimes_{R_{n,0}} R_{n,l}$.
\end{prop}
\begin{proof}
The hypothesis of the lemma already implies that $b_1, \dots, b_m$ are not all zero.  Without loss of generality, we assume that $b_1>0$.
Let $\bbe $ be an element of $\bbt^{-\alpha}M_0$.  Let $M'_{(1)}$ denote the $\partial_j$-differential submodule of $M_{(1)}$ generated by $\bbe$.  We take the log-structure to be $\Log = \{\partial_j\}$ if $j\leq m$ and to be $\Log = \emptyset$ if $j>m$.  
Then $M'_{(1)}$ has pure partially intrinsic $\partial_j$-radius $|x_1|^{-b_1}$ and every refined partially intrinsic radius is of the form $\bbt^{-1}\theta' \omega_j$, where $\theta' \in \gothR_1$ for some local ring $\gothR_1$ finite over $k\llbracket x_2, \dots, x_n\rrbracket$ and $\theta' \equiv \theta_j \mod \gothm_{\gothR_1}$.  We write the  twisted polynomial associated to $\bbe$ with respect to the differential operator $\Delta_j$ (not $\partial_j$) as $X^s + a_1X^{s-1} + \cdots + a_s$, where $a_1, \dots, a_s \in F_{(1)}$.  By \cite[Remark~1.3.29]{xiao-refined}, we may apply \cite[Corollary~1.3.13]{xiao-refined} to the differential operator $\Delta_j$ (note that $|\Delta_j|_{F_{(1)}} \leq 1$); from this we know that
\[
\bbt^i a_i - (-\theta_j)^i \in x_1 \calO_{F_{(1)}}+ (x_2, \dots, x_n)k\llbracket x_2, \dots, x_n\rrbracket
\]
for any $i =1, \dots, s$.  (Note here the last term on the right is the intersection of the maximal ideal of $\gothR_1$ with the residue field of $\calO_{F_{(1)}}$.)
Now, by the definition of twisted polynomial, we have $\bbt^{s-1} \Delta_j^s\bbe + \bbt^{s-1} a_1\Delta_j^{s-1}\bbe + \cdots + \bbt^{s-1} a_s \bbe = 0$.  Hence,
\begin{align*}
-\bbt^{-1}(-\theta_j)^s\bbe &= \bbt^{s-1} \Delta_j^s\bbe + \bbt^{s-1} a_1\Delta_j^{s-1}\bbe + \cdots + \bbt^{s-1} a_s \bbe -\bbt^{-1}(-\theta_j)^s\bbe \\
&=\bbt^{s-1} \Delta_j^s\bbe + \bbt^{s-1} a_1\Delta_j^{s-1}\bbe + \cdots  + \bbt^{-1}(\bbt^{s} a_s -(-\theta_j)^s)\bbe
\end{align*}
We observe that the last term of the equation above belongs to 
\[
\big(x_1 \calO_{F_{(1)}}+ (x_2, \dots, x_n)k\llbracket x_2, \dots, x_n\rrbracket \big) \cdot \bbt^{-\alpha-1}M_0,
\]
and, by Lemma~\ref{L:action-on-M}, we have, if $j>m$
\[
\bbt^{s-1}a_i\Delta_j^{s-i}\bbe =
\bbt^ia_i\Delta_j^{s-i}(\bbt^{s-i-1}\bbe) \in \big(x_1\calO_{F_1} + k\llbracket x_2, \dots, x_n\rrbracket\big) \cdot \Delta_j(\bbt^{-\alpha}M_0),
\]
and if $j \leq m$
\begin{align*}
\bbt^{s-1}a_i\Delta_j^{s-i}\bbe &=
\bbt^ia_i\Delta_j\big(\bbt^{s-i-1}(\Delta_j^{s-i-1} + (s-i-1)\Delta_j^{s-i-2} + (s-i-1)(s-i-2)\Delta_j^{s-i-3} + \cdots )\bbe\big)\\
& \in \big(x_1\calO_{F_{(1)}} + k\llbracket x_2, \dots, x_n\rrbracket\big) \cdot \Delta_j(\bbt^{-\alpha}M_0).
\end{align*}
Therefore, we may write $-\bbt^{-1}(-\theta_j)^s\bbe$ as $\bbe_1 + \bbe_2 + \bbe_3$ with
\begin{align*}
\bbe_1 &\in 
x_1 \calO_{F_{(1)}}\cdot \bbt^{-\alpha-1}M_0,\\
\bbe_2 & \in (x_2, \dots, x_n)k\llbracket x_2, \dots, x_n\rrbracket \cdot \bbt^{-\alpha-1}M_0, \textrm{ and}\\
\bbe_3 &\in k\llbracket x_2, \dots, x_n\rrbracket \cdot \Delta_j(\bbt^{-\alpha}M_0).
\end{align*}
Viewing $\bbe_1 = -\bbt^{-1}(-\theta_j)^{-s}\bbe - \bbe_2 - \bbe_3$ forces $\bbe_1 \in M$ and $|\bbe_1|_{M_{(l)}} \leq |x_l|^{-(\alpha+1)b_l}$ for $l =2, \dots, n$.  Moreover, we know that $|\bbe_1|_{M_{(1)}} \leq |x_1|^{-(\alpha+1)b_1+1}$, yielding $\bbe_1 \in x_1\cdot \bbt^{-\alpha-1}M_0$.  Since $\bbe$ is arbitrary and $\theta_j \neq 0$, we conclude that
\[
\bbt^{-\alpha-1}M_0 \subseteq R_{n,0}\cdot \Delta_j(\bbt^{-\alpha}M_0) + (x_1, \dots, x_n)\bbt^{-\alpha-1}M_0.
\]
The proposition follows by Nakayama's lemma.
\end{proof}

\begin{corollary}
If $\theta_j \neq 0$ for some $j$, $\Delta_j$ induces an $R_{n,0}$-linear isomorphism $\bbt^{-\alpha}M_0 / \bbt^{-\alpha+1}M_0 \to \bbt^{-\alpha-1}M_0 / \bbt^{-\alpha}M_0$ for any $\alpha \in \ZZ$.  As a consequence, $\gr_\bullet \widetilde M_0$ is isomorphic to $M_0 \otimes_{R_{n,0}} R_{n,0}[\xi_j] / (\bbt \xi_j)$ as an $R_{n,0}[\xi_j]$-module.
\end{corollary}
\begin{proof}
The first statement follows from Proposition~\ref{L:xi_1-generates} and the fact that $\Delta_j(\bbt^{-\alpha} R_{n,0}) \subseteq \bbt^{-\alpha} R_{n,0}$ for any $\alpha \in \ZZ$.  The second statement follows immediately.
\end{proof}

\subsubsection{Proof of Proposition~\ref{P:local-calculation}}
Keep the notation as before.
First of all, if $b_1=\cdots = b_m =0$, $\widetilde M_0$ is coherent and the filtration is trivial.  Obviously, $\ZCar(M) = d\cdot [X]$.

Now, we assume that $\underline b \neq \underline 0$. Without loss of generality, we assume that $b_1>0$.  As above, we have a good filtration on $\widetilde M_0$.  Fix $j$ such that $\theta_j \neq 0$.  By an easy commutative algebra Lemma~\ref{L:commutative-algebra} below, we need only to show that, for any $i \in \{1, \dots, n\}$, the action of $\theta_i \xi_j - \theta_j \xi_i$ is nilpotent on $\gr_\bullet \widetilde M_0/ (x_1, \dots, x_n)$ as an $R_{n,0}[\xi_1, \dots, \xi_n]$-module.  Let $\Delta :=\theta_i \Delta_j - \theta_j \Delta_i$.  Recall that $\Delta_j: \gr_\alpha \widetilde M_0 \stackrel \sim \to \gr_{\alpha+1}\widetilde M_0$ is an isomorphism.  So, in explicit terms, we need only to show that, for any $\alpha \in \NN$ and any $\bbe \in \fil_\alpha \widetilde M_0$, $\Delta^d(\bbe) \subseteq (x_1, \dots, x_n) \fil_{\alpha+d} M_0$.

We only consider the differential operators $\partial_i$ and $\partial_j$.  Take the log-structure to be $\Log = \{\partial_1, \dots, \partial_m\} \cap \{\partial_i, \partial_j\}$.  As in previous lemma, we consider $M \otimes F_{(1)}$; this differential module has pure partially intrinsic radius $|x_1|^{-b_1}$ and all of its refined partially intrinsic radii are of the form $\bbt^{-1}\vartheta$, where $\vartheta \in \gothR_1$ for some local ring $\gothR_1$ finite over $k\llbracket x_2, \dots, x_n\rrbracket$ and $\vartheta \equiv \theta_i \omega_i + \theta_j \omega_j$ modulo $\gothm_{\gothR_1}$.

We claim that if we view $M \otimes F_{(1)}$ as a $\Delta$-differential module over $F_{(1)}$, then any Jordan-H\"older factor of $M \otimes F_{(1)}$ either has $\Delta$-radii $> |x_1|^{-b_1}$, or has $\Delta$-radii $|x_1|^{-b_1}$ and its refined $\Delta$-radii lies in $\bbt^{-1}\gothm_{\gothR_1}$ for some local ring $(\gothR_1, \gothm_{R_1})$ finite over $k\llbracket x_2, \dots, x_n\rrbracket$.
Indeed, we may first apply Proposition~\ref{P:decomposition-field}(iii) to reduce to the case when $M \otimes F_{(1)}$ has pure refined partially intrinsic radius.  Then we apply \cite[Theorem~1.4.20]{xiao-refined} to conclude.  Strictly speaking, $\Delta$ is not a differential operator of rational type, but one uses \cite[Remark~1.4.22]{xiao-refined} and the fact that $|\Delta|_{F_{(1)}} = 1$.

Following the proof of Proposition~\ref{L:xi_1-generates}, we pick an arbitrary element $\bbe \in \fil_\alpha \widetilde M_0 = \bbt^{-\alpha}M_0$.  It generates a $\Delta$-differential submodule of $M \otimes F_{(1)}$.  Let $X^s+a_1X^{s-1} + \cdots + a_s$ denote the twisted polynomial associated to $\bbe$ with respect to the differential operator $\Delta$.  By the claim above and \cite[Corollary~1.3.13]{xiao-refined} (using the version described in \cite[Remark~1.3.29]{xiao-refined}), we know that
\[
\bbt^ia_i  \in x_1 \calO_{F_{(1)}} + (x_2, \dots, x_n)k\llbracket x_2, \dots, x_n\rrbracket
\]
for any $i=1, \dots, s$.  This implies that
\[
\Delta^s(\bbe) = -a_1 \Delta^{s-1}(\bbe) - \cdots -a_s \bbe \in \big(x_1 \calO_{F_{(1)}} + (x_2, \dots, x_n)k\llbracket x_2, \dots, x_n\rrbracket \big) \bbt^{-\alpha-s}M_0.
\]
Therefore, we can write $\Delta^s(\bbe)$ as $\bbe_1 + \bbe_2$ with
\[
\bbe_1 \in 
x_1 \calO_{F_{(1)}}\cdot \bbt^{-\alpha-s}M_0, \textrm{ and }
\bbe_2 \in (x_2, \dots, x_n)k\llbracket x_2, \dots, x_n\rrbracket \cdot \bbt^{-\alpha - s}M_0.
\]
The equality $\bbe_1 = \Delta^s(\bbe)-\bbe_2$ forces $\bbe_1 \in M$ and $|\bbe_1|_{M_{(l)}} \leq |x_l|^{-(\alpha+s)b_l}$ for $l =2, \dots, n$, yielding $\bbe_1 \in x_1\cdot \bbt^{-\alpha-s}M_0$.  Hence, we have  $\Delta^s(\bbe) \in (x_1, \dots, x_n) \bbt^{-\alpha-s}M_0$, which trivially implies that 
\[
\Delta^d(\bbe) \in \Delta^{d-s}\big((x_1, \dots, x_n) \bbt^{-\alpha-s}M_0\big) \subseteq (x_1, \dots, x_n) \bbt^{-\alpha-d}M_0.
\]
This concludes the proof of Proposition~\ref{P:local-calculation}.

\begin{lemma}
\label{L:commutative-algebra}
Let $R$ be a noetherian ring and let $N$ be a finite 
$R$-module. Let $\gothp$ be a prime ideal of $R$.  Let $r \in R$ be an element such that $r^dN \subseteq \gothp N$ for some positive integer $d$.  Then $
\Supp(N) \cap \overline{\{\gothp\}}$ is contained in the closed subset $Z(r)$ defined by $r$.
\end{lemma}
\begin{proof}
Pick $\gothm \in \Supp(N) \cap \overline{\{\gothp\}}$.  If $r \notin \gothm R_\gothm$, then $r \in R_\gothm^\times$; the condition would imply that $N_\gothm = r^d N_\gothm \subseteq \gothp N_\gothm \subseteq \gothm N_\gothm$.  By Nakayama's lemma, $N_\gothm = 0$, which is a contradiction.  Hence, $r \in \gothm R_\gothm$ and the lemma follows.
\end{proof}

\end{document}